\documentclass[11pt]{amsart}
\usepackage{fullpage,amsmath, amssymb,amsfonts,amsmath,latexsym,amscd,amsthm, mathtools, stackrel,hyperref}
\usepackage{tikz-cd}
\usepackage[top=2cm, bottom=4.5cm, left=3cm, right=3cm]{geometry}
\usepackage{float, graphicx}
\usepackage{bm}
\graphicspath{ {./images/} }

\usepackage{tikz-cd}
\usepackage{tikz}
\usepackage{graphicx,subfigure}

\usepackage{enumerate}
\usepackage[shortlabels]{enumitem}

\newtheorem{theorem}{Theorem}[section]
\newtheorem{corollary}[theorem]{Corollary}
\newtheorem{proposition}[theorem]{Proposition}
\newtheorem{lemma}[theorem]{Lemma}
\newtheorem{claim}[theorem]{Claim}
\newtheorem{fact}[theorem]{Fact}

\newtheorem{construction}[theorem]{Construction}

\newtheorem{conjecture}[theorem]{Conjecture}

\theoremstyle{definition}

\newtheorem{remark}[theorem]{Remark}
\newtheorem{definition}[theorem]{Definition}

\makeatletter

\newcommand{\Bd}{\text{Bd}}
\newcommand{\Ch}{\text{Ch}}
\newcommand{\imp}{\text{Imp}}
\newcommand{\Ss}{\mathbb S}
\newcommand{\phat}{\hat{\varphi}}
\newcommand{\RS}{\hat{\C}}
\newcommand{\U}{\mathcal{U}}
\newcommand{\trdeg}{\text{trdeg}}
\newcommand{\ord}{\text{ord}}
\newcommand{\zar}{\text{Zar}}

\newcommand{\bpsi}{\mathbf{\Psi}}

\newcommand{\D}{\mathbb D}

\newcommand{\Z}{\mathbb Z}
\newcommand{\N}{{\mathbb N}}
\newcommand{\C}{\mathbb C}

\newcommand{\bP}{\mathbb P}

\newcommand{\bG}{\mathbb G}

\newcommand{\lra}{\longrightarrow}
\newcommand{\dra}{\dashrightarrow}
\newcommand{\CC}{\mathcal C}

\makeatletter

\title{Bialgebraic geometry of B\"ottcher coordinates}
\author{Sina Saleh}
\date{}

\begin{document}
\begin{abstract}
Becker and Bergweiler showed that if $f$ is a non-exceptional polynomial, then the B\"ottcher coordinate $\Psi_f \colon \D_R \to B_\infty(f)$ associated to $f$ is a transcendental function. In this paper, we study $f$-bialgebraic sets: algebraic subsets of $\D_R^n$ whose image under the coordinate-wise action of $\Psi_f$ is contained in an algebraic set of the same dimension.
 We give a complete dynamical classification of bialgebraic sets under the additional assumption that the Julia set of $f$ is either disconnected, or connected and admits a nondegenerate locally connected model. Inspired by the Ax--Lindemann--Weierstrass theorem and the Ax--Schanuel conjecture, we formulate analogs with $\Psi_f$ in place of the exponential function and prove them in the case where the Julia set $J_f$ is disconnected.
\end{abstract}
\maketitle

\section{Introduction}
Let $f$ be a polynomial of degree $d \ge 1$. Denote by $B_\infty(f)$ the basin of attraction of $\infty$, and by $J_f$ the Julia set of $f$. We assume that $\infty \in B_\infty(f)$ as well, and for each \(R>0\), let \(\mathbb{D}_R\) denote the open disk of radius \(R\). A classical theorem of B\"ottcher asserts that there exists $R \le 1$ and a conformal isomorphism
\[
\Psi_f \colon \mathbb{D}_R \longrightarrow U_\infty(f) \subset B_\infty(f)
\]
that conjugates $f|_{U_\infty(f)}$ to $z \mapsto z^d$ on $\mathbb{D}_R$. A theorem of Becker and Bergweiler \cite[Theorem 1]{Bottcher-transcendence} shows that if $f$ is \emph{non-exceptional}, that is if $f$ is not affine conjugate to a power map, a Chebyshev polynomial, or a negative Chebyshev polynomial, then $\Psi_f$ is transcendental over $\mathbb{C}(z)$. In view of this transcendence, one expects $\Psi_f$ to fail to preserve algebraic structure. For example, it is shown in \cite{Bottcher-transcendence} that if $f$ is non-exceptional and $\alpha \in \mathbb{D}_R$ is an algebraic number, then $\Psi_f(\alpha)$ is necessarily transcendental. One may also interpret this as a Lindemann-type theorem \cite{lindemann}, with $\Psi_f$ in place of the exponential map $e^z$. 

From a more algebro-geometric point of view, it is natural to consider the images of complex algebraic sets—namely, sets defined by finitely many complex polynomials—under the uniformizing map
\[
\Psi_{f,n} \colon \mathbb{D}_R^n \longrightarrow U_\infty(f)^n,
\qquad
(x_1,\dots,x_n) \longmapsto (\Psi_f(x_1),\dots,\Psi_f(x_n)).
\]
This parallels the study of images of algebraic sets under the uniformization
\[
\pi_n \colon \mathbb{C}^n \longrightarrow (\mathbb{C}^\ast)^n,
\qquad
(z_1,\dots,z_n) \longmapsto (\exp(z_1),\dots,\exp(z_n)),
\]
a setting that is well understood through the Ax--Schanuel theorem \cite{Ax-paper} and its consequences, such as the Ax--Lindemann--Weierstrass theorem and the classification of bialgebraic sets; see \cite[Theorems~1.2.6 and~1.2.7]{Bakker-Tsimerman-notes}. 

Given that $\Psi_f$ is transcendental when $f$ is non-exceptional, one expects the image of complex algebraic sets under $\bpsi_{f,n}$ to be non-algebraic. However, as the next list of examples shows, there are exceptions:
\begin{enumerate}
    \item The diagonal $C_1 := \{y = x\} \subset \D_R^2$. The image of $C_1$ under $(\Psi_f, \Psi_f)$ is also contained in the diagonal. 

    \item The curves $C_2:= \{x = a\}, C_3:=\{y = b\}$ where $a,b \in \D_R$. The images of $C_2$ and $C_3$ are contained in $\{x = \Psi_f(a)\}$ and $\{y = \Psi_f(b)\}$, respectively. 
\end{enumerate}
These examples show that this situation is subtler than \cite[Theorem 2]{Bottcher-transcendence} and motivate the following definition. We say that an algebraic branch $V \subset \mathbb{D}_R^n$ of dimension $r \ge 0$ is \textit{$f$-bialgebraic}, or simply \textit{bialgebraic}, if $\bpsi_{f,n}(V) \subset B_\infty(f)^n$ is an algebraic branch of a subvariety of $\C^n$ of dimension $r$. Here, by an \emph{algebraic branch of dimension \(r\)}, we mean an irreducible component, of dimension \(r\), of the analytic germ \((V',p)\), where \(V'\subset \mathbb{C}^n\) is an algebraic subvariety and \(p\in V'\). The examples above suggest a natural method for constructing such sets. To describe this construction, we introduce the notion of an \textit{$f$-special} subset of $\C^n$.

We say that an algebraic set $W \subset \C^n$ is \textit{$f$-special} if there exists a subset of coordinates $\mathcal{I}=\{i_1,\dots,i_s\} \subseteq \{1,\dots,n\},$
for some $s \ge 0$, such that $W$ is contained in a fiber of the projection to the complementary coordinates, and the projection $\pi_\mathcal{I}(W)$
to the coordinates indexed by $\mathcal{I}$ is preperiodic under the map $F_s \colon \C^s \to \C^s,$
\[
\qquad
F_s(x_1,\dots,x_s) = \bigl(f(x_1),\dots,f(x_s)\bigr).
\]
That is, there exist integers $0 \le N < M$ such that
\[
F_s^{N}(\pi_\mathcal{I}(W)) = F_s^{M}(\pi_\mathcal{I}(W)).
\]

This definition produces a large supply of bialgebraic sets. Indeed, if $W \subset \C^n$ is an $f$-special subvariety, and $W_1$ is a branch of $W$ contained in $W \cap U_\infty(f)^n$, then
\[
(\Psi_f^{-1},\dots,\Psi_f^{-1})(W_1) \subset \mathbb{D}_R^n
\]
is $f$-bialgebraic; see Proposition~\ref{prop:f-special-gives-bialg}. 

Our first main result shows that if $J_f$ is either connected and admits a non-degenerate locally connected model (see Section~\ref{sec:Bottcher}) or disconnected, then every bialgebraic branch arises in this way.  

\begin{theorem}
\label{thm:classification-of-bialgebraic}
Suppose that $f$ is a non-exceptional polynomial of degree $d \ge 2$. Assume that one of the following holds: 
\begin{enumerate}
    \item $J_f$ is connected and has a non-degenerate locally connected model, or

    \item $J_f$ is disconnected.
\end{enumerate}
Let $V_1 \subset \C^n$ be an irreducible subvariety with non-empty intersection with $\D_R^n$. Then, a branch $V_1'$ of $V_1$ contained in $\D_R^n$ is $f$-bialgebraic if and only if there exists an irreducible $f$-special subvariety $V_2$ such that $\dim_\C(V_1) = V_2$ and $\bpsi_{f,n}(V_1') \subset V_2$. In fact, we have 
\[
\overline{\bpsi_{f, n}(V_1')}^{\zar} = V_2.
\]
\end{theorem}

\begin{remark}
\label{rem:one-dir-clear}
The ``if'' direction of Theorem~\ref{thm:classification-of-bialgebraic} follows directly from the definition of $f$-bialgebraic sets. Therefore, our proof will focus on the ``only if'' direction. 
\end{remark}
\begin{remark}
\label{conj:non-lc}
Although our methods rely on the two topological assumptions above, we expect Theorem~\ref{thm:classification-of-bialgebraic} to hold in full generality without any extra assumptions on $J_f$.
\end{remark}

\begin{remark}
The hypothesis that $J_f$ admit a non-degenerate locally connected model does indeed exclude certain polynomials. For example, \cite[Theorem~2.2]{Blokh-trivial-models} shows that the model is degenerate whenever \(f\) is a quadratic polynomial with a Cremer fixed point. More generally, \cite[Theorem~5.2]{blokh-uniCremer} proves that every basic uniCremer polynomial, of arbitrary degree, has a degenerate locally connected model.
\end{remark}
\begin{remark}
The hypothesis that $J_f$ admit a non-degenerate locally connected model is less restrictive than the hypothesis that $J_f$ be locally connected. In particular, \cite{milnor-non-lc-construction} gives an example of an infinitely renormalizable quadratic polynomial with non-locally connected Julia set. By construction, this polynomial has no Cremer or Siegel periodic orbit, so it lies within the scope of Kiwi's theory \cite{kiwi-real-lamination}. It follows that, although $J_f$ is not locally connected, its locally connected model is still non-degenerate.
\end{remark}
\begin{remark}
\label{rem:motivation}
Theorem \ref{thm:classification-of-bialgebraic} and Remark \ref{conj:non-lc} are partially motivated by \cite[Lemmas 5.9 and 6.11]{schmidt} which together prove Theorem \ref{thm:classification-of-bialgebraic} without any assumptions on $J_f$ when $n = 2$ and $V_1 = \{y = \theta x\}$ for some $\theta \in \Ss^1$. The original proof of \cite[Lemma 5.9]{schmidt} contained a gap, which has since been resolved following extensive correspondence between the author and Harry Schmidt.
\end{remark}

An interesting consequence of Theorem \ref{thm:classification-of-bialgebraic} is the next rigidity result.

\begin{corollary}
\label{cor:shared-basin}
Let $f$ be as in Theorem \ref{thm:classification-of-bialgebraic} and assume that the Julia set is connected and has a non-degenerate locally connected model. Suppose that $r(z) \in \C(z)$ is a rational function with $B_\infty(f)$ as a periodic Fatou component. Then, 
\[
r^{b}(z) = L \circ h^a(z),
\]
for some $a \ge 0$, some $b \ge 1$, and some polynomial $h$ satisfying $h^\ell = f$ for some $\ell \ge 1$, and a linear function $L$ commuting with a compositional power of $h$.
\end{corollary}
% \begin{remark}
% Corollary \ref{cor:shared-basin} is closely related to \cite[Conjecture 4.1]{Hinkkanen-Martin}. Let us explain why. Suppose we are in the setting of Corollary \ref{cor:shared-basin}, and let $G=\langle r,f\rangle$
% be the semigroup generated by $r$ and $f$. We claim that the Julia set of $G$ is equal to $J(r)$ (see \cite[Section 2]{Hinkkanen-Martin}).

% Indeed, since $J(f)\subset J(r)$, both $r$ and $f$ are normal on $\widehat{\mathbb{C}}\setminus J(r)$. It follows that the family $G$ is normal on $\widehat{\mathbb{C}}\setminus J(r)$, and hence $ 
% N(G)\supset \widehat{\mathbb{C}}\setminus J(r)$. Equivalently, $
% J(G)\subset J(r)$. 
% On the other hand, by \cite[Lemma 2.1]{Hinkkanen-Martin}, the Julia set of each generator is contained in $J(G)$, so in particular $ J(r)\subset J(G).$
% Therefore $J(G)=J(r).$

% Thus, we are exactly in the setting of \cite[Conjecture 4.1]{Hinkkanen-Martin}. Moreover, in the terminology of \cite{Hinkkanen-Martin}, the conclusion of Corollary \ref{cor:shared-basin} is precisely that the semigroup $\langle r,f\rangle$ is \emph{nearly abelian}. Hence, Conjecture 4.1 of \cite{Hinkkanen-Martin} holds in the special case covered by Corollary \ref{cor:shared-basin}.
% \end{remark}
\begin{remark}
In fact, we can take $b \in \{1, 2\}$. This is because if an iterate of a rational function is a polynomial, then the second iterate must already be a polynomial; see \cite[Section 4.1]{beardon}.     
\end{remark}

\begin{remark}
\label{rem:Julia-rigidity}
If we are in the setting of Corollary \ref{cor:shared-basin}, then $J(f)\subset J(r)$, since $B_\infty(f)$ is a periodic Fatou component of $r$. Thus, this situation naturally touches on rigidity questions concerning rational maps with identical Julia sets or with Julia sets having substantial overlap.

For instance, the case of equal Julia sets was studied by Levin and Przytycki in \cite{levin-przytcki}, under additional hypotheses involving the measure of maximal entropy. Another related direction is the theory of buried Julia components: Wang and Yang show in \cite{Wang-Yang-buried} that one may realize a buried Julia component quasiconformally conjugate to a given Julia set, although their setting concerns such embedded dynamical copies rather than literal set-theoretic inclusion $J(f)\subsetneq J(g)$
for two given rational maps. The theory of local symmetries of Julia sets, which we discuss briefly in Section \ref{sec:alg-symms}, is also closely related to this situation. Indeed, our proof of Theorem \ref{thm:classification-of-bialgebraic}, and hence ultimately of Corollary \ref{cor:shared-basin}, relies heavily on this point of view. Nevertheless, Corollary \ref{cor:shared-basin} does not appear to be a direct consequence of any of the existing results on overlapping Julia sets.
\end{remark} 
In view of Theorem~\ref{thm:classification-of-bialgebraic} and by analogy with the Ax--Lindemann--Weierstrass theorem, it is natural to expect that the Zariski closure in $\C^n$ of the image of an algebraic branch under $\bpsi_{f,n}$ is $f$-special. We therefore pose the following conjecture.
\begin{conjecture}
\label{conj:DALW}
Suppose that $f$ is a non-exceptional polynomial of degree $d \ge 2$. Let $V_1 \subset \C^n$ be an irreducible subvariety with non-empty intersection with $\D_R^n$, and let $V_1'$ be a branch of $V_1$ contained in $\D_R^n$. Then, the Zariski closure $\overline{\bpsi_{f,n}(V_1')}^{\text{Zar}}$ is $f$-special.
\end{conjecture}
\begin{remark}
\label{rem:DALW-gives-classification}
Conjecture \ref{conj:DALW} implies Theorem \ref{thm:classification-of-bialgebraic}. 
\end{remark}
Ultimately, by analogy with the Ax--Schanuel theorem, we arrive at the following conjecture.
\begin{conjecture}
\label{conj:dyn-ax-schanuel}
Suppose that $f$ is a non-exceptional polynomial of degree $d \ge 2$. Let $Z \subset \D_R^n \times U_\infty(f)^n$ be an analytic subset of dimension $0 \le m \le n$ defined by the parametrization
\[
\{(h_0(x), \dots,h_{n-1}(x), \Psi_f(h_0(x)), \dots,\Psi_f(h_{n-1}(x))) : x \in D \subset \C^m\},
\]
where $D$ is a non-empty open subset of $\C^m$ for some $m \ge 1$ and $h_0,\dots,h_{n-1}$ are holomorphic functions on $D$. Suppose that 
\[
(\Psi_f(h_0(x)), \dots,\Psi_f(h_{n-1}(x)))
\]
is not contained in a proper $f$-special subvariety of $\C^n$. Then, 
\[
\dim_\C\left(\overline{Z}^{Zar}\right) \ge m + n.
\]
\end{conjecture}
\begin{remark}
Conjecture \ref{conj:dyn-ax-schanuel} holds in the case \(m=n\). Indeed, in this setting,
\[
Z=\{(x_1,\dots,x_n,\Psi_f(x_1),\dots,\Psi_f(x_n)) : (x_1,\dots,x_n)\in D\},
\]
where $D$ is an open subset of $\mathbb{D}_R^n$.

To prove Conjecture \ref{conj:dyn-ax-schanuel} we need to show that $Z$ is Zariski dense in $\C^{2n}$. Assume, for contradiction, that there exists a non-zero polynomial $P \in \mathbb{C}[X_1,\dots,X_n,Y_1,\dots,Y_n]$ satisfying
\[
P(x_1,\dots,x_n,\Psi_f(x_1),\dots,\Psi_f(x_n))=0.
\]
 The above identity then forces an algebraic relation between \(x_i\) and \(\Psi_f(x_i)\) for some $1 \le i \le n$, which contradicts the transcendence of \(\Psi_f\).
\end{remark}
As in the classical setting of the Ax--Schanuel and Ax--Lindemann--Weierstrass theorems, one can show that Conjecture~\ref{conj:dyn-ax-schanuel} implies Conjecture~\ref{conj:DALW}, yielding the following result. 
\begin{theorem}
\label{thm:DAS-then-DALW}
Conjecture \ref{conj:dyn-ax-schanuel} implies Conjecture \ref{conj:DALW}.
\end{theorem}
Our next theorem proves Conjecture \ref{conj:dyn-ax-schanuel} under the assumption that $J_f$ is disconnected and also assuming that the projection of $Z$ to the first $n$ coordinates defines an algebraic subvariety of $\C^n$.
\begin{theorem}
\label{thm:DAS-totally-disconnected}
Assume we are in the setting of Conjecture \ref{conj:dyn-ax-schanuel} and that $J_f$ is disconnected. Moreover, assume that 
\[
(h_0(x), \dots,h_{n-1}(x))
\]
is a branch of an irreducible algebraic subvariety of $\C^n$. Suppose that 
\[
(\Psi_f(h_0(x)), \dots,\Psi_f(h_{n-1}(x)))
\]
is not contained in a proper $f$-special subvariety of $\C^n$. Then,
\[
\dim_\C\left(\overline{Z}^{Zar}\right) = m + n.
\]
\end{theorem}
As a corollary of Theorem \ref{thm:DAS-totally-disconnected} and the proof of Theorem \ref{thm:DAS-then-DALW} we get  the following result.
\begin{corollary}
\label{cor:DALW-disc}
If $J_f$ is disconnected, then Conjecture \ref{conj:DALW} holds. 
\end{corollary}

\subsection{Related dynamical transcendence questions}
\label{transcendence-of-canonical-ht}
Although we have focused on functional transcendence properties of B\"ottcher coordinates, algebraic independence of values of B\"ottcher coordinates has also been given some attention recently; see, for instance, \cite{xie-transcendence}. Moreover, while our results concern the B\"ottcher coordinate associated to a fixed polynomial, one may also ask analogous algebraic independence questions involving B\"ottcher coordinates associated to distinct polynomials. Some results in this direction are known; see, for example, \cite{nguyen-alg-independence} and the forthcoming work \cite{saleh-yap}. Finally, we note that our results are similar in spirit to questions concerning the transcendence of polynomial canonical height values; see \cite{nguyen-transcendence-hts}.

\subsection{Other Ax-Schanuel-type results}
Over the last decade, Ax--Schanuel type results have been established for a broad range of more general uniformization maps. Pila and Tsimerman \cite{Pila-Tsimerman} proved an Ax--Schanuel theorem for the modular $j$-function, showing that unexpected algebraic relations are explained by modular relations. Mok, Pila, and Tsimerman \cite{Mok-Pila-Tsimerman} then proved the corresponding theorem for uniformizing maps of pure Shimura varieties, while Bakker and Tsimerman \cite{Bakker-Tsimerman}, and later Gao and Klingler \cite{Gao-Klinger}, extended these ideas to period maps arising from variations of Hodge structure. For further recent developments, we refer the reader to \cite{Gao,Chiu-Tak,Freitag-Nagloo} and the references therein. We also mention that \cite{gamma-bialg} recently proved a classification of the bialgebraic varieties associated with Euler's $\Gamma$-function.

\subsection{Functional transcendence of Mahler functions}
\label{rem:transcendence-of-
mahler-functions}
The B\"ottcher function $\Psi_f(z)$ defined above is a Mahler function, i.e., a function satisfying a functional equation involving the transformation $z \mapsto z^d$ (see also \cite[Theorem 1.3]{Nishioka-book}). The study of Mahler functions and the transcendence and algebraic independence of their values dates back to Mahler \cite{Mahler-1,Mahler-2,Mahler-3}. Algebraic independence of Mahler functions was later studied by Kubota, Loxton, Van der Pooten, and Nishioka (see \cite{Nishioka-book} and the references therein). However, we note that Conjectures \ref{conj:DALW} and \ref{conj:dyn-ax-schanuel} require us to study the algebraic independence of functions
\[
z,a_1(z),\dots,a_{n-1}(z),\Psi_f(z),\Psi_f(a_1(z)),\dots,\Psi_f(a_{n-1}(z)),
\]
where $a_1,\dots,a_{n-1}$ are analytic functions with values in $\D_R^\ast$. Even though the function $\Psi_f(z)$ itself is a Mahler function, the functions $\Psi_f(a_i(z))$ typically are not Mahler functions. Therefore, the algebraic independence results of Mahler functions do not directly apply in our situation. 

\vspace{1em}

The proof of Theorem \ref{thm:classification-of-bialgebraic} in the case of a disconnected Julia set is deduced from Corollary \ref{cor:DALW-disc}, which in turn follows from Theorem \ref{thm:DAS-totally-disconnected}. Accordingly, the proofs are organized in a slightly non-linear order: we first establish Theorem \ref{thm:DAS-totally-disconnected} and Corollary \ref{cor:DALW-disc}, and only then prove Theorem \ref{thm:classification-of-bialgebraic}. We next outline the proof strategies for Theorems \ref{thm:DAS-totally-disconnected} and \ref{thm:classification-of-bialgebraic}.
\subsection{Proof Strategy for Theorem \ref{thm:DAS-totally-disconnected}.} If $J_f$ is disconnected, then $\Psi_f$ is only defined on $\D_R$ for some $R < 1$. But, we can still use the functional equation 
\begin{equation}
\label{eqn:psi-z-dn}
\Psi_f(z^{d^j}) = f^j(\Psi_f(z)),    
\end{equation}
to analytically continue $\Psi_f$ along paths in $\D_{R^{1/d^j}}$ for any $j \ge 1$. If $c$  is a critical point of $f$ on the boundary of $\Psi_f(\D_R)$, we can define 
\[
E_{j,c} = \{z \in \D: \Psi_f(z^{d^j}) = f(c)\},
\]
for any $j \ge 1$. It is shown  in \cite[Lemma 6.1]{schmidt} that $\Psi_f$ exhibits non-trivial monodromy along certain loops. More precisely, we can take the loops to be of the form
\[
\gamma_j := p_j \cdot \ell_j \cdot p_j^{-1}
\]
where $\ell_j$ is a small simple loop around a point $z_j \in E_{j,c}$ and $p_j$ is a carefully chosen simple path (see Definition \ref{def:precritical-cont-path}) starting in $\D_R$ and ending at a point of $\ell_j$. 

The nontrivial monodromy arising from the continuation of \(\Psi_f\) via equation~\eqref{eqn:psi-z-dn} is because of the fact that, by definition, \(\Psi_f(z_j^{d^j})\) has critical preimages under \(f^j\). Thus, by choosing an appropriate branch of \(f^{-j}\), we can ensure that the continuation of \(\Psi_f(z)\) along \(\ell_j\), which amounts to continuing \(f^{-j}\) along a small loop around \(f(c)\), actually changes  \(\Psi_f(z)\). In fact, \cite[Lemma 6.1]{schmidt} shows that continuing $\Psi_f$ along $\gamma_j$ using \eqref{eqn:psi-z-dn} sends $\Psi_f(z)$ to $\Psi_f(\zeta_jz)$ for some root of unity $\zeta_j$ of order at least $\alpha 2^j$ for some constant $\alpha > 0$; see also Lemma \ref{lem:loop-monodromy-around-E_n} and Proposition \ref{prop:path-for-non-triv-monodromy}. 

We now sketch the idea of the proof of Theorem~\ref{thm:DAS-totally-disconnected} in the case of curves, that is, when \(m=1\); see Proposition~\ref{prop:DAS-curve}. Suppose we have an algebraic relation
\[
P(x, \Psi_f(x), \Psi_f(h_1(x)), \dots, \Psi_f(h_{n-1}(x))) = 0,
\]
for all $x$ in a non-empty open subset $U$ of $\C$ where $h_1,\dots,h_{n-1}$ are holomorphic on $U$ defining a local chart $(x,h_1(x),\dots,h_{n-1}(x))$ of an algebraic curve in $\D_R^{n}$. Using the rich monodromy discussed above, and choosing the loops $\gamma_j$ carefully (see Definition \ref{def:increasing-psi-monodromy} and Lemma \ref{lem:permute+rotate-for-monodromy}), we can turn
\[
P(x, \Psi_f(x), \Psi_f(h_1(x)), \dots, \Psi_f(h_{n-1}(x))) = 0
\]
to a continuous family of relations
\[
P(x, \Psi_f(g_0x), \Psi_f(g_1h_1(x)), \dots, \Psi_f(g_{n-1}h_{n-1}(x))) = 0
\]
for all $(g_0, g_1, \dots, g_{n-1}) \in \mathcal{G}$, where $\mathcal{G}$ is a positive-dimensional analytic set. A specialization argument, together with an application of Laurent's theorem \cite[Th\'eor\`eme 1]{Laurent} (see Lemma \ref{lem:many-preimages-then-special}) allows us to turn the above relation into a relation of the form
\[
Q_1(x, \Psi_f(h_{s+1}(x)), \dots, \Psi_f(h_{n-1}(x))) = 0
\]
or a relation
\[
Q_2(\Psi_f(y), \Psi_f(\tilde h_1(y)), \dots, \Psi_f(\tilde h_s(y))) = 0,
\]
for some $0 \le s \le n-1$ where each $(y, \tilde{h}_i(y))$ defines a branch of a torsion coset of an algebraic subgroup of $\bG_m^2$; see Claim \ref{claim:h-sends-to-E-comp} and equations \eqref{eqn:u-rel} and \eqref{eqn:w-rel}. In the first case, we can conclude the proof inductively. In the second case, we use the special form of $\tilde{h}_i$ to prove the Zariski closure of the branch defined by 
\[
\Psi_f(y), \Psi_f(\tilde h_1(y)), \dots, \Psi_f(\tilde h_s(y))
\]
must be $F_{s+1}$-preperiodic. 

Finally, to prove the theorem in all dimensions, we proceed by induction on the dimension. The key ingredient is a slicing lemma (see Lemma~\ref{lem:dim-red-via-fibration}), which uses a suitable fibration to cut the given branch into a family of branches of lower dimension. We refer the reader to the end of Section~\ref{sec:pf-of-DAS-disconnected} for further details.
\subsection{Proof Strategy for Theorem \ref{thm:classification-of-bialgebraic}.} When $J_f$ is disconnected, Theorem \ref{thm:classification-of-bialgebraic} follows from Corollary \ref{cor:DALW-disc} (see also Remark \ref{rem:DALW-gives-classification}). So, we focus on the case where $J_f$ is connected and has a non-degenerate locally connected model. In this case, we know that $R = 1$ and $\Psi_f(\D) = B_\infty(f)$. 

The high-level flow of the proof is as follows: 
\begin{align}
\begin{tikzcd}
\text{Theorem \ref{thm:classification-of-bialgebraic} for $n = 2$, $V_1$ irreducible curve} \arrow[Rightarrow, d] \\
\text{Theorem \ref{thm:classification-of-bialgebraic} for $V_1$ a hypersurface} \arrow[Rightarrow, d] \\
\text{Theorem \ref{thm:classification-of-bialgebraic}}
\end{tikzcd}\notag
\end{align}
Thus, the first step of the proof is the case where $n=2$ and $V_1$ is an irreducible curve defined by a polynomial $P \in \C[X, Y]$; see Proposition \ref{prop:bialg-curve-case}. Suppose a branch $V'_1 \subset \D^2$ of $V_1$ is $f$-bialgebraic. Then, there exists some irreducible algebraic curve $V_2$ given by some polynomial $Q \in \C[X, Y]$ such that $\bpsi_{f,2}(V'_1) \subset V_2$. Suppose $(x, h(x))$ is a local parametrization of $V'_1$ for some holomorphic function $h: U \lra \D$ and some $U \subset \D$. We then also have
\begin{equation}
\label{eqn:Q-eqn}
Q(\Psi_f(x) ,\Psi_f(h(x))) = 0,
\end{equation}
for all $x \in U$. 

Since $P$ is an algebraic curve, we can use analytic continuation to extend $h$ along paths in $\C$. Using these analytic continuations we can choose a ``suitable'' $w_0 \in \partial\D$ (see Claims \ref{claim:w_0-is-nice} and \ref{claim:image-also-endpoint}) and a continuation of $h$ to some neighborhood $U_0$ of $w_0$ such that one of the following occurs:
\begin{enumerate}
    \item $h(U_0) \subset \D$, or

    \item $h(U_0 \cap \D) \subset \D$ and $h(U_0 \cap \overline{\D}^c) \subset \overline{\D}^c$; 
\end{enumerate}
see also the discussion preceding Claim \ref{claim:image-also-endpoint}. In the first case, we use equation \eqref{eqn:Q-eqn} to extend $\Psi_f$ beyond $\D$ and prove that an open subset of the Julia set must be smooth (see Claim \ref{lem:extend-psi}). This contradicts Fatou's theorem \cite[pp. 250]{Fatou} and finishes the proof. 

The second case is much more delicate. Here we use \eqref{eqn:Q-eqn} to extend the function
\[
H := \Psi_f \circ h \circ \Psi_f^{-1},
\]
initially defined on $\Psi_f(U_0 \cap \D)$, to an open neighborhood $N$ of the impression $\imp(w_0)$ in such a way that
\begin{equation}
\label{eqn:H-local-sym}
H^{-1}(H(N) \cap J_f) = N \cap J_f.
\end{equation}
In other words, $H$ is a local symmetry of $J_f$ in the sense of \cite{Levin-symm}. Establishing this identity is the main difficulty. Let us explain the source of the problem.

It is easy to see, by a simple continuity argument, that every point of
\[
\partial \Psi_f(U_0 \cap \D) \cap J_f
\]
is mapped by $H$ into $J_f$. However, the set $N \cap J_f$ may also contain points that do not lie in
\[
\partial \Psi_f(U_0 \cap \D) \cap J_f.
\]
This already occurs, for instance, when $J_f$ is locally connected and $\imp(w_0)$ is a biaccessible or polyaccessible point of the Julia set. For such points, the continuity argument gives no information, and a priori there is no reason for their images under $H$ to lie in $J_f$.

It is precisely at this stage that we use the assumption that $J_f$ admits a non-degenerate locally connected model. Depending on the type of this model, we can ensure the impression $\imp(w_0)$ corresponds to either a Jordan point, an interval point, or an endpoint in the sense of subsection~\ref{subsec:nbhd-bases}; see Claim \ref{claim:w_0-is-nice}. In each case, we construct a basis of simply connected open neighborhoods $\{N_i\}_{i \ge 1}$ of $\imp(w_0)$ (see Construction \ref{cons:base-at-fibers}) such that
\[
N_i \cap J_f \subset \partial \Psi_f(U_0 \cap \D) \cap J_f
\]
By the observation above, this implies that
\[
H(N_i \cap J_f) \subset J_f,
\]
and hence
\[
H^{-1}(H(N_i) \cap J_f) \supset N_i \cap J_f.
\]

To prove the reverse inclusion, we show that $\imp(h(w_0))$ is of the same type as $\imp(w_0)$; see Claim~\ref{claim:image-also-endpoint}. We may then repeat the above argument with $h^{-1}$ in place of $h$. The key point is that $h$ respects the lamination relation induced by the locally connected model and satisfies
\[
\imp(h(w_0)) = H(\imp(w_0));
\]
see Lemma~\ref{lem:h-respects-lamination}. The desired conclusion about $\imp(h(w_0))$ then follows from the fact that the model of $J_f$ is either a Jordan curve, an interval, or else has the property that uncountably many points are endpoints (see Lemma~\ref{lem:model-types}).

After showing $H$ gives a local symmetry of the Julia set, we can conclude the theorem using the fact that algebraic local symmetries must come from preperiodic curves under $(f,f)$ (see Proposition \ref{prop:alg-symmetries}). 

Once the curve case has been established, we prove the hypersurface case by induction; see Proposition~\ref{prop:bialg-hypersurface-case}. The main idea is to study the slices $V_1^a$ obtained by intersecting $V'_1$ with the fibers of the projection $\pi_1 \colon \C^n \lra \C$ 
onto the first coordinate. We show that, for all but finitely many $a \in \D_R$, the projection
\[
\pi_{2,\dots,n}\bigl(V_1^a\bigr)
\]
is  $f$-bialgebraic, and hence \(f\)-special by the inductive hypothesis. We can then use our definition of $f$-special sets, together with Medvedev and Scanlon's classification \cite{Scanlon}, to translate this information about the slices into the conclusion that \(\Psi_{f,n}(V_1')\) itself must also be contained in an \(f\)-special set.

Finally, let \(V_1 \subset \C^n\) be a subvariety of dimension \(k\). To prove Theorem~\ref{thm:classification-of-bialgebraic}, we consider the projections of \(V_1\) onto suitable collections of \(k+1\) coordinates and apply the hypersurface case. This shows that the image of each such projection under $\Psi_{f,k+1}$ is either contained in an \(F_{k+1}\)-preperiodic hypersurface or contained in a fiber. It is then straightforward to combine these projection constraints and conclude that $\Psi_{f,n}(V_1')$ is contained in an $f$-special subvariety of $\C^n$. 

\subsection{Notations and conventions.} We write $\D$, $\D_R$, and $\mathbb{S}^1$ for the unit disk, the open disk of radius $R$, and the unit circle, respectively.

Given a polynomial $f \in \C[x]$ of degree $d \ge 1$, we let $B_\infty(f)$ denote the basin of infinity, and throughout the paper we assume that $\infty \in B_\infty(f)$. We let $R \le 1$ denote the maximal radius of convergence of the B\"ottcher coordinate $\Psi_f$ (see Section~\ref{sec:Bottcher}), and we set
\[
U_\infty(f) := \Psi_f(\D_R).
\]
We also write
\[
\Phi_f \colon U_\infty(f) \to \D_R
\]
for the inverse of $\Psi_f$. Throughout the paper, $f$ will be a polynomial of degree $d \ge 2$ and $\Psi_f$ and $\Phi_f$ will be its associated B\"ottcher and inverse B\"ottcher coordinates. When no confusion can arise, we omit the subscript $f$ and write
$\Psi$ and $\Phi$ instead of $\Psi_f$ and $\Phi_f$.

We let $\bP^1$ denote the complex projective line, and we fix a choice of point at infinity so as to identify $\bP^1$ with the Riemann sphere $\RS$. In this way, we obtain the natural inclusion
\[
\C = \bP^1 \setminus \{\infty\} \subset \bP^1.
\]
Given a subset $J = \{j_1,\dots,j_r\} \subset \{1,\dots,n\}$, we let
\[
\pi_J = \pi_{j_1,\dots,j_r} \colon (\bP^1)^n \to (\bP^1)^{|J|}
\]
denote the projection
\[
(x_1,\dots,x_n) \mapsto (x_{j_1},\dots,x_{j_r}).
\]

We let $I = [0,1]$. A path in $\C^n$ is a continuous map $\gamma \colon I' \to \C^n$ where $I' \subset I$ is a subinterval, not necessarily closed. Throughout the paper, all paths are assumed to be smooth and locally injective. By abuse of notation, we also use $\gamma$ to denote the image of the path in $\C^n$.  Given a path $\gamma \colon I \to X$, we let $\gamma^{-1}$ denote the reversed path defined by
\[
\gamma^{-1}(s) = \gamma(1-s).
\]
If $I'' \subset I'$ is a subinterval, not necessarily closed, we write $\gamma(I'')$ for the corresponding subpath of $\gamma$. Given two paths $a$ and $b$ such that the endpoint of $a$ coincides with the starting point of $b$, we write $a \cdot b$ for the path obtained by first traversing $a$ and then $b$.

For any $k \ge 1$, we let $\mu_k$ denote the set of $k$-th roots of unity and we set
\[
\mu_\infty := \bigcup_{k \ge 1} \mu_k.
\]
Similarly, we let
\[
\mu_{d^\infty} := \bigcup_{k \ge 1}\mu_{d^k}.
\]

Suppose $X \subset \C^n$ is an analytic subvariety of $\C^n$. By an analytic branch of $X$ at a point $p \in X$, we mean an irreducible component of the germ $(X, p)$, or equivalently, an irreducible component of $B_\epsilon(p) \cap X$ for a sufficiently small $\epsilon > 0$.  
\subsection{Outline of the paper} In Section~\ref{sec:Bottcher}, we recall some basic facts about B\"ottcher coordinates and laminations. We also review the main result of \cite{block-lc-models} on locally connected models of polynomial Julia sets. The remainder of the section is devoted to constructing neighborhood bases at Jordan, interval, and endpoint impressions of $J_f$ (see Constructions~\ref{cons:base-at-model} and~\ref{cons:base-at-fibers}) and to proving several useful topological properties of these bases (see Propositions~\ref{prop:D_i-is-base} through~\ref{prop:D_i-cap-B_infty-2-conn-comps} and Corollary~\ref{cor:imps-are-equal}). We conclude the section by proving that a certain class of analytic continuations of holomorphic functions beyond the basin of infinity must send impressions to impressions and must be locally compatible with the lamination; see Lemma~\ref{lem:h-respects-lamination}. This lemma is crucial for the proof of Theorem~\ref{thm:classification-of-bialgebraic}.

Section~\ref{sec:alg-symms} is a short section in which we collect a classification of algebraic local symmetries of Julia sets of non-exceptional polynomials.

In Section~\ref{sec:reductions}, we prove Theorem~\ref{thm:DAS-then-DALW}, along with several other useful reductions needed later in the proof of Theorem~\ref{thm:DAS-totally-disconnected}.

In Section~\ref{sec:f-special-sets}, we prove that $f$-special sets give rise to $f$-bialgebraic sets.

In Section~\ref{sec:monodromy}, we study the nontrivial monodromy of $\Psi_f$ in the case where $J_f$ is disconnected. Our goal is to formulate some of the results of \cite[Section~6]{schmidt} in a more detailed form suited to our purposes.

Finally, in Sections~\ref{sec:pf-of-DAS-disconnected}, \ref{sec:pf-of-bialg-classification}, and~\ref{sec:pf-of-shared-comp}, we prove Theorems~\ref{thm:DAS-totally-disconnected} and~\ref{thm:classification-of-bialgebraic}, as well as Corollary~\ref{cor:shared-basin}, respectively.

\vspace{1em}
\textbf{Acknowledgements.} I would like to express my deepest gratitude to my PhD advisor, Laura DeMarco, for her continuous support and constructive feedback throughout this project. Many helpful conversations with her made the results of this paper possible. Theorem \ref{thm:classification-of-bialgebraic} and Conjecture \ref{conj:DALW} were motivated by a joint work in progress with Jit Wu Yap, and I would like to thank him for many fruitful conversations on this topic. I would also like to emphasize that he was the first to observe that Conjecture \ref{conj:DALW} is the right type of statement to hope for in the context of B\"ottcher coordinates. I am very grateful to Harry Schmidt for many helpful conversations. I am also grateful to Saeed Zakeri for many useful discussions, which helped me refine my initially simplistic understanding of non-locally connected Julia sets and pointed me in the right direction for learning more about the non-locally connected case. I also thank Yusheng Luo for very helpful conversations regarding local symmetries of polynomial Julia sets and on his feedback on Proposition \ref{prop:alg-symmetries}. Special thanks go to Max Weinreich for his constructive feedback which improved the presentaition of this article. Finally, I thank Adam Melrod, who suggested during a talk that ``special sets'' be renamed ``$f$-special sets'' for more clarity.

\section{B\"ottcher coordinates, laminations, and locally connected Julia set models}
\label{sec:Bottcher}
In this section, we recall some basic facts about B\"ottcher coordinates. We then review the main results of \cite{block-lc-models} and explain how they allow us to associate to the polynomial $f$ a $d$-invariant lamination, in the sense of Thurston \cite{Thurston-laminations}. The remainder of the section is devoted to using the locally connected model to construct suitable simply connected local neighborhood bases on $J_f$ and to prove useful topological properties of these bases. We conclude with a key lemma showing that a special kind of analytic continuation beyond $B_\infty(f)$ locally sends impressions to impressions and also respects the lamination.

Let $f$ be a polynomial of degree $d \ge 1$ and let $\infty \in B_\infty(f)$ be the basin of attraction of infinity.  B\"ottcher's theorem \cite[Theorem 9.1]{milnor} asserts that there is an open subset $\infty \in U_\infty(f) \subset B_\infty(f)$, some $0 < R \le 1$, and a biholomorphic function $\Psi_f: \D_R \lra U_\infty(f)$ with inverse $\Phi_f$ such that $\Psi_f(0) = \infty$ and 
\begin{equation}
\label{eqn:bottcher-eqns}
\Psi_f(z^d) = f(\Psi_f(z)) \ \ \text{ and } \ \ \Phi_f(f(w)) = \Phi_f(w)^d,
\end{equation}
for all $z \in \D_R$ and all $w \in U_\infty(f)$. 
\begin{fact}[\text{\cite[Theorem 9.1]{milnor}}]
$\Psi_f(z)$ is unique up to replacement by $\Psi_f(\zeta z)$ for some $(d-1)$-th root of unity $\zeta$.   
\end{fact}
\begin{fact}[\text{\cite[Theorem 9.3]{milnor}}]
There exists a unique $R \le 1$ such that $\Psi_f$ extends holomorphically to $\D_R$. If $R = 1$, $\Psi_f$ sends $\D$ biholomorphically to $B_\infty(f)$. In this case, $J_f$ must be connected and $\infty$ is the only critical point of $f$ in $B_\infty(f)$. If $R < 1$, $J_f$ must be disconnected and we must have $\overline{U_\infty(f)} \subsetneq B_\infty(f)$. Moreover, the boundary of $U_\infty(f)$ must contain  a critical point of $f$.
\end{fact}
\begin{fact}[\text{\cite[Thoerem 18.3]{milnor}}]
\label{fact:loc-conn}
When $J_f$ is locally connected, $\Psi_f$ extends continuously to the boundary $\partial\D$ and we must have $\Psi_f(\partial\D) = J_f$. 
\end{fact}
Now assume that $J_f$ is connected. For a point $\theta = e^{2\pi i t} \in \Ss^1$, we say that $z \in J_f$ lies in the \textit{impression of $\theta$}, denoted $\imp(\theta)$, if there exists a sequence $\{w_n\} \subset \D$ converging to $\theta$ such that
\[
\lim_{n\to\infty}\Psi_f(w_n)=z.
\]
If $J_f$ is locally connected, then every impression is a singleton by Fact~\ref{fact:loc-conn}. In that case, one obtains an equivalence relation on $\Ss^1$ by 
\[
x \sim_f y \quad \text{if and only if} \quad \Psi_f(x)=\Psi_f(y).
\]
This relation, first introduced by Thurston \cite{Thurston-laminations}, is called a \textit{$d$-invariant lamination}. The quotient $\Ss^1/{\sim_f}$ is homeomorphic to $J_f$, and hence provides a topological model for the Julia set. Kiwi \cite{kiwi-real-lamination} later extended this construction to polynomials without irrationally neutral cycles. Blokh, Curry, and Oversteegen \cite{block-lc-models} subsequently proposed a different approach, based on continuum theory, to construct locally connected models for arbitrary polynomial Julia sets. In what follows, we first review the definition of abstract circle laminations and then turn to the main result of \cite{block-lc-models}.

Given $t,t' \in \Ss^1$, let $\overline{tt'}$ denote the Poincar\'e geodesic in $\D$ connecting $t$ and $t'$. Also, for any set $G \subset \Ss^1$, let $\Ch(G)$ be the convex hull of the set $G$ in $\D$ i.e. the smallest convex set in $\overline{\D}$ containing $G$. Let $\sigma_d: \Ss^1 \lra \Ss^1$ be the multiplication by $d$ map given by 
\[
e^{2\pi i t} \mapsto e^{2\pi i dt}.
\]
Suppose $\sim$ is an equivalence relation on $\Ss^1$ and for any $x \in \Ss^1$ let $[x]_\sim$ denote the $\sim$-class of $x$. Following \cite{blokh-levin} and \cite{block-lc-models}, we say that the equivalence relation $\sim$ is a \textit{$d$-invariant lamination} if it satisfies the following properties:
\begin{itemize}
    \item[(E1)] The graph of $\sim$ in $\Ss^1 \times \Ss^1$ is a closed set; 

    \item[(E2)] If $\mathrm{g}_1$ and $\mathrm{g}_2$ are distinct $\sim$-classes, then the convex hulls $\Ch({\mathrm{g}_1})$ and $\Ch(\mathrm{g}_2)$ are disjoint; 

    \item[(D1)] $\sim$ is forward invariant i.e. the image $\sigma_d({\mathrm{g}_1})$ of a class is equal to another class ${\mathrm{g}_2}$; 

    \item[(D2)] $\sim$ is backward invariant i.e. the preimage of a class under $\sigma_d$ is a union of classes;

    \item[(D3)] For any class $\mathrm{g}$ with at least three elements, the restriction $\sigma_d|_{{\mathrm{g}}}: {\mathrm{g}} \lra \sigma_d({\mathrm{g}})$ is a covering map with positive orientation. 
\end{itemize}
A lamination is called \textit{non-degenerate} whenever it has at least two distinct classes. A class with three elements or more is called a \textit{gap}. The edges of the convex hulls of the $\sim$-classes are called the \emph{leaves} of the associated lamination.

By \cite[Theorem 2]{block-lc-models}, there exists a monotone continuous map $\phat: \RS \lra \RS$ that semiconjugates $f$ onto a topological polynomial $g: \RS \lra \RS$ such that 
\begin{enumerate}
    \item $\phat|_{U}$ is a homeomorphism when $U$ is equal to $B_\infty(f)$, or any other Fatou component whose boundary is not contracted to a single point by $\phat$; and
    \item $J_\sim := \phat(J_f)$ is locally connected; and 
    \item  every fiber of $\phat$ is equal to a union of impressions of $J_f$.
\end{enumerate}
 The image $J_\sim = \phat(J_f)$ is called the \textit{topological model} of $J_f$. We say that the model is \textit{non-degenerate} whenever $J_\sim$ is not a singleton. We also call $\phat(B_\infty(f))$ the \textit{basin of infinity of $g$} and denote it $B_\infty(g)$.

Let $\iota_f: \Ss^1 \lra J_\sim$ be the map given by 
\begin{equation}
\label{eqn:def-of-iota}
\iota_f := \phat \circ \imp
\end{equation}
which is well-defined since the impressions of $J_f$ are contracted by $\phat$. \cite[Theorem 32]{block-lc-models} shows that $\iota_f$ gives a semiconjugacy 
\begin{equation}
\iota_f(z^d) = g(\iota_f(z)). 
\end{equation}
The map $\iota$ induces a natural equivalence relation on $\Ss^1$ identifying the fibers of $\iota_f$ i.e. 
\[
x \sim_f y \iff \iota_f(x) = \iota_f(y). 
\]
Indeed, this equivalence relation defines a $d$-invariant circle lamination.

Let $K \subset \C$ be a full continuum whose boundary $\partial K$ is locally connected. A point $x \in \partial K$ may admit, up to homotopy, one, two, or more than two accesses. In these cases, we call $x$ \emph{uniaccessible}, \emph{biaccessible}, or \emph{polyaccessible}, respectively. See \cite[Definition 2.7]{kiwi-real-lamination} for the definition of accessibility. It is worth mentioning that aside from countably many points of $\partial K$, all points are either uniaccessible or biaccessible; see \cite[Proposition 2.18]{Pommerenke}.

Now consider the model $J_\sim$ which is the boundary of $B_\infty(g) = \C \setminus \phat(K_f)$. Whenever the model $J_\sim$ is not a Jordan curve, it necessarily contains a distinguished type of point called an \emph{endpoint}. Informally, endpoints are uni-accessible points of $J_\sim$ that arise as limits of poly-accessible points . The following definition makes this precise. 
\begin{definition}[$J_\sim$-endpoints]
\label{def:endpoints}
A point \(t\in \Ss^{1}\) is called a \emph{\(J_\sim\)-endpoint} if $[t]_{\sim_f} = \{t\}$ and there exist sequences
\(\{t_i\}_{i\ge 1}, \{t_i'\}_{i\ge 1}\subset \Ss^{1}\) such that
\begin{itemize}
    \item $t_i \ne t'_i$ and \(\iota_f(t_i)=\iota_f(t_i')\) for every \(i\ge 1\);
    \item \(t_i\to t\) and \(t'_i\to t\) as \(i\to\infty\); and
    \item for every \(i\ge 1\), the point \(t\) lies on the shorter closed arc of \(\Ss^{1}\) with endpoints \(t_i\) and \(t_i'\).
\end{itemize}
\end{definition}
\begin{remark}
\label{t_i,t'_i-are-leaves}
We may moreover assume that $\overline{t_i t'_i}$ is a non-degenerate leaf of the lamination. Indeed, if $\overline{t_i t'_i}$ is not a leaf, we replace the pair $(t_i,t'_i)$ by endpoints $a,b$ of a non-degenerate leaf $\overline{ab}$ in the convex hull of the class $[t_i]_{\sim_f}=[t'_i]_{\sim_f}$, chosen so that the shorter arc of $\Ss^1$ joining $a$ to $b$ contains $t$.
\end{remark}

To simplify notation, in the rest of the paper, we write $\iota := \iota_f$ and $[\cdot] := [\cdot]_{\sim_f}$. The next trichotomy is an analog of \cite[Lemma 3.2]{Meerkamp} for topological polynomials. 
\begin{lemma}
\label{lem:model-types}
One of the following three conditions must hold for the model $J_\sim$: 
\begin{itemize}
    \item $J_\sim$ is a Jordan curve;
    \item $J_\sim$ is an interval;
    \item $J_\sim$ has uncountably many endpoints. In fact, given any leaf $\overline{ab}$ of the lamination on $\Ss^1$, there exist uncountably many $J_\sim$ endpoints on the two arcs connecting $a$ to $b$. Moreover, any element $t \in \Ss^1$ whose equivalence class $[t]$ is the singleton $\{t\}$ must be a $J_\sim$-endpoint.  
\end{itemize}
\end{lemma}
\begin{proof}
Suppose that none of the first two cases hold. As a result of this, we get a class of $\sim$ with at least three elements. This produces a gap in the lamination. To prove there are uncountably many endpoints using the presence of gaps, we follow an approach similar to the proof of \cite[Lemma 2]{zdunik}.

Let $\mathbf{g}$ be a gap of $\Ss^1$ and set $G:=\Ch(\mathbf{g})$. By the pigeonhole principle, there exist two distinct arcs
\[
I_1:=(a_1b_1)\quad\text{and}\quad I_2:=(a_2b_2),
\]
with positive arc length at most $\frac{2\pi}{3}$ such that $\overline{a_1b_1}, \overline{a_2b_2} \in \Bd(G)$ are leafs of the lamination. Since preimages of any point $g \in \mathbf{g}$ under the iterates of $\sigma_d$ are dense in $\Ss^1$, it is straightforward to find gaps $\mathbf{g}_1\subset I_1$ and $\mathbf{g}_2\subset I_2$ such that
\[
\sigma_d^{n_1}(\mathbf{g}_1)=\mathbf{g}
\qquad\text{and}\qquad
\sigma_d^{m_1}(\mathbf{g}_2)=\mathbf{g}
\]
for some integers $n_1,m_1\ge 1$. Let $G_i:=\Ch(\mathbf{g}_i)$ for $i=1,2$. Repeating the same argument for each $G_i$, we obtain, for every $i\in\{1,2\}$, two distinct arcs
\[
I_{i1}:=(a_{i1}b_{i1})
\quad\text{and}\quad
I_{i2}:=(a_{i2}b_{i2}),
\]
with positive arc length arc length at most $\frac{2\pi}{3^2}$ such that $\overline{a_{i1}b_{i1}}, \overline{a_{i2}b_{i2}}\in \Bd(G_i)$ are leafs of the lamination. 

Iterating this construction, we obtain for each finite word $s_1\cdots s_r\in\{1,2\}^r$ an arc $I_{s_1\cdots s_r}$ of length at most $\frac{2\pi}{3^r}$ such that
\[
I_{s_1}\supset I_{s_1s_2}\supset \cdots \supset I_{s_1\cdots s_r},
\]
and such that for any words $s_1\cdots s_r$ and $s'_1\cdots s'_{r'}$ the arcs $I_{s_1\cdots s_r}$ and $I_{s'_1\cdots s'_{r'}}$ are disjoint unless the shorter word agrees with the initial segment of the longer one.

Consequently, to every infinite sequence $\mathbf{s}=(s_i)_{i\ge 1}\in\{1,2\}^{\N}$ we can associate a unique point
\[
e_{\mathbf{s}}:=\bigcap_{i\ge 1} I_{s_1\cdots s_i}.
\]
By construction, $e_{\mathbf{s}}$ is a $J_\sim$-endpoint and each $\mathbf{s}$ determines a unique $e_{\textbf{s}}$. Hence, there are uncountably many endpoints. The proof of the second statement is analogous.

For the third statement, take any leaf $\overline{a_0b_0}$ belonging to a gap $\mathbf{c}$, such that $t$ is on an arc $I$ connecting $a_0$ and $b_0$. Take $\mathbf{g}$ to be the same as before. Arguing as above, we can find some $\mathbf{g}'$ such that 
\[
\sigma_d^n(\mathbf{g}') = \mathbf{g},
\]
for some $n \ge 2$ and the class $\mathbf{g'}$ is contained in $I$. This produces a smaller leaf $\overline{a_1b_1}$ such that $t$ is on the shorter arc $I$ connecting $a_0$ and $b_0$. Continuing in this way, we can construct leaves $\overline{a_ib_i}$ such that $a_i,b_i \to t$ as $i \to \infty$. The last statement follows.
\end{proof}
In light of Lemma \ref{lem:model-types}, we also make the next definition. 
\begin{definition}[interval points and Jordan points]
\label{def:interval-points}
In the case where $J_{\sim}$ is an interval, an \emph{interval point} is a point $t \in \Ss^1$ that is not an endpoint. In the case where $J_\sim$ is a Jordan curve, any point $t \in \Ss^1$ will be called a \textit{Jordan point}.
\end{definition}

\begin{figure}
    \centering
    \includegraphics[width=1\linewidth]{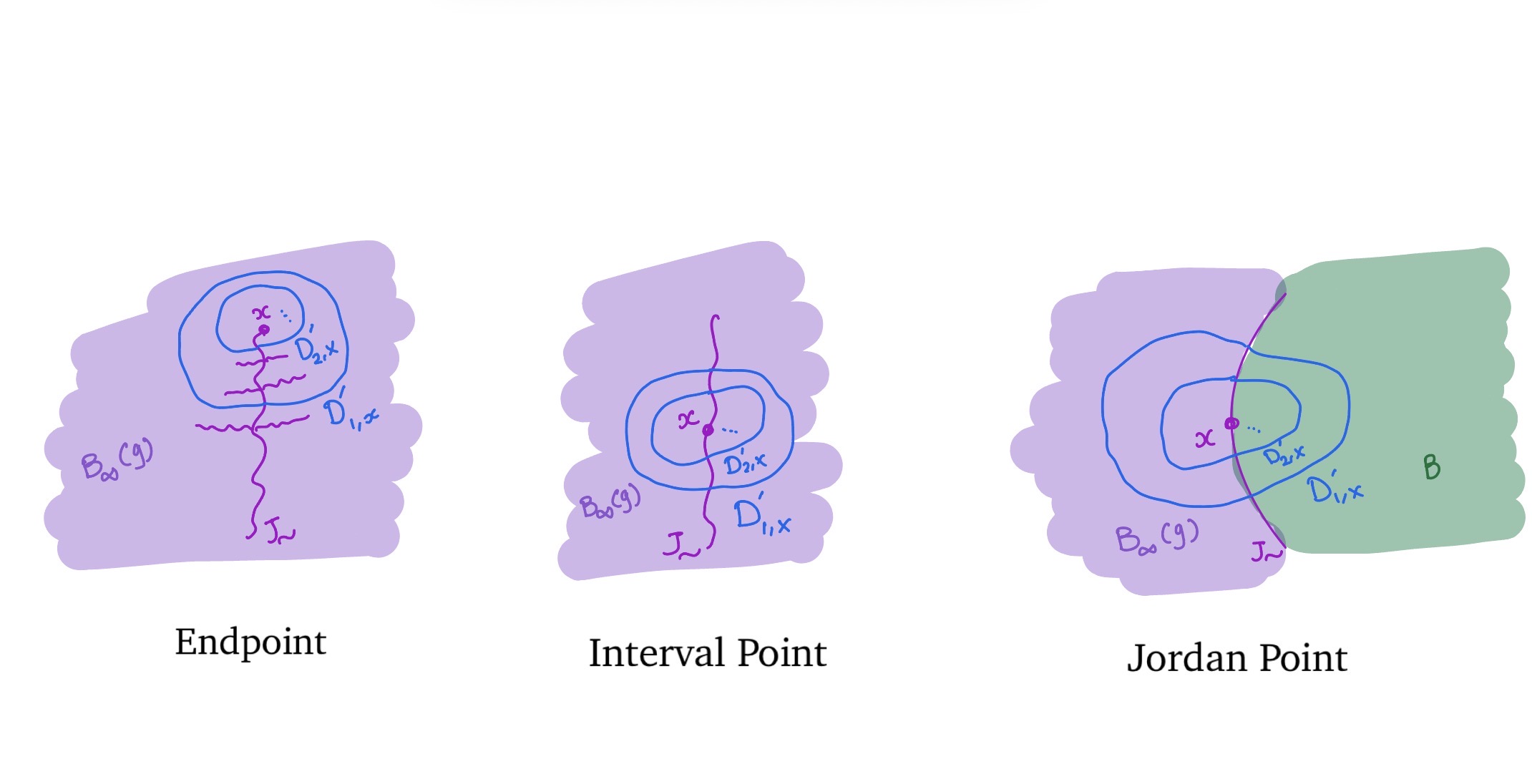}
    \caption{An illustration of the neighborhood bases $\{D'_{i,y}\}_{i \ge 1}$ around three different types of $J_\sim$ points}
    \label{fig:types-of-pts}
\end{figure}
\subsection{Construction of neighborhood bases around suitable fibers of \texorpdfstring{$\phat$}{phat}}
\label{subsec:nbhd-bases}
Our proof of Theorem~\ref{thm:classification-of-bialgebraic} hinges on constructing sufficiently small simply connected neighborhoods of carefully chosen points in the Julia set of $f$, with controlled topology and additional structural properties. 

Using Lemma \ref{lem:model-types}, there are three types of models $J_\sim$. In each case, we construct a suitable base around certain points of $J_\sim$. 
\begin{construction}
\label{cons:base-at-model}
Let $y = \iota(t) \in J_\sim$ be a point in $J_\sim$. We construct a neighborhood basis near $y$ in the next three cases as follows (see also Figure \ref{fig:types-of-pts}): 
\begin{enumerate}
    \item\textnormal{\textbf{ $t$ is an endpoint:}  Let \(\{t_i\}_{i\ge 1}, \{t_i'\}_{i\ge 1}\subset \Ss^{1}\) be as in Definition \ref{def:endpoints}. Let $\gamma_i$ be the hyperbolic geodesic in $\D$ connecting $t_i$ and $t'_i$. Then, $\phat \circ \Psi(\gamma_i)$ is a Jordan curve which divides $\C$ into two connected open sets. Let $D'_{i, y}$ be the open piece containing $y$. }
    \item \textnormal{\textbf{$t$ is an interval point:}} In this case there is another point $t' \in\Ss^1$ such that $\iota(t') = y$. Choose \(\{t_{i,1}\}_{i\ge 1}, \{t_{i,2}\}_{i\ge 1}\subset \Ss^{1}\) and \(\{t'_{i,1}\}_{i\ge 1}, \{t'_{i,2}\}_{i\ge 1}\subset \Ss^{1}\) such that $\iota(t_{i,1}) = \iota(t_{i,2})$, $\iota(t'_{i,1}) = \iota(t'_{i,2})$, $t$ is on the shorter arc connecting $t_{i,1}$ and $t_{i,2}$ and $t'$ is on the shorter arc connecting $t'_{i,1}$ and $t'_{i,2}$ for all $i \ge 1$, and $t_{i,j} \to t$ and $t'_{i,j} \to t'$ as $i \to \infty$ for $j = 1,2$. Let $\gamma_{i}$ and $\gamma'_i$ be the hyperbolic geodesic in $\D$ connecting $t_{i,1}$ and $t_{i,2}$, and $t'_{i,1}$ and $t'_{i,2}$, respectively. Then, $\phat \circ \Psi(\gamma_i \cup \gamma'_i)$ is a Jordan curve which divides $\C$ into two connected open sets. Let $D'_{i, y}$ be the open piece containing $y$. 

    \item \textnormal{\textbf{$t$ is a Jordan point:}} In this case $J_\sim$ divides $\C$ into $B_\infty(g)$ and a bounded component $B$. Choose \(\{t_{i}\}_{i\ge 1}, \{t'_{i}\}_{i\ge 1}\subset \Ss^{1}\) such that $t$ is on the shorter arc connecting $t_{i}$ and $t'_{i}$ for all $i \ge 1$, and $t_{i},t'_i \to t$ as $i \to \infty$. Let $\gamma_{i}$ be the hyperbolic geodesic in $\D$ connecting $t_{i}$ and $t'_{i}$. Let $p_i$ be a geodesic in $B$ connecting $\iota(t_{i})$ to $\iota(t'_{i})$. Then, $\left(\phat \circ \Psi(\gamma_i)\right) \cup p_i$ is a Jordan curve which divides $\C$ into two connected open sets. Let $D'_{i, y}$ be the open piece containing $y$. 
\end{enumerate}
\end{construction}

Naturally, we can pull back the neighborhoods $D'_{i,y}$ via $\phat$ to obtain open neighborhoods of $\phat^{-1}(y)$. This gives us the next construction.
\begin{construction}
\label{cons:base-at-fibers}
Given any $y \in J_\sim$ corresponding to an endpoint, interval point, or Jordan point, we define $D_{i, y} := \phat^{-1}(D'_{i,y})$ for all $i \ge 1$.
\end{construction}
% Also note that \(\Psi(\gamma_k)\) is a crosscut of \(B_\infty(g)\). It splits \(B_\infty(g)\) into two open connected components, one of which is precisely \(D''_i := D'_i \cap B_\infty(g)\).

The next series of results aim to give us a better understanding of the topological properties of the neighborhoods $D_{i,y}$. We start by showing that the sets $D_{i,y}$ are simply connected.
\begin{proposition}
\label{prop:D_i-is-base}
$D_{i,y} = \phat^{-1}(D'_{i,y})$ is a simply connected neighborhood base at $\phat^{-1}(y)$. 
\end{proposition}
\begin{proof}
Let $U\subset \C$ be an open set containing $\phat^{-1}(y)$. For ease of notation, we drop $y$ from $D_{i,y}$ and $D'_{i,y}$ and denote them by $D_i$ and $D'_i$ for the remainder of the proof. We claim that there exists $i\ge 1$ such that
$D_i\subset U$. Assume otherwise. Then for every $i\ge 1$ we can choose a point
$x_i\in D_i\cap U^c$. In particular, $\hat{\varphi}(x_i)\in D'_i$ for all $i\ge 1$, and hence
$\hat{\varphi}(x_i)\to y$ as $i\to\infty$. Since $U^c$ is closed and $\RS$ is compact, the sequence $\{x_i\}$ has a convergent subsequence $\{x_{n_i}\}$ with
$x_{n_i}\to x\in U^c$. By continuity of $\hat{\varphi}$ we obtain
\[
y =\lim_{i\to\infty}\hat{\varphi}(x_{n_i})=\hat{\varphi}(x).
\]
Thus $x\in \hat{\varphi}^{-1}(y) \subset U$ contradicting $x\in U^c$. Therefore, $\{D_i\}$ is indeed a neighborhood basis at $\phat^{-1}(y)$.

We now show that each $D_i$ is connected. Suppose that
$D_i = U \cup V$
for some disjoint open sets $U,V\subset \C$. Observe that $U \cup V$ also covers $\phat^{-1}(y)$ which is connected by monotonicity of $\phat$. So, we may assume, without loss of generality, that $\phat^{-1}(y) \subset U$. Since $D_i$'s form a basis at $\phat^{-1}(y)$, there is some $j \ge i$ such that $D_j \subset U$. Now note that by construction, $D'_i \setminus J_\sim$ has at most two connected components depending on the type of $y$. By \cite[Theorem 2]{block-lc-models}, $\phat$ is a homeomorphism when restricted to the Fatou components that are not collapsed to points by $\phat$. This implies that $\phat^{-1}(D'_i \setminus J_\sim)$ has the same number of components as $D'_i \setminus J_\sim$. Also, $D'_j$ clearly intersects all of the components of $D'_i \setminus J_\sim$. Therefore, using $D_j \subset U$ we can conclude that $U$ must intersect all connected components of $\phat^{-1}(D'_i \setminus J_\sim)$. Since $U$ and $V$ cover these connected components, we conclude that $U$ contains all of them. Hence, $\phat^{-1}(D'_i \setminus J_\sim) \subset U$. This also implies $V \subset \phat^{-1}(J_\sim)$. To finish the proof, we must show that $V$ is empty. So, we assume for the sake of contradiction that $V$ is non-empty.

Since $V$ is open and non-empty, it has non-empty intersection with at least one Fatou component of $f$, say $W$. But, $V \subset \phat^{-1}(J_\sim)$ shows that $W$ must contract by $\phat$ to a point say $w \in D'_i \cap J_\sim$. Thus, $W \subset D_i$. Since $U$ and $V$ cover $D_i$ and $W$ has non-empty intersection with $V$ we also get $W \subset V$.
% Let $W\neq B_\infty(f)$ be a Fatou component of $f$ such that $\hat{\varphi}(W)$ meets $D'_i\setminus J_\sim$ at a point outside $D''_i$. Then $\hat{\varphi}(W)\subset D'_i$, since $\hat{\varphi}(W)$ does not meet $\partial D'_i$. This is immediate if $\hat{\varphi}(W)$ is a singleton. Otherwise, by \cite[Theorem~2]{block-lc-models}, $\hat{\varphi}$ maps $W$ homeomorphically onto a component of $\RS\setminus J_\sim$ different from $B_\infty(g)$, and such a component clearly cannot meet $\partial D'_i$ which is contained in $\overline{B_\infty(g)}$. Consequently,
% $W\subset \hat{\varphi}^{-1}(D'_i)=D_i$.
% We claim that in fact $W\subset U$.

The boundary of $W$ also contracts to $w$ by $\phat$ by continuity. So, $\partial W \subset D_i$. Choose a point $z\in \partial W \subset J_f$. Any sufficiently small neighborhood of $z$ contained in $D_i$ must intersect $B_\infty(f)$, and therefore also meets $D_i\cap B_\infty(f)$. But, we have
\[
D_i\cap B_\infty(f) \subset \phat^{-1}(D'_i \setminus J_\sim) \subset U. 
\]
Since $z \in \partial W$ any neighborhood of $z$ also intersects $W \subset V$. Since $U$ and $V$ are disjoint, this forces $z\in U$ and $z \in V$ which is absurd. So,  $V$ is empty. It follows that $D_i$ is connected for all $i \ge 1$.

The final step is to prove that $D_i$ is simply connected. Let $\gamma\subset D_i$ be a simple closed curve, and let $B$ denote the bounded component of $\C\setminus \gamma$. We claim that $B\subset D_i$. Suppose otherwise, and choose a point $x\in B\setminus D_i$. Then $\hat{\varphi}(x)\in (D'_i)^c$. We can choose a path $\gamma'\subset (D'_i)^c$ joining $\infty$ to $\hat{\varphi}(x)$ as follows: 
\begin{itemize}
    \item If $\phat(x) \in \overline{B_\infty(g)}$, let $\gamma'$ be a simple path in $B_\infty(g) \setminus D'_i$ connecting $\infty$ to $\phat(x)$; and
    \item if $\phat(x)$ lies in some bounded component $W$ of $\C \setminus J_\sim$, we can choose a point $z \in \partial W \setminus D'_i$ and define $\gamma'$ to be the union of a simple path in $W \setminus D'_i$ connecting $\phat(x)$ to $z$ and a simple path in $B_\infty(g) \setminus D'_i$ connecting $\infty$ to $z$.
\end{itemize}

\begin{claim}
\label{claim:gamma'-conn}
$\hat{\varphi}^{-1}(\gamma')$ is connected.
\end{claim}
\begin{proof}
This is clear if $\phat(x) \in B_\infty(g)$ since $\phat|_{B_\infty(f)}$ is a homeomorphism onto its image. If $\phat(x) \in J_\sim$ and $\hat{\varphi}^{-1}(\gamma') = U \cup V$, then we may assume $\phat^{-1}(\phat(x)) \subset U$ since $\phat^{-1}(\phat(x))$ is connected. It follows that $U$ must intersect $\phat^{-1}(\gamma' - \{\phat(x)\})$ which is also connected. Hence, $\phat^{-1}(\gamma' - \{\phat(x)\}) \subset U$ which implies that $\phat^{-1}(\gamma') \subset U$ and yields that $\phat^{-1}(\gamma')$ is connected. The proof for the case of $\phat(x) \in W$ for some bounded component $W$ of $\C \setminus J_\sim$ is analogous.   
\end{proof}
Since $x\in \hat{\varphi}^{-1}(\gamma')$, the set $\hat{\varphi}^{-1}(\gamma')$ meets $B$. On the other hand, $\hat{\varphi}^{-1}(\gamma')$ is disjoint from $\gamma$ because $\gamma\subset \hat{\varphi}^{-1}(D'_i)=D_i$. Therefore, $\hat{\varphi}^{-1}(\gamma')$ is contained in a single component of $\C\setminus \gamma$, and since it meets $B$ it must be contained in $B$. In particular,
\[
\infty=\hat{\varphi}^{-1}(\infty)\in \hat{\varphi}^{-1}(\gamma')\subset B,
\]
contradicting that $B$ is bounded. This contradiction shows that $B\subset D_i$, and hence every simple closed curve in $D_i$ bounds a disk contained in $D_i$. Therefore, $D_i$ is simply connected.  
\end{proof}

It is important for our proof of Theorem \ref{thm:classification-of-bialgebraic} to also get a handle on the topology of the intersections $D_{i,y} \cap B_\infty(f)$. The next Proposition aims to achieve this.
\begin{proposition}
\label{prop:D_i-cap-b_infty-connected}
If $t \in \Ss^1$ corresponds to an endpoint, Jordan point, or interval point and $y = \iota(t)$, then $$D_{i,y} \cap B_\infty(f) = \phat^{-1}(D'_{i,y} \cap B_\infty(g)).$$
Consequently, if $t$ is a $J_\sim$-endpoint or a Jordan point, $D_{i,y} \cap B_\infty(f)$ must be connected and if $t$ is an interval point, $D_{i, y} \cap B_\infty(f)$ has exactly two connected components $D^1_{i,y}$ and $D_{i,y}^2$.  
\end{proposition}
\begin{proof}
The inclusion $D_{i,y} \cap B_\infty(f) \subset \phat^{-1}(D'_{i,y} \cap B_\infty(g))$ is immediate by the definition of $D_{i,y}$ and the fact that $\phat(B_\infty(f)) = B_\infty(g)$. The reverse inclusion $D_{i,y} \cap B_\infty(f) \supset \phat^{-1}(D'_{i,y} \cap B_\infty(g))$ is also clear since $\phat^{-1}(B_\infty(g)) = B_\infty(f)$ and 
\[
\phat^{-1}(D'_{i,y} \cap B_\infty(g)) \subset \phat^{-1}(D'_{i,y}) = D_{i,y}.
\]
The second conclusion follows since $\phat^{-1}|_{B_\infty(g)}$ is a homeomorphism onto $B_\infty(f)$.
\end{proof}

Proposition \ref{prop:D_i-cap-B_infty-2-conn-comps} shows that whenever $J_\sim$ is an interval and $y$ corresponds to an interval point $t \in \Ss^1$, $D_{i,y} \cap B_\infty(f)$ has two connected components $D_{i,y}^1$ and $D_{i,y}^2$. The next Proposition shows that any neighborhood $N$ of a point in the impression $\imp(t)$ must intersect both $D_{i,y}^1$ and $D_{i,y}^2$. Equivalently, the image of the neighborhood $N$ under $\phat$ intersects both $(D'_{i,y})^1$ and $(D'_{i,y})^2$.
\begin{proposition}
\label{prop:D_i-cap-B_infty-2-conn-comps}
Suppose that $J_\sim$ is an interval. Let $t \in \Ss^1$ be an interval point and set $y = \iota(t)$. Then, given any point $x \in \imp(t)$, any neighborhood $N$ of $x$ has non-empty intersection with both $D_{i,y}^1$ and $D_{i,y}^2$ (as defined in Proposition \ref{prop:D_i-cap-b_infty-connected}). 
\end{proposition}
\begin{proof}
Let
\begin{align}
V := \{x \in J_f :\;& x \in \phat^{-1}(y) \text{ for some non-endpoint } y \in J_{\sim}, \notag\\
&\text{and there exists an open neighborhood } N \ni x \text{ such that } \notag \\ 
&N \cap D_{i,y}^j = \emptyset \text{ for some } j \in \{1,2\}\}. \notag
\end{align}
Since the defining condition is clearly open, the set $V$ is open in $J_f$.

Let $V'$ be the union of $V$ with the impressions $\imp(e_1)$ and $\imp(e_2)$, where $e_1,e_2 \in \Ss^1$ correspond to the two endpoints of $J_{\sim}$. We claim that $f(V') \subset V'$.

First suppose that $x \in \imp(e_k)$ for some $k \in \{1,2\}$. Since endpoint fibers map to endpoint fibers, we have
\[
f(\imp(e_k)) = \imp(e_\ell)
\]
for some $\ell \in \{1,2\}$. Hence $f(x) \in V'$.

Now suppose that $x \in V$. Let $y := \phat(x)$, and let $N$ be an open neighborhood of $x$ as in the definition of $V$. If $f(x)$ belongs to the fiber of an endpoint of $J_{\sim}$, then $f(x) \in V'$ and there is nothing to prove. Thus we may assume that
\[
f(x) \in \phat^{-1}(y')
\]
for some non-endpoint $y' \in J_{\sim}$. Our goal is to show that after shrinking $N$ if necessary, $f(N)$ only intersects one of $D^1_{i,y'}$ and $D^2_{i,y'}$.

Set 
\[
N' := \phat(N). 
\]
After shrinking $N$ if necessary, we may assume that $N'$ is contained in the euclidean closure of one of
\[
(D'_{i,y})^1 := \phat(D_{i,y}^1)
\qquad \text{or} \qquad
(D'_{i,y})^2 := \phat(D_{i,y}^2).
\]
Without loss of generality assume that
\[
N' \subset \overline{(D'_{i,y})^1}.
\]
Note that 
\begin{equation}
\label{eqn:intersection-of-images}
\bigcap_{i \ge 1}g\left(\overline{(D'_{i,y})^1}\right) = y',
\end{equation}
by continuity of $g$.
Choose $i$ sufficiently large so that $g\left((D'_{i,y})^1\right)$ is contained in one of 
\[
(D'_{k,y'})^1 \qquad \text{or} \qquad (D'_{k,y'})^2,
\]
for some $k \ge 1$. We can ensure this since if $g\left((D'_{i,y})^1\right)$ continues to intersect both $(D'_{k,y'})^1 \text{ and }  (D'_{k,y'})^2$ in $p_i$ and $p'_i$, respectively, then because $g\left((D'_{i,y})^1\right) \subset B_\infty(g)$ is connected, it must also contain a path in $B_\infty(g)$ connecting $p_i$ and $p'_i$. Such paths must intersect at least one of the two rays from $\infty$ landing at the endpoints of $J_\sim$. Taking limits it follows that $$\bigcap_{i \ge 1}g\left(\overline{(D'_{i,y})^1}\right)$$ must also intersect one of these rays which contradicts equation \ref{eqn:intersection-of-images}. So, suppose without loss of generality that 
\[
g\left((D'_{i,y})^1\right)  \subset (D'_{k,y'})^1.
\]
We must have
\[
\phat(f(N)) = g(N') \subset g\left(\overline{(D'_{i,y})^1}\right) \subset \overline{(D'_{k,y'})^1}.
\]
Hence, $f(N)$ can only intersect $D_{k,y'}^1$ as desired. This shows that $f(x) \in V$ and therefore also $f(x) \in V'$. So, indeed
\[
f(V') \subset V'.
\]

To complete the proof of the proposition, it remains to show that $V$ is empty. Assume otherwise. Since $V$ is nonempty and open, it follows from \cite[Theorem 4.10]{milnor} that
\[
\bigcup_{n \ge 0} f^n(V)
\]
contains all of $J_f$. On the other hand, since $f(V') \subset V'$, we have
\[
f^n(V) \subset V'
\qquad \text{for all } n \ge 0.
\]
Hence $V'$ contains all of $J_f$. Thus $V' = J_f$. But, this is impossible. Indeed, if $t,t' \in \Ss^1$ are distinct and satisfy $y = \iota(t)=\iota(t')$, then
\[
\imp(t) \cap \imp(t') \neq \emptyset,
\]
by \cite[Theorem 48]{block-lc-models}. Any point
\[
x \in \imp(t) \cap \imp(t')
\]
must lie on the boundary of both $D_{i,y}^1$ and $D_{i,y}^2$ for all $i \ge 1$, and therefore cannot belong to $V'$. This contradiction shows that $V=\emptyset$.
\end{proof}

As an important consequence of Proposition \ref{prop:D_i-cap-B_infty-2-conn-comps} we obtain the next result.
\begin{corollary}
\label{cor:imps-are-equal}
Suppose that $J_\sim$ is an interval. For any $t,t' \in \Ss^1$, with $\iota(t) = \iota(t')$ we have $$\imp(t) = \imp(t').$$
\end{corollary}
\begin{proof}
Let $y = \iota(t)$ and take $x \in \imp(t)$. By Proposition \ref{prop:D_i-cap-B_infty-2-conn-comps}, any neighborhood of $x$ must intersect both $D_{i,y}^1$ and $D_{i,y}^2$ for all $i \ge 1$. Thus, we can construct sequences $\{x_i\}$ and $\{x_i'\}$ such that $x_i \in D_{i,y}^1$ and $x_i \in D_{i,y}^2$ for all $i \ge 1$ and $x_i, x'_i \to x$ as $i \to \infty$. By the defintion of $D_{i,y}^1$ and $D_{i,y}^2$ we must have 
\[
\Phi_f(x_i) \to t_1 \text{ and } \Phi_f(x_i') \to t_2,
\]
with $\{t_1,t_2\} = \{t,t'\}$. Thus, $x \in \imp(t') \cap \imp(t)$. Since $x$ was arbitrary we conclude
\[
\imp(t) \subset \imp(t) \cap \imp(t') \subset \imp(t').
\]
Repeating the same argument with $t'$, we get the reverse inclusion and the corollary follows. 
\end{proof}

The next important lemma is the final result of this section.  Roughly speaking, it shows that holomorphic functions obtained from a special class of analytic continuations are locally compatible with the equivalence relation $\sim_f$ induced on $J_f$.  This lemma will be a key ingredient in the proof of Theorem~\ref{thm:classification-of-bialgebraic}, and will also play an important role in showing that certain functions induce local symmetries of $J_f$.
\begin{lemma}
\label{lem:h-respects-lamination}
Let $t \in \Ss^1$ be a $J_\sim$-endpoint or Jordan point and let $y = \iota(t)$. Let $(t_i, t'_i, \gamma_i)_{i \ge 1}$ be as in Construction \ref{cons:base-at-model} and define $\tilde{\gamma}_i$ as the geodesic in $\C \setminus \bar\D$ connecting $t_i$ to $t'_i$. Let $U_i$ be the bounded component of $\C \setminus (\gamma_i \cup \tilde{\gamma_i})$ which is a simply connected neighborhood base at $t$. Suppose that for some $i \ge 1$, there is a holomorphic function $h:U_i \lra h(U_i)$ that satisfies 
\begin{equation}
h(U_i \cap \partial\D) = h(U_i) \cap \partial\D \: \text{ and } \: h(U_i \cap \D) \subset \D.
\end{equation}
Moreover, assume that $D_{i,y} \cap B_\infty(f) \subset \Psi(U_i \cap \D)$ and there is a holomorphic function $\tilde{h}: D_{i,y} \lra \tilde{h}(D_{i,y})$ that satisfies
\begin{equation}
\label{eqn:h-h-tilde}
\tilde{h}(z) = \Psi(h(\Phi(z))),
\end{equation}
for all $z \in D_{i,y} \cap B_\infty(f)$. Then $h$ and $\tilde{h}$ have the following properties:
\begin{enumerate}
    \item For every $\theta \in U_i \cap \partial\D$ with $\imp(\theta) \subset D_{i,y}$, we have $\tilde h(\imp(\theta)) = \imp(h(\theta))$ ; and
    \item For any $\sim$-class $[x]$ of $J_f$ such that $x \in D_{i,y}$, we have $\tilde h([x]) \subset [\tilde{h}(x)]$; and
    
    \item For any $\theta, \theta' \in U_i \cap \partial\D$ with $\imp(\theta), \imp(\theta') \subset D_{i,y}$, we have $$\theta \sim_f \theta' \ \Longrightarrow \ h(\theta) \sim_f h(\theta').$$  
\end{enumerate}

\end{lemma}
\begin{proof}
We begin with \((1)\) and, for ease of notation, we denote $D_{i,y}$ by $D_i$. Let $\{x_j\}_{j\ge 1}\subset U_i\cap \D$ be a sequence with $x_j\to \theta\in \Ss^1$ and suppose that $\Psi(x_j)\to z\in \imp(\theta)$. Then $\{h(x_j)\}_{j\ge 1}\subset h(U_i)\cap \D$ converges to $h(\theta)$. Using \eqref{eqn:h-h-tilde} and continuity, we obtain
\[
\tilde h(z)
=\lim_{j\to\infty}\tilde h(\Psi(x_j))
=\lim_{j\to\infty}\Psi\bigl(h(x_j)\bigr)
\in \imp\bigl(h(\theta)\bigr).
\]
Hence $\tilde h(\imp(\theta))\subset \imp(h(\theta))$. 

The reverse inclusion is shown similarly. Suppose we are given a sequence $\{y_j\}_{j \ge 1} \subset h(U_i) \cap \D$ converging to $h(\theta)$ such that $\Psi(y_j) \to w \in \imp(h(\theta))$. Then, using local surjectivity of $h$ to find a sequence $\{x_j\}_{j \ge 1}$ such that $x_j \to \theta$ and $y_j = h(x_j)$. Then, 
\[
\lim_{j \to \infty} \Psi(y_j) = \lim_{j \to \infty} \Psi(h(x_j)) = \lim_{j \to \infty} \tilde h(\Psi(x_j)) \subset \tilde h(\imp(\theta)).  
\]
Therefore, $\imp(h(\theta)) \subset \tilde h(\imp(\theta))$ and we are done with (1).

 Note that \((2)\), and \((3)\) are clear whenever $J_\sim$ is a Jordan curve since classes correspond to single impressions and the conclusions all follow from part \((1)\). So, we focus on the case where $t$ corresponds to an endpoint. 
 
 We prove \((2)\), and \((3)\) will follow immediately as a consequence of \((2)\) and $(1)$. Recall that $\sim$ is constructed in \cite[Section~4]{block-lc-models} by a transfinite procedure: to each ordinal $\alpha$ one associates an equivalence relation $\sim_\alpha$, with classes denoted by $[\cdot]_\alpha$. It suffices to show that
\[
\tilde h\bigl([x]_\alpha\bigr)\subset [\tilde h(x)]_\alpha
\]
for every ordinal $\alpha$ and every class $[x]_\alpha$ with $x \in D_i$. We proceed by transfinite induction on $\alpha$. Our approach is inspired by the proof of \cite[Theorem 30]{block-lc-models}.

The base case is $\alpha = 0$. By definition, $x\sim_0 y$ if and only if there exists a finite chain of impressions
\[
x \in K_1,\dots,K_s=\imp(\theta')  \ni y
\]
such that $K_1\cup\cdots\cup K_s$ is connected. Moreover, since $[x]_0\subset [x]_{\sim f} \subset D_i$, we have $K_j\subset D_i$ for each $j\in\{1,\dots,s\}$. By \((1)\), the sets $\tilde h(K_j)$ are impressions with $\tilde h(x) \in \tilde h(K_1)$ and $\tilde h(y) \in \tilde h(K_s)$, and $\tilde h(K_1\cup\cdots\cup K_s)$ is connected because $\tilde h$ is holomorphic on $D_i$. Therefore $\tilde h(x)\sim_0 \tilde h(y)$, and hence
\[
\tilde h([x]_0)\subset [\tilde h(x)]_0
\qquad\text{whenever }[x]_0\subset D_i.
\]

Now suppose that $\tilde h$ has the property that $\tilde h([x]_\beta) \subset [\tilde h(x)]_\beta$ whenever $[x]_\beta \subset D_i$ for all $\beta < \alpha$. Suppose $\alpha$ is a limit ordinal. Then, $x \sim_\alpha y$ if and only of $x \sim_\beta y$ for some $\beta < \alpha$. Then, by the inductive hypothesis, we must have $\tilde h(x) \sim_\beta \tilde h(y)$, which shows that $\tilde h(x) \sim_\alpha \tilde h(y)$. Hence, $\tilde h([x]_\alpha) \subset [\tilde h(x)]_\alpha$.

Now suppose that $\alpha$ has an immediate predecessor $\beta$. By definition, $x\sim_\alpha y$ if there exist finitely many sequences of $\sim_\beta$-classes
\[
\{K_{j,1}\}_{j \ge 1}\to K_1,\ \dots,\ \{K_{j,s}\}_{j \ge 1}\to K_s,
\]
such that $K_1,\dots,K_s$ form a chain of continua that joins any $x$ to $y$. Since $[x]_\alpha \subset [x]_\sim\subset D_i$, each $K_r$ is contained in $D_i$. Passing to subsequences if necessary, we may therefore assume that $K_{j,r}\subset D_i$ for all $1\le r\le s$ and all $j\ge 1$.

Applying $\tilde h$ and using continuity, we obtain
\[
\{\tilde h(K_{j,1})\}_{j \ge 1}\to \tilde h(K_1),\ \dots,\  \{\tilde h(K_{j,s})\}_{j \ge 1}\to \tilde h(K_s),
\]
and the continua $\tilde h(K_1),\dots,\tilde h(K_s)$ form a chain joining $\tilde h(x)$ to $\tilde h(y)$. By the inductive hypothesis, for each $j$ and $r$, the set $\tilde h(K_{j,r})$ is contained in some (possibly larger) $\sim_\beta$-class, which we denote by $\widetilde K_{j,r}$. Using the Blaschke selection theorem, after passing to further subsequences, we may assume that $\widetilde K_{j,r}$ converges to a continuum $\widetilde K_r\supset \tilde h(K_r)$ for each $1\le r\le s$. Then $\widetilde K_1,\dots,\widetilde K_s$ form a chain connecting $\tilde h(x)$ to $\tilde h(y)$, and hence $\tilde h(x)\sim_\alpha \tilde h(y)$. This completes the induction and the proof.
\end{proof}

\section{Algebraic local symmetries of polynomial Julia sets}
\label{sec:alg-symms}
\label{def:loc-sym}
In this short section, we prove a classification theorem for algebraic local symmetries of Julia sets of non-exceptional polynomials. We begin by recalling \cite[Definition~1]{Levin-symm}.
\begin{definition}[Local symmetries of $J_f$]
Let $U$ be a nonempty open subset of $\C$ and $\sigma: U \lra \sigma(U)$ be a holomorphic function. We say that $\sigma$ is a \textit{local symmetry of $J_f$} whenever
\[
\sigma^{-1}(J_f) \cap U = U \cap J_f.
\]
\end{definition}

Our proof of Theorem \ref{thm:classification-of-bialgebraic} relies heavily upon using the assumption that the image of an algebraic subvariety is an algebraic subvariety of the same dimension to construct a local symmetry of $J_f$ in the sense of Definition \ref{def:loc-sym}.  Local symmetries between locally connected Julia sets of polynomials of the same degree have been classified in \cite{Yusheng}. The case of pairs of rational functions has recently been studied in \cite{xie-ji} and \cite{Dujardin-Favre-Gauthier}. Such local symmetries tend to come from very specific algebraic functions. In fact, if $h:(U, U \cap J_f) \lra (V, V \cap J_g)$ is a local symmetry between $J_f$ and $J_g$, under some extra measure-theoretic assummptions (see \cite[Theorem 1.7]{xie-ji}), we can conclude that the graph $\{x \in U: (x,h(x))\}$ is contained in an $(f^a,g^b)$-preperiodic curve for some $a,b \ge 0$. The next proposition gives a similar classification of local symmetries of polynomial Julia sets that come from an algebraic function. 
\begin{proposition}
\label{prop:alg-symmetries}
Let $f$ be a non-exceptional polynomial of degree $d \ge 2$ and let $J_f$ be its Julia set. Suppose that $\sigma: U \lra \C$ is a holomorphic map such that $U \cap J_f \ne \emptyset$ and $(x, \sigma(x))$ gives a local parametrization of an irreducible algebraic curve $\mathcal{C}$ in $\C^2$. Moreover, suppose that
\[
\sigma^{-1}(J_f) \cap U = U \cap J_f.
\]
Then, $\mathcal{C}$ must be an $(f,f)$-preperiodic curve. 
\end{proposition}
\begin{proof}
From the discussion in \cite[Remark 3.4]{Dujardin-Favre-Gauthier} we deduce that we may assume without loss of generality that $\sigma(U) \subsetneq U$ and $\sigma^k$ defines a branch of an $(f,f)$-preperiodic curve $C$ satisfying
\[
(f^{k_1}, f^{k_1})(C) = (f^{k_2}, f^{k_2})(C),
\]
for some $k_1 < k_2$. Shrink $U$ if necessary and let $f^{-1}$ be a branch of the inverse of $f$ satisfying $f^{-1}(U) \subset U$. Let  $\sigma_1 := f^{k_1} \circ \sigma \circ f^{-k_1}$ which is still a local symmetry such that $(x, \sigma_1^{k}(x))$ is a branch of an $(f^{k_2 - k_1}, f^{k_2 - k_1})$-invariant curve. 

At the expense of replacing $\sigma_1$ with $\sigma_1^{-1}$ if necessary and using the classification of Medvedev and Scanlon \cite[Theorem 6.24]{Scanlon}, we must have
\[
\sigma_1^{k} = g
\]
for some polynomial $g$ commuting with $f^{k_2 - k_1}$. Note that $g$ cannot be the identity since that implies both $\sigma_1$ and its inverse have finite order which contradicts $\sigma(U) \subsetneq U$. By the result of Ritt on commuting polynomials \cite{Ritt} we deduce that $f$ and $g$ share an iterate i.e.
\[
f^a = g^b,
\]
for some $a, b \in \Z_{>0}$. By $\sigma_1^{k} = g$, $\sigma_1$ must commute with $g^b = f^a$. Hence, 
\[
(f^a, f^a)(x, \sigma_1(x)) = (f^a(x), f^a(\sigma_1(x))) = (f^a(x), \sigma_1(f^a(x))),
\]
showing that $(x, \sigma_1(x))$ is invariant under $(f^a,f^a)$. Finally, we use the hypothesis that $(x, \sigma_1(x))$ is a branch of an algebraic curve $C_1$ to conclude that $C_1$ is invariant under $(f^a,f^a)$. 

By the definition of $\sigma_1$, it is then easy to see that the algebraic curve defined by $(x, \sigma_1(x)) = (x, f^{k_1} \circ \sigma \circ f^{-k_1}(x))$ is precisely equal to $(f^{k_1},f^{k_1})C'$  where $C'$ is the curve defined by $(x, \sigma(x))$. So, $(f^{k_1},f^{k_1})C'$ is fixed by $(f^a,f^a)$ and we are done.   
\end{proof}

\section{Proof of Theorem \ref{thm:DAS-then-DALW} and some other useful reductions}
\label{sec:reductions}

We begin this section by proving Theorem \ref{thm:DAS-then-DALW}, which shows that the dynamical Ax--Schanuel conjecture implies the dynamical Ax--Lindemann--Weierstrass conjecture.

\begin{proof}[Proof of Theorem \ref{thm:DAS-then-DALW}]
Let $V_1$ be a proper irreducible algebraic subvariety of $\C^n$ of dimension $r$ containing an $f$-bialgebraic branch $V_1' \subset \D_R^n$ , and let $V_2$ be the Zariski closure of $\bpsi_{f,n}(V'_1)$. First suppose that $V_2$ is a hypersurface. If 
\[
(h_1(x), \dots,h_n(x))  \subset \D_R^n,
\]
is a local chart of $V_1$ contained in $V'_1$, then 
\[
(h_1(x),\dots,h_n(x), \Psi_f(h_1(x)),\dots,\Psi_f(h_n(x))),
\]
must be contained in $V_1 \times V_2$ which has dimension $r + n - 1 \le r + n$. Hence, by Conjecture \ref{conj:dyn-ax-schanuel} we must have that $V_2$ is contained in a proper $f$-special subvariety of $(\bP^1)^n$. Since $V_2$ is an irreducible hypersurface we conclude that $V_2$ itself is $f$-special. 

Now suppose that $V_2$ is a subvariety of dimension $k \ge r$ in $(\bP^1)^n$. After permuting the coordinates if necessary, we may assume that $V_2$ is mapped dominantly onto $(\bP^1)^k$ by the projection $\pi_{1,\dots,k}$. Take any $\ell \in \{k + 1,\dots, n\}$. We must have 
\[
\bpsi_{f, k + 1}(\pi_{1,\dots,k,\ell}(V'_1)) \subset \pi_{1,\dots,k,\ell}(V_2).
\]
So, by the proof of the hypersurface case above, $\pi_{1,\dots,k,\ell}(V_2)$ must be $f$-special for every $\ell \in \{k + 1,\dots, n\}$.

Let $S \subset \{k + 1,\dots, n\}$ be the subset of all indices $\ell$ such that $\pi_{1,\dots,k,\ell}(V_2)$ is contained in a fiber of $\pi_{k+1}: (\bP^1)^{k + 1} \lra \bP^1$. Then, for every $\ell \in \{k + 1,\dots,n\} \setminus S$, $\pi_{1,\dots,k,\ell}(V_2)$ is $F_{k + 1}$-preperiodic. We can choose $0 \le N < M$ such that 
\[
F_{k+1}^N(\pi_{1,\dots,k,\ell}(V_2)) = F_{k+1}^M(\pi_{1,\dots,k,\ell}(V_2)),
\]
for all $\ell \in \{k+1,\dots,n\} \setminus S$. It follows that $V_2$ is contained in a fiber of $\pi_S$ and that $\pi_{S^c}(V_2)$ satisfies
\[
F_{n-s}^N(\pi_{S^c}(V_2)) = F_{n-s}^M(\pi_{S^c}(V_2)),
\]
where $s = |S|$. Thus, $V_2$ is an $f$-special subvariety of $(\bP^1)^n$. 
\end{proof}

In the remainder of this section, we collect several useful reductions that will play an important role in the proof of Theorem~\ref{thm:DAS-totally-disconnected}. We first show that all irreducible components of $F_n$-preimages of $F_n$-preperiodic subvarieties are $F_n$-preperiodic as well.

\begin{proposition}
\label{prop:preimage-also-prep}
Suppose that $T$ is an irreducible $F_n$-preperiodic subvariety of $(\bP^1)^n$. Then, for every $m \ge 0$, all irreducible components of $F_n^{-m}(T)$ are also $F_n$-preperiodic.
\end{proposition}
\begin{proof}
Note that $F_n$, and hence every iterate $F_n^m$, is flat. Indeed, $F_n$ is a product of polynomial endomorphisms of $\bP^1$, each of which is finite and flat. It follows that by \cite[Corollary~9.6]{Hartshorne} that all irreducible components of $F_n^{-m}(T)$ have the same dimension, and each maps dominantly onto $T$ under $F_n^m$. Consequently, each irreducible component is preperiodic.
 
\end{proof}
As a result, we see that the irreducible components of $F_n$-preimages of $f$-special subvarieties must also be $f$-special.
\begin{proposition}
\label{prop:preimage-also-special}
Suppose that $T$ is an irreducible and $f$-special subvariety of $(\bP^1)^n$. Then, for every $m \ge 0$, all irreducible components of $F_n^{-m}(T)$ are also $f$-special.
\end{proposition}
\begin{proof}
This is immediate from the definition of $f$-special subvarieties and Proposition \ref{prop:preimage-also-prep}. 
\end{proof}
Let $p_{d,n}:(\bP^1)^n \lra (\bP^1)^n$ be the map given by 
\[
(z_1,\dots,z_n) \mapsto (z_1^d,\dots,z_n^d).
\]
The next useful reduction allows us to replace the set $Z$ in the statement of Conjecture \ref{conj:dyn-ax-schanuel} given by the parametrization 
\begin{equation}
\label{eqn:Z-param}
\left\{(h_0(x), \dots,h_{n-1}(x), \Psi_f(h_0(x)), \dots,\Psi_f(h_{n-1}(x))) : x \in D \subset \C^m\right\},
\end{equation}
with
\[
\left\{
\begin{aligned}
&\bigl(
h_0(x)^{d^M}, \dots, h_{n-1}(x)^{d^M},
\Psi_f\!\bigl(h_0(x)^{d^M}\bigr), \dots,
\Psi_f\!\bigl(h_{n-1}(x)^{d^M}\bigr)
\bigr): x \in D \subset \mathbb{C}^m
\end{aligned}
\right\}.
\]
for any $M \ge 0$.
\begin{proposition}
\label{prop:forward-image-red}
Let $\widetilde{Z}$ be an analytic branch contained in $\D_R^n$ and let $Z := \Gamma(\bpsi_{f,n}|_{\widetilde{Z}})$. Then, Conjecture \ref{conj:dyn-ax-schanuel} holds for $Z$ if and only if it holds for $(p_{d,n}^M, F_n^M)(Z)$ for some $M \ge 0$. 
\end{proposition}
\begin{proof}
Note that we must have 
\[
\overline{((p_{d,n}^M, F_n^M)(Z))}^\zar = (p_{d,n}^M, F_n^M)\left(\overline{Z}^\zar\right).
\]
It follows that 
\[
\dim_\C\left(\overline{((p_{d,n}^M, F_n^M)(Z))}^\zar \right) = \dim_\C\left((p_{d,n}^M, f^M)\left(\overline{Z}^\zar\right)\right) = \dim_\C\left(\overline{Z}^\zar\right).
\]
Therefore, 
\[
\dim_\C\left(\overline{Z}^\zar\right) \ge \dim_\C(Z) + n,
\]
holds if and only if 
\[
\dim_\C\left(\overline{((p_{d,n}^M, F_n^M)(Z))}^\zar \right) \ge \dim_\C(Z) + n = \dim_\C((p_{d,n}^M, F_n^M)(Z)) + n,
\]
holds. To finish the proof note that $\pi_{n+1,\dots,2n}((p_{d,n}^M, F_n^M)(Z))$ is contained in an $f$-special subvariety $T$ if and only if $\pi_{n+1,\dots,2n}(Z)$ is contained in an irreducible component of $F_n^{-M}(T)$ which is also $f$-special by Proposition \ref{prop:preimage-also-special}.
\end{proof}
Similar to Proposition \ref{prop:forward-image-red}, the next reduction allows us to replace the parametrization \ref{eqn:Z-param} with
\[
\left\{(\zeta_0h_0(x), \dots,\zeta_{n-1}h_{n-1}(x), \Psi_f(\zeta_0h_0(x)), \dots,\Psi_f(\zeta_{n-1}h_{n-1}(x))) : x \in D \subset \C^m\right\},
\]
for some $(\zeta_0, \dots, \zeta_{n-1}) \in \mu_{d^{\infty}}^n$.
\begin{proposition}
\label{prop:rotation-red}
Let $\widetilde{Z}$ be an analytic branch contained in $\D_R^n$ and let $Z := \Gamma(\bpsi_{f,n}|_{\widetilde{Z}})$. For any $(\zeta_1,\dots,\zeta_n) \in \mu_{d^\infty}^n$ let $Z_{\zeta_1,\dots,\zeta_n} := \Gamma(\bpsi_{f,n}|_{(\zeta_1,\dots,\zeta_n) \cdot \widetilde{Z}})$. Then, Conjecture \ref{conj:dyn-ax-schanuel} holds for $Z$ if and only if it holds for $Z_{(\zeta_1,\dots,\zeta_n)}$ and all $(\zeta_1,\dots,\zeta_n) \in \mu_{d^\infty}^n$. 
\end{proposition}
\begin{proof}
The reverse direction is clear. To prove the forward direction suppose that Conjecture \ref{conj:dyn-ax-schanuel} holds for $Z$ and let $(\zeta_1,\dots,\zeta_n) \in \mu_{d^\infty}^n$ be given. It follows from the definition of $Z_{\zeta_1,\dots,\zeta_n}$ that 
\[
(p_{d,n}^M,F_n^M)(Z) = (p_{d,n}^M,F_n^M)(Z_{\zeta_1,\dots,\zeta_n}),
\]
for some $M \ge 1$. We conclude the proof by Proposition \ref{prop:forward-image-red}.
\end{proof}

\section{\texorpdfstring{$f$}{f}-special sets}
\label{sec:f-special-sets}
The main result of this section is Proposition~\ref{prop:f-special-gives-bialg}, which shows that $f$-special sets naturally give rise to $f$-bialgebraic sets.

Recall that $F_n$ was the map on $\C^n$ given by the coordinatewise action of $f$. The following lemma is a consequence of the classification of Medvedev and Scanlon \cite[Theorem 6.24]{Scanlon} and describes the structure of $F_n$-invariant subvarieties of $\C^n$ in a clean manner.  
\begin{lemma}
\label{lem:form-of-invariant}
Suppose that $W \subset \C^n$ is an irreducible subvariety of dimension $r \ge 1$ that is invariant under $F_n$. Then, after permuting the coordinates we may assume that $W$ is given by the parametrization 
\[
(x_1,\dots, x_r, g_1(x_{i_1}), \dots, g_{n-r}(x_{i_{n-r}})),
\]
where $i_1,\dots,i_{n-r}$ belong to $\{1,\dots,r\}$ and 
\[
g_j = L_j \circ h^{m_j},
\]
for some $m_j \ge 0$, where $h^\ell = f$ for some $\ell \ge 1$ and $L_j$ is a linear function of finite order commuting with a compositional power of $h$.  
\end{lemma}
\begin{proof}
After permuting the coordinates if necassary, we may assume that the projection of $W$ onto the first $r$ coordinates is dominant. For any $1 \le j \le n-r$ we have the next commutative diagram
\[
\begin{tikzcd}
\C^n \arrow[r, "F_n"] \arrow[d, "\pi_{1,\dots,r, r + j}"'] &  \C^n \arrow[d, "\pi_{1,\dots,r, r + j}"] \\ 
\C^{r + 1} \arrow[r, "F_{r+1}"] &  \C^{r + 1}.
\end{tikzcd}
\]
The diagram shows that $\pi_{1,\dots,r,r + j}(W)$ must  be an invariant hypersurface in $\C^{r + 1}$ given that $W$ is invariant. Hence, by \cite[Theorem 6.24]{Scanlon}, we get $x_{r + j} = g_{j}(x_{i_j})$ or $x_{i_j} = g_j(x_{r + j})$ for some $i_j \in \{1,\dots, r\}$, where $g_j$ is a polynomial commuting with $f$. In the case $x_{i_j} = g_j(x_{r + j})$, we swap the coordinates $i_j$ and $r + j$. After this permutation $W$ still clearly projects to the first $r$ coordinates dominantly. Moreover, this permutation replaces the polynomials $g_s$ with $s \le j - 1$ and $i_s = i_j$ with $g_s \circ g_j$ that still commute with $f$. Continuing in this manner we obtain the desired parametrization for $W$. The fact that $g_j$ has the desired structure described in the statement of the Lemma is a consequence of \cite[Theorem 6.24]{Scanlon}.
\end{proof}

\begin{lemma}
\label{lem:commutation-of-L-and-h}
Let $g_j$ be as in the statement of Lemma \ref{lem:form-of-invariant} for some $j \in \{1,\dots, n - r\}$. Then, 
\[
\Phi_f \circ g_j(x) = \zeta \Phi_f(x)^{d_1},
\]
where $d_1 = \deg(g_j)$ and $\zeta \in \mu_\infty$. 
\end{lemma}
\begin{proof}
Note that we have $h^\ell = f$. Hence, any B\"ottcher coordinate for $h$ must also be a B\"ottcher coordinate for $f$. Using the fact that the B\"ottcher coordinate is unique up to a $(d-1)$-th root of unity we conclude
\[
h\circ \Psi_f(x) = \Psi_f(\theta x^{d_2}) \ \  \text{ and } \ \ \Phi_f \circ h(x) = \theta \Phi_f(x)^{d_2}
\]
where $d_2 = \deg(h)$ and $\theta$ is a $(d - 1)$-th root of unity. Also, because $L$ commutes with some power of $h$, we conclude that $L$ must preserve $B_\infty(h) = B_\infty(f)$. So, $$\Phi_f \circ L \circ \Psi_f$$ must be an automorphism of $\D$ of finite order fixing $0$. Thus, $$\Phi_f \circ L \circ \Psi_f(x) = \theta'x$$ for some $\theta' \in \mu_\infty$. Putting everything together we get 
\[
\Phi_f \circ g_j(x) = \Phi_f \circ L \circ h^{m_j}(x) = \theta' \Phi_f(h^{m_j}(x)) = \theta' \cdot \theta^p \Phi_f(x)^{d_2^{m_j}},
\]
for some $p \ge 1$. Setting $\zeta = \theta' \cdot \theta^p$ and $d_1 = d_2^{m_j}$ finishes the proof.
\end{proof}
\begin{proposition}
\label{prop:f-special-gives-bialg}
Suppose that $W \subset \C^n$ is an irreducible $F_n$-preperiodic subvariety and let $W_1$ be a connected component of $W \cap U_\infty(f)^n$. Then, 
\[
(\Phi_f, \dots, \Phi_f)(W_1)
\]
is an  $f$-bialgebraic set. 
\end{proposition}
\begin{proof}
Suppose that $W \subset \C^n$ is $f$-special. After projecting to the coordinates of $W$ that are non-constant we may assume without loss of generality that $W$ is $F_n$-preperiodic and the projections of $W$ to all of the coordinates are dominant. 

First suppose that $W$ is $F_n$-periodic. Since $\Psi_f$ and $\Phi_f$ remain unchanged after replacing $f$ with an iterate, we may replace $f$ with an iterate and assume that $W$ is fixed under $F_n$. Also, using Lemma \ref{lem:form-of-invariant}, we can permute the coordinates if necessary and assume that $W$ is given by 
\[
(x_1,\dots, x_r, g_1(x_{i_1}), \dots, g_{n-r}(x_{i_{n-r}})),
\]
where $g_j$ are as in the statement of Lemma \ref{lem:form-of-invariant}. Given the form of $g_j$'s it is also straightforward to see that 
\[
W \cap U_\infty(f)^n = \{(x_1,\dots, x_r, g_1(x_{i_1}), \dots, g_{n-r}(x_{i_{n-r}}))): (x_1,\dots,x_r) \in U_\infty(f)^r\},
\]
which is a connected set. Using Lemma \ref{lem:commutation-of-L-and-h} we get $(\Phi_f,\dots,\Phi_f)(W \cap U_\infty(f)^n)$ is equal to the set
\[
W' := \{(\Phi_f(x_1),\dots, \Phi_f(x_r), \zeta_1\Phi_f(x_{i_1})^{d_1}, \dots, \zeta_{n-r}\Phi_f(x_{i_{n-r}})^{d_{n-r}}): (x_1,\dots,x_r) \in U_\infty(f)^r \},
\]
where $d_j = \deg(g_j)$ for all $1 \le j \le n - r$ and for some $\zeta_1,\dots,\zeta_s \in \mu_\infty$. So, $W' = V \cap \D_R^n$ where $V$ is  the $r$ dimensional and irreducible subvariety of $\C^n$ defined by the equations
\[
x_{r + 1} = \zeta_1x_{i_1}^{d_1}, \dots, x_{n} = \zeta_{n-r}x_{i_{n-r}}^{d_{n-r}}. 
\]
Hence, $(\Phi_f,\dots,\Phi_f)(W \cap U_\infty(f)^n)$ is a branch of an algebraic set of dimension $r$. By definition, its image under $\bpsi_{f,n}$ is contained in $W$ which is also $r$ dimensional. Hence, $$(\Phi_f,\dots,\Phi_f)(W \cap U_\infty(f)^n)$$ is bialgebraic. 

Now, suppose that $W$ is $F_n$-preperiodic with $W^m := F_n^m(W)$ periodic for some $m \ge 1$. Suppose that $W_1$ is a branch of $W$ contained in $W \cap U_\infty(f)^n$. Then, by properness of $F_n$, $F_n^{m}(W_1)$ is a branch of $W^m$ contained in $W^m \cap U_\infty(f)^n$. By the argument above, 
\[
(\Phi_f, \dots, \Phi_f)(F_n^{m}(W_1))
\]
must be a branch of an irreducible variety $V$ of dimension $r$. Let $p_{d,n}$ be the map defined in Section \ref{sec:reductions}. It follows that 
\[
(\Phi_f, \dots, \Phi_f)(W_1)
\]
must be a branch of an irreducible component $V_1$ of $p_{d,n}^{-m}(V)$ which must also be $r$ dimensional by \cite[Corollary 9.6]{Hartshorne} and using the fact that $p_{d,n}$ is flat by the miracle flatness theorem \cite[\href{https://stacks.math.columbia.edu/tag/00R4}{Lemma 00R4}]{stacks-project}. Therefore, $(\Phi_f, \dots, \Phi_f)(W_1)$ is bialgebraic. 
\end{proof}

\section{Monodromy of B\"ottcher coordinates of polynomials with disconnected Julia sets}
\label{sec:monodromy}
In this section, we recall the nontrivial monodromy of the function $\Psi_f$, as discussed in \cite[Section~6]{schmidt}, in the case where $J_f$ is disconnected. Throughout, we assume that $J_f$ is disconnected, so that $R<1$. We recall the definitions given in \cite[Section 6]{schmidt}. 
Let $C_{f} := \{z \in B_\infty: f'(z) = 0\}$. For any $c \in C_{f}$, let $n_c \ge 1$ be the smallest integer such that $f^{n_c}(c) \in \Psi(\D_R)$. Let 
\[
C_n := f^{-n}(C_{f})
\]
For any $n \ge 0$, let 
\[
D'_n := \D_{R^{1/d^n}}^\ast.
\]
For any $c \in C_{f}$ let
\[
E_{n,c} = \{z \in \D: \Psi_f(z^{d^n}) = f^{n_c}(c)\},
\]
and let
\[
E_n := \bigcup_{c \in C_{f}} E_{n,c}
\]
(see Figure \ref{fig:E_n,c}). We define 
\[
U_n := \{\zeta z : z \in \cup_{m \le n}E_m, \zeta \in \mu_{d^n}\}, 
\]
and 
\[
\mathcal{U} := \bigcup_{n \ge 1}U_n. 
\]
Lastly, we let
\[
D_n := D'_n \setminus U_n. 
\]
As discussed in \cite[Section 6]{schmidt}, if we are given a path $\gamma : I \lra D_n$ with $\gamma(0) \in \D_R^\ast$ we can use the algebraic relation 
\begin{equation}
\label{eqn:nth-alg-eqn}
f^n(\Psi_f(z)) = \Psi_f(z^{d^n}),
\end{equation}
to analytically continue $\Psi_f$ along $\gamma$. Figure \ref{fig:disc-julia} provides a pictorial illustration, in the case of a quadratic polynomial, of the different branches of $\Psi_f$ that can arise through such analytic continuations.
\begin{proposition}
\label{prop:loop-monodromy}
Let $\gamma: I \lra D_n$ be a loop in $D_n$ with $\gamma(0) = \gamma(1) \in \D_R$ (see $\gamma_2$ shown in Figure \ref{fig:combined-paths} as an example). Continuing $\Psi(z)$ along this loop using the relation given by equation \eqref{eqn:nth-alg-eqn}, sends the branch $\Psi(z)$ to $\Psi(\zeta z)$ for some $\zeta \in \mu_{d^n}$. 
\end{proposition}
\begin{proof}
Let $z_0 = \gamma(0)$. Consider the local branch $(\Psi(z), \Psi(z^{d^n}))$ of the curve $f^n(x) = y$ near the point $(\Psi(z_0), \Psi(z_0^{d^n}))$. We can use equation \eqref{eqn:nth-alg-eqn} to continue $\Psi(z)$ along $\gamma$ to get an analytic function $\widetilde{\Psi}(z)$ defined near $z_0$ and a local branch $(\widetilde{\Psi}(z), \Psi(z^{d^n}))$ of the curve $f^n(x) = y$. This curve has degree $d^n$ which implies that there are at most $d^n$ such branches. On the other hand, for any $\zeta \in \mu_{d^n}$, $(\Psi(\zeta z), \Psi(z^{d^n}))$ also defines such a branch. The proposition then follows. 
\end{proof}

\begin{figure}
    \centering
    \includegraphics[width=1\linewidth]{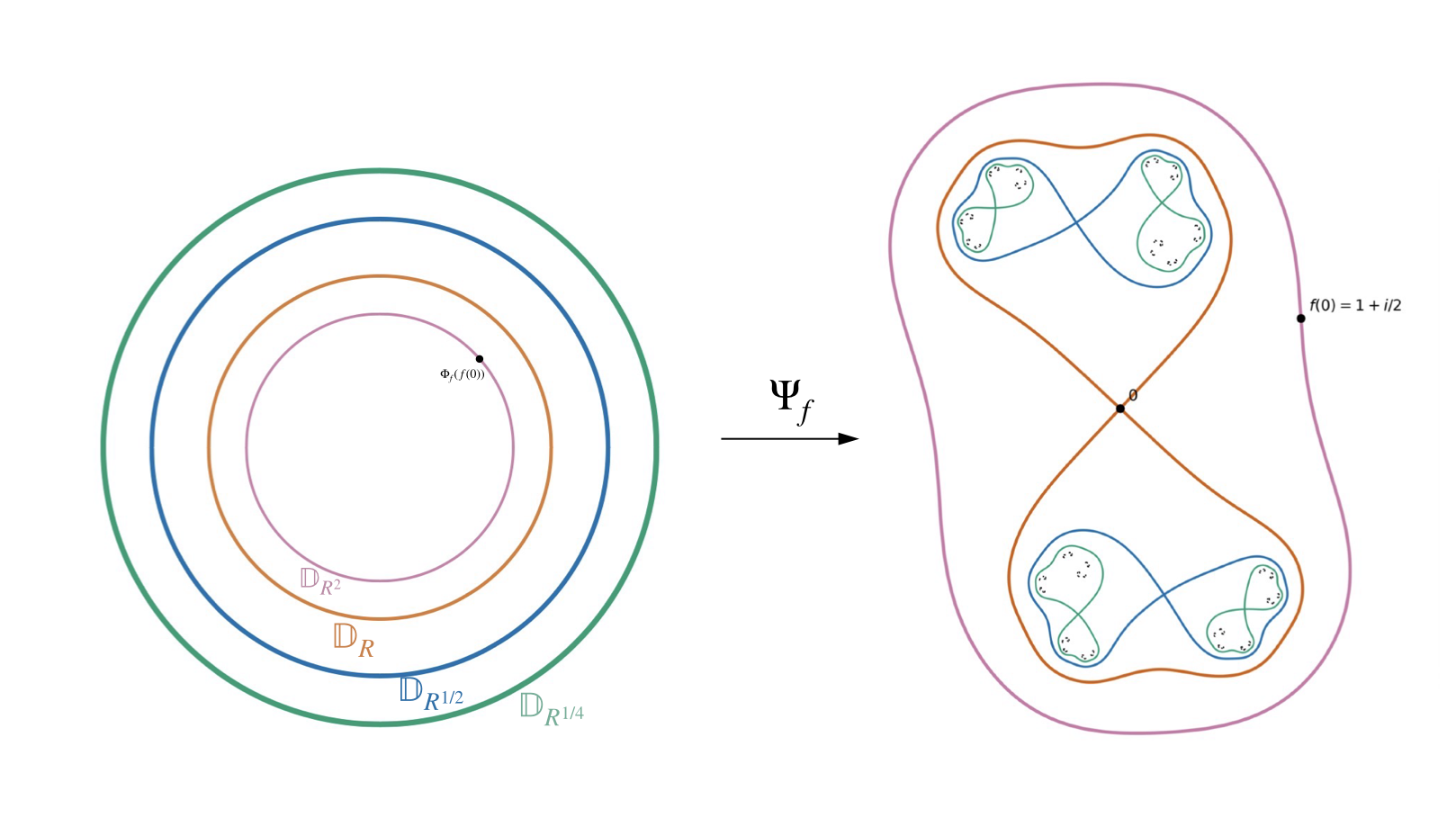}
    \caption{An illustration of the disconnected Julia set of $f = z^2 + 1 + \frac{i}{2}$ and the images of circles with radii $R^{1/4}, R^{1/2}, R, R^2$ under $\Psi_f$.}
    \label{fig:disc-julia}
\end{figure}
\begin{figure}
    \centering
 \includegraphics[width=0.7\linewidth]{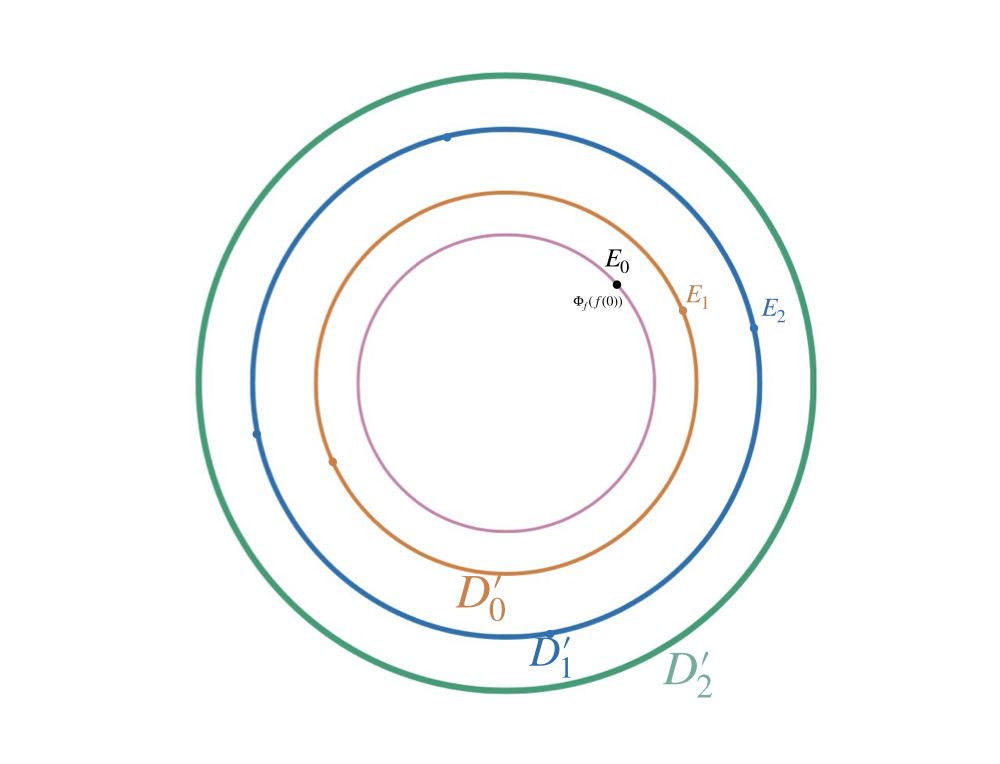}
    \caption{An illustration of the sets $E_n$ and $D'_n$ when $f(z) = z^2 + c$.}
    \label{fig:E_n,c}

\end{figure}
\newpage

Suppose that a path $\gamma:[0,1) \lra D_n$  with $\gamma(0) \in \D_R$ and $\gamma(1^-) = z_0 \in U_n$ for some $n \ge 1$ is given. We can use this path and the equation \eqref{eqn:nth-alg-eqn} to analytically continue $\Psi$. Even though the analytic continuation is only defined on $\gamma([0, 1))$, we can still make sense of $\Psi(z_0)$ by setting 
\begin{equation}
\label{eqn:extend-to-endpt}
\Psi(z_0) = \lim_{t \to 1}\Psi(\gamma(t)).
\end{equation}
This limit must exist since any limit point $w$ of the path $\Psi(\gamma(t))$ must satisfy $f^n(w) = \Psi(z_0^{d^n})$ by continuity, and there are only finitely many such values of $w$. Since the set of all limit points on the path $\Psi(\gamma(t))$ is connected, the desired conclusion follows.   
\begin{definition}
\label{def:precritical-cont-path}
Let $c \in C_{f}$ and $z_0 \in E_{n,c}$ be given. We call a path $\gamma:[0,1) \lra D_n$ a \textit{precritical continuation path for the pair $(c,z_0)$} whenever $\gamma(0) \in \D_R$, $\gamma(1^-) = z_0$, and continuing the relation given by equation \eqref{eqn:nth-alg-eqn} along $\gamma$ results in $\Psi(z_0) \in f^{-(n-n_c)}(c)$ (see equation \eqref{eqn:extend-to-endpt} and the preceding discussion). We call a path a precritical continuation path if it is a precritical continuation path for some pair $(c, z_0)$. 
\end{definition}

The goal of the next lemma is to make the statement and the proof of \cite[Lemma 6.1]{schmidt} a little more detailed.
\begin{lemma}[Lemma 6.1 of \cite{schmidt}]
\label{lem:loop-monodromy-around-E_n}
Let $z_0 \in E_{n, c}$ for some $c \in C_{f}$ and some $n \ge n_c$. Let $\gamma: [0,1) \lra D_n$ be a precritical continuation path for the pair $(c,z_0)$. Let $\gamma_1$ be a small simple loop contained in a neighborhood $U$ of $z_0$ such that $p_{d^n}(z)$ is a local biholomorphism onto its image when restricted to $U$. Let $s \in [0,1]$ be the smallest $s$ such that $\gamma(s) \in \gamma_1$. Then, continuing $\Psi$ using equation \eqref{eqn:nth-alg-eqn} along the loop (illustrated in Figure \ref{fig:combined-paths})
\[
\gamma_2 := \gamma[0,s) \cup \gamma_1 \cup \gamma[0,s)^{-1},
\]
sends $\Psi(z)$ to $\Psi(\zeta z)$ for some $\zeta \in \mu_{d^n}$ of order at least $2^{n-n_c + 1}$.
\end{lemma}
\begin{figure}
    \centering
    \includegraphics[width=1\linewidth]{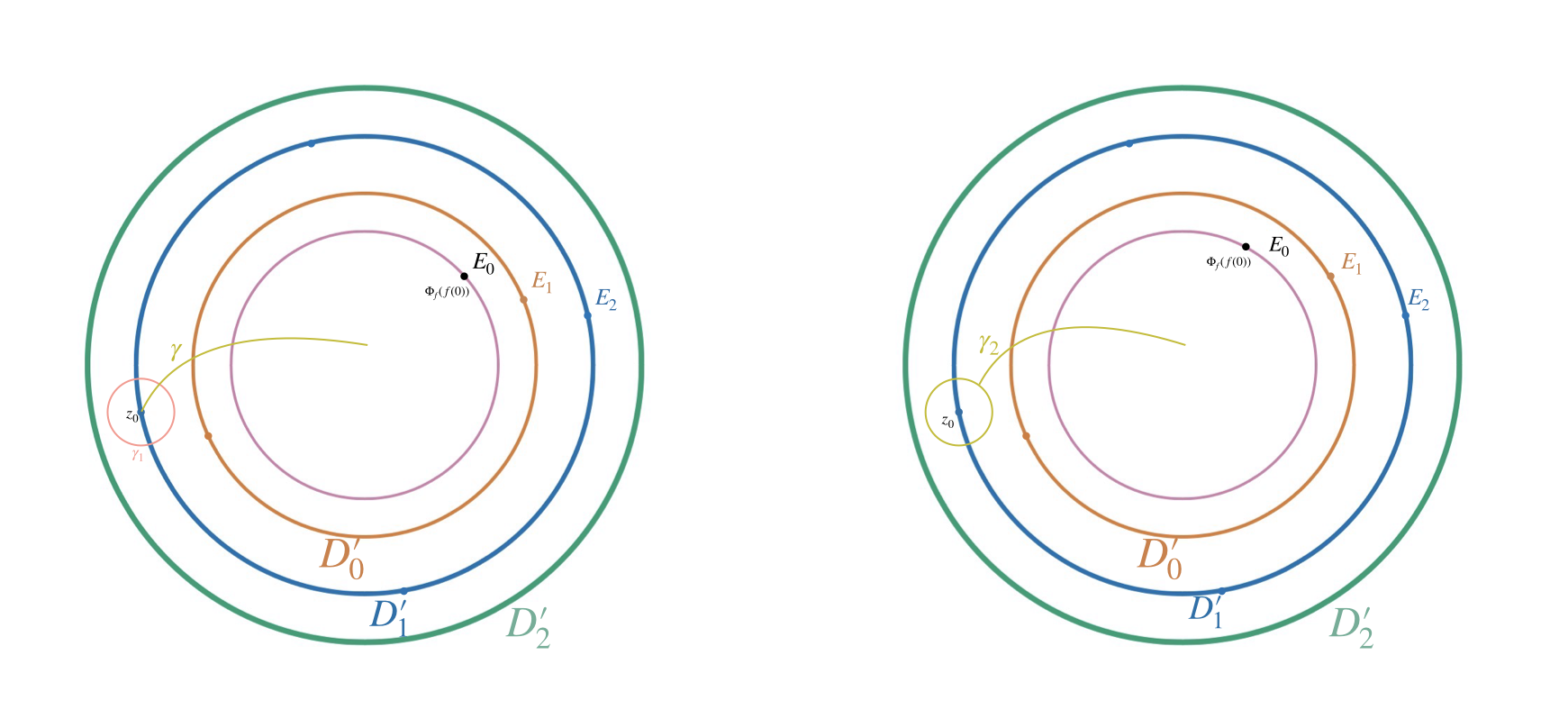}
    \caption{The combination of paths $\gamma$ and $\gamma_1$ to get a loop $\gamma_2$ around $z_0$.}
    \label{fig:combined-paths}
\end{figure}
\begin{proof}
    We argue by induction on $n$.

\medskip

\noindent\textbf{Base case: $n=n_c$.}
We continue $\Psi$ using
\[
\Psi(z^{n_c}) \;=\; f^{n_c}(\Psi(z)).
\]
Set $w:=\gamma(s)$. Analytically continuing $\Psi$ along $\gamma([0,s])$ to the point $w$, the pair
\[
\bigl(\Psi(z),\,\Psi(z^{n_c})\bigr)
\]
parametrizes the curve $y=f^{n_c}(x)$ near the point $\bigl(c,f^{n_c}(c)\bigr)$.

Near the critical value $f^{n_c}(c)$, the curve $y=f^{n_c}(x)$ admits a Puiseux-type parametrization
\[
\bigl(g(t),\,t^i+f^{n_c}(c)\bigr),
\]
where $i\ge 2$ is minimal and $g$ is analytic near $0$ with $g(0)=c$. If $\gamma_1$ is chosen sufficiently small, then $w$ is close to $z_0$, hence $\Psi(w)$ is close to $c$ by the assumption on $\gamma$. Therefore the local branch
\[
\bigl(\Psi(z),\Psi(z^{n_c})\bigr)
\]
must coincide with $\bigl(g(t),t^i+f^{n_c}(c)\bigr)$ in a neighborhood of $(c, f^{n_c}(c))$. 

Continuing $\Psi$ along $\gamma_1$ is equivalent (via $\Psi(z^{n_c})=f^{n_c}(\Psi(z))$) to continuing the second coordinate $\Psi(z^{n_c})$ along the loop $p_{d^{n_c}}(\gamma_1)$. By assumption on $\gamma_1$, the loop
\[
\Psi\!\bigl(p_{d^{n_c}}(\gamma_1)\bigr)
\]
is a simple loop around $f^{n_c}(c)$. Since the two parametrizations agree, this continuation acts by
\[
\bigl(g(t),\,t^i+f^{n_c}(c)\bigr)\longmapsto \bigl(g(\lambda t),\,t^i+f^{n_c}(c)\bigr)
\]
for some primitive $i$th root of unity $\lambda$. Hence, after continuing along $\gamma_1$, $\Psi$ is sent to a local analytic function $\widetilde\Psi$ with $\widetilde\Psi \neq \Psi$ because $g(t)\not= g(\lambda t)$ by the minimality of $i$.

Finally, continuing back along $\gamma([0,s))^{-1}$, Proposition~\ref{prop:loop-monodromy} implies that we end at $
\Psi(\zeta z)$ for some $\zeta\in\mu_{d^{n_c}}$.
Since the continuation produced a different branch, we must have $\zeta\neq 1$, so $\zeta$ has order at least $2$. This proves the base case.

\medskip

\noindent\textbf{Inductive step.}
Assume the lemma holds for $n-1\ge n_c$. Let $\gamma_1,\gamma_2$ be as in the statement of the lemma, and set $\tilde z:=z^d$. We first continue $\Psi$ along $\gamma_2$ using \eqref{eqn:nth-alg-eqn}. Then the functional equation
\begin{equation}\label{eqn:first-func-eqn}
\Psi(\tilde z)=f(\Psi(z))
\end{equation}
defines $\Psi(\tilde z)$ along the path $p_d(\gamma_2)$. Since $z^d$ is a finite covering map from $D_n$ onto $D_{n-1}$, the resulting function $\Psi(\tilde z)$ is analytic in the variable $\tilde z$.

Observe that
\[
\Psi\!\bigl(\tilde z^{d^{n-1}}\bigr)=\Psi(z^{d^n})=f^n(\Psi(z))=f^{n-1}(\Psi(\tilde z)).
\]
Moreover,
\[
\lim_{x\to 1}\Psi\!\bigl(p_d(\gamma(x))\bigr)
=\lim_{x\to 1} f(\Psi(\gamma(x)))
\in f^{-(n-n_c-1)}(c),
\]
since by assumption $\lim_{x\to 1}\Psi(\gamma(x))\in f^{-(n-n_c)}(c)$.
Therefore, by the inductive hypothesis, analytically continuing $\Psi(\tilde z)$ along $p_d(\gamma_2)$ sends $\Psi(\tilde z)$ to $\Psi(\zeta \tilde z)$ for some
\[
\zeta\in\mu_{d^n}
\quad\text{with}\quad
\ord(\zeta)\ge 2^{\,n-n_c}.
\]

On the other hand, Proposition~\ref{prop:loop-monodromy} shows that continuing $\Psi(z)$ along $\gamma_2$ sends it to $\Psi(\zeta_1 z)$ for some $\zeta_1\in\mu_{d^n}$. Applying \eqref{eqn:first-func-eqn} along $\gamma_2$ gives
\[
\Psi(\zeta w)=f(\Psi(\zeta_1 z))
=\Psi(\zeta_1^d z^d)
=\Psi(\zeta_1^d w),
\]
hence $\zeta=\zeta_1^d$. Consequently,
\begin{equation}\label{eqn:orders}
\ord(\zeta)=\frac{\ord(\zeta_1)}{\gcd\!\bigl(d,\ord(\zeta_1)\bigr)}.
\end{equation}

Now $2\le 2^{n-n_c}\le \ord(\zeta)\le \ord(\zeta_1)$ and $\ord(\zeta_1)\mid d^n$. In particular, $\gcd(d,\ord(\zeta_1))\ge 2$. Combining this with \eqref{eqn:orders} and the bound $\ord(\zeta)\ge 2^{n-n_c}$ yields
\[
\ord(\zeta_1)\;\ge\;2\,\ord(\zeta)\;\ge\;2^{\,n-n_c+1},
\]
which is exactly the desired conclusion.

\end{proof}
It is crucial in our proofs to show that precritical continuation paths exist. The next proposition shows that any path $\gamma \subset D_n$ with $\gamma(0) \in \D_R$ and $\gamma(1) \in E_{n,c}$ can become a precritical continuation path after a suitable rotation by a $d^n$-th root of unity. 
\begin{proposition}
\label{prop:path-for-non-triv-monodromy}
Suppose $c \in C_{f}$, $n \ge n_c$, $y_0 \in f^{-(n-n_c)}(c)$, and a path $\gamma:[0,1) \lra D_n$ with $\gamma(0) \in \D_R$ and $\lim_{x \to 1} \gamma(x) = z_0$ for some $z_0 \in E_{n,c}$ are given. Then, there exists some $\zeta \in \mu_{d^n}$ such that continuing $\Psi$ along $\gamma' := \zeta \cdot \gamma$ satisfies $\lim_{x\to 1} \Psi(\gamma'(x)) = y_0$. 
\end{proposition}
\begin{proof}
We can choose a small neighborhood $N$ around $z_0$ such that $\Psi(z^{d^n})$ sends $N$ biholomorphically onto a neighborhood $N'$ of $f^{n_c}(c)$. By our assummption on $\gamma$, the path must eventually be contained in $N$. So, there exists $0 < s_0 < 1$ such that $\gamma(s) \in N$ for all $s \ge s_0$.
\begin{figure}[t] % [h] = here, [t] = top, [b] = bottom
    \centering
    \includegraphics[width=0.5\textwidth]{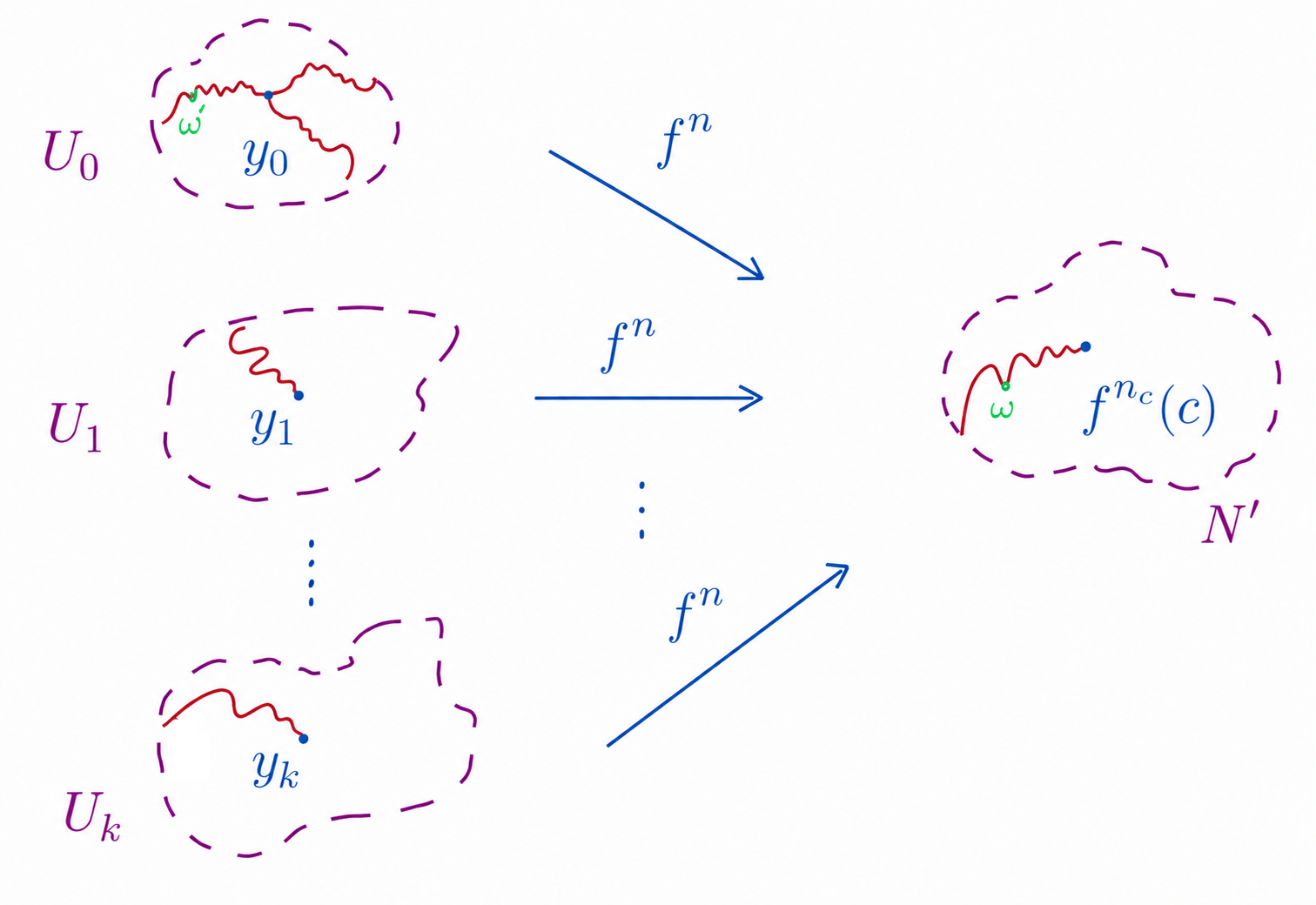} % Adjust width
    \caption{Preimages of $N'$ and $\Psi(p_{d^n}(\gamma_1)) \cap N'$ under $f^n$.}
    \label{fig:preimages-of-N'}
\end{figure}

Let $y_0, \dots, y_k$ be all the preimages of $f^{n_c}(c)$ under $f^{n}$. If we choose $N$ and $N'$ to be sufficiently small, we can find pairwise disjoint neighborhoods $U_i$ of $y_i$ such that $f^n(U_i) = N'$ for all $0 \le i \le k$ as shown in Figure \ref{fig:preimages-of-N'}. Now let $w := \Psi(\gamma(s_0)^{d^n})$ and let $w'$ be a preimage of $w$ under $f^n$ inside $U_0$ (see Figure \ref{fig:preimages-of-N'}). 

Choose a branch of $f^{-n}$ near $w$ satisfying $f^{-n}(w) = w'$. Using the fact that the preimages of $f^{-n}(p_{d^n}(\gamma[s,1))) \cap U_i$ are pairwise disjoint, we conclude that the analytic continuation of the branch of $f^{-n}$ along $\Psi(p_{d^n}(\gamma[s_0,1)))$ satisfies $\lim_{x \to f^{n_c}(c)}f^{-n}(x) = y_0$. Therefore, if we let 
\[
\widetilde{\Psi}(z) = f^{-n}(\Psi(z^{d^n})), 
\]
$\widetilde{\Psi}$ must satisfy 
\[
\lim_{s \to 1}\widetilde{\Psi}(\gamma(s)) = y_0. 
\]
If we use equation \eqref{eqn:nth-alg-eqn} to continue $\widetilde{\Psi}$ along $\gamma[0,s]^{-1}$ back to $\D_R$ we end up with $\Psi(\zeta z)$ for some $\zeta \in \mu_{d^n}$ by Proposition \ref{prop:loop-monodromy}. Finally, note that continuing $\Psi(\zeta z)$ along $\gamma$ is the same as continuing $\Psi(\tilde z)$ along $\zeta\cdot \gamma$ where $\tilde z = \zeta z$. So, we conclude that the path $\gamma' := \zeta \cdot \gamma$ has the desired property.  
\end{proof}

\section{Proof of Theorem \ref{thm:DAS-totally-disconnected}}
\label{sec:pf-of-DAS-disconnected}

In this section, we prove Theorem \ref{thm:DAS-totally-disconnected}, which proves dynamical Ax-Schanuel holds when $J_f$ is disconnected and $(h_0,\dots,h_{n-1})$ is a branch of an algebraic subset of $\C^n$. We start by proving the next key lemma which is a consequence of \cite[Th\'eor\`eme 1]{Laurent}.

\begin{lemma}
\label{lem:many-preimages-then-special}
Let $C$ be an irreducible curve in $\bG_m^2$ and $(c_1,c_2) \in \bG_m^2(\C)$ be given. Suppose there exists a sequence of distinct pairs $(a_n,b_n)$ of pairs of non-negative integers and a sequence of points $(z_n,w_n) \in C(\C)$ such that 
\[
z_n^{d^{a_n}} = c_1 \text{ and } w_n^{d^{b_n}} = c_2,
\]
for every $n \ge 1$. Then, $C$ is either of the form $\{x = c\}$ or $\{y = c\}$ for some $c \in \C^\ast$, or it must be a translate of an algebraic subgroup of $\bG_m^2(\C)$. In the latter case, if we further assume that $|c_1|,|c_2| \ne 1$, then $C$ must be a translate of an algebraic subgroup by a torsion element of $\bG_m^2(\C)$. 
\end{lemma}
\begin{proof}
Let $\Gamma$ be the division hull of the finitely generated subgroup of $\bG_m^2$ generated by $(c_1,1)$ and $(1,c_2)$. Each pair $(z_n,w_n)$ is then an element of $C \cap \Gamma$ which by Laurent's theorem \cite[Th\'eor\`eme 1]{Laurent} is contained in finitely many algebraic subgroups of the form $\{cx^i = y^j\}$. After replacing $(z_n,w_n)$ with a subsequence we may assume that all $(z_n,w_n)$'s are contained in the same set of one of the following forms 
\begin{enumerate}
    \item $\{x = c\}$ for some $c \in \C^\ast$; or
    \item $\{y = c\}$ for some $c \in \C^\ast$; or
    \item $\{cx^i = y^j\}$ for some $c \in \C^\ast$ and some non-zero integers $i, j$.
\end{enumerate}
In the first two cases we are done. So, we assume that we are in the third situation. Since there are infinitely many pairs $(z_n, w_n)$ in $C$, they are Zariski dense and we conclude that $C = \{cx^i = y^j\}$. We may assume that $i = j = 1$ at the expense of replacing $C$ with its image under the map $(z^i, w^j)$ and replacing $c_1$ and $c_2$ with $c_1^i$ and $c_2^j$. So, we assume that $$C = \{y = cx\}.$$

We have 
\begin{equation}
\label{eqn:z_n-w_n}
z_n^{d^{a_n}} = c_1 \text{ and } w_n^{d^{b_n}} = c^{d^{b_n}}z_n^{d^{b_n}} = c_2, 
\end{equation}
for all $n \ge 1$. Taking logs it follows that $d^{a_n}\log|z_n| = \log|c_1|$ and 
\begin{equation}
\label{eqn:c-c1-c2}
d^{b_n}\log(|c|) + d^{b_n}\log(|z_n|) = d^{b_n}\log(|c|) + d^{b_n - a_n}\log(|c_1|) = \log(|c_2|),
\end{equation}
for all $n \ge 1$. Since the pairs $(a_n,b_n)$ are distinct we get that the pairs $(d^{b_n}, d^{b_n-a_n})$ give infinitely many solutions to the equation 
\[
\log|c|X + \log|c_1|Y = \log|c_2|.
\]
Using Laurent's theorem again and arguing as before we can conclude that $\{\log|c|X + \log|c_1|Y = \log|c_2|\}$ is the coset of an algebraic subgroup of $\bG_m^2$ which forces $\log|c| = 0$ since $\log|c_1|, \log|c_2| \ne 0$ by hypothesis. By equation \eqref{eqn:c-c1-c2} we see that $d^{b_n - a_n} = \frac{\log|c_1|}{\log|c_2|}$ which means that $b_n - a_n$ is always a constant $k \in \Z$. Lastly, by equation \eqref{eqn:z_n-w_n} we conclude that $c^{d^{b_n}} = \frac{c_2}{c_1^{d^k}}$ for all $n \ge 1$. This implies that $c$ must be a root of unity by the pigeonhole principle.  
\end{proof}

The proof of Theorem \ref{thm:DAS-totally-disconnected} in the curve case hinges on analytic continuation along carefully chosen loops. The following definition makes precise what we mean by “carefully chosen”.
\begin{definition}
\label{def:increasing-psi-monodromy}
Let $\mathcal C \subset \C^n$ be a curve with $\mathcal C \cap \mathbb D_R^n \neq \varnothing$.
We say that $\mathcal C$ has the \emph{increasing-order $\Psi$-monodromy property} if there exist
\begin{itemize}
    \item[$\ast$] $c \in C_{f}$,
    \item[$\ast$] a point $w_0 \in \mathbb D_R$ and a neighborhood $U \subset \mathbb D_R$ of $w_0$,
    \item[$\ast$] holomorphic functions $h_1,\dots,h_{n-1}:U\to \mathbb D_R$, such that the map
\[
w \longmapsto (w,h_1(w),\dots,h_{n-1}(w))
\]
parametrizes a local branch of $\mathcal C$ over $U$,
    \item[$\ast$] an increasing sequence of integers $0 \le n_1 < n_2 < \cdots$,
    \item[$\ast$] and loops $\gamma_j:\Ss^1\to D_{n_j}$ based at $w_0$ for all $j \ge 1$ given by a composition of the form
    \[
    p_j \cdot \ell_j \cdot p_j^{-1},
    \]
    where $\ell_j$ is a small loop around some $z_{n_j} \in E_{n_j, c}$ and $p_j$ is a path from $w_0$ to some $w'_j \in \ell_j$,
\end{itemize}
for which the following hold:

1. (\textbf{Trivial monodromy of the coordinates}) For each $i=1,\dots,n-1$, analytic continuation of $h_i$ along $\gamma_j$ returns to the same germ at $w_0$ (equivalently, the monodromy action on $h_i$ along $\gamma_j$ is trivial). In particular, we assume that the germ of $h_i$ obtained at $w'_j$ by continuing along $p_j$ is defined on an open neighborhood of $w'_j$ containing the loop $\ell_j$.

2. (\textbf{Continuations remain in $\D \setminus \U$}) For all $s\in \Ss^1$ and all $i=1,\dots,n-1$, one has
\[
h_i(\gamma_j(s)) \in \mathbb D \setminus \U
\]

3. (\textbf{$\Psi$-monodromy of increasing order}) Analytic continuation of $\Psi$ along $\gamma_j$ sends $\Psi(z)$ to $\Psi(\zeta_{n_j}z)$
for some root of unity $\zeta_{n_j}\in\mu_\infty$ of order at least $2^{\,n_j}$.

\end{definition}
With this definition in hand, we can now show that—after permuting the factors and applying a suitable shift—we may replace a given curve by one that has the increasing-order 
$\Psi$-monodromy property.
\begin{lemma}
\label{lem:permute+rotate-for-monodromy}
Let $\CC$ be an irreducible curve in $\C^n$ intersecting $\D_R^n$ and let $d \ge 2$ be a given integer. Moreover, assume that the projections to the $i$-th coordinate $\pi_i: \CC \lra \bP^1$ are dominant. Then, there exists a permutation $\sigma \in S_n$ and a root of unity $\lambda \in \mu_{d^m}$ for some $m \ge 1$ such that the image $\CC'$ of $\CC$ under the automorphism
\[
(x_1,x_2, \dots,x_{n}) \mapsto (\lambda x_{\sigma(1)}, x_{\sigma(2)}, \dots, x_{\sigma(n)}) 
\]
has the increasing-order $\Psi$-monodromy property as defined in Definition \ref{def:increasing-psi-monodromy}. 
\end{lemma}
\begin{proof}
Let $\pi_i:\CC \lra \bP^1$ denote the projection onto the $i$-th coordinate, and let $S_i$ be the set of branch points of $\pi_i$. Let
\[
\U_\CC := \bigcup_{\zeta \in \mu_\infty}\bigcup_{i \ge 1} \pi_i^{-1}(\zeta \cdot \U) \: \text{ and } \: S_\CC:= \bigcup_{\zeta \in \mu_\infty}\bigcup_{1 \le i \le n}\bigcup_{1 \le j \le n} \pi_i^{-1}(\zeta \cdot \pi_i(S_j)),
\]
and $\CC^\ast := \CC \setminus (\U_\CC \cup S_\CC)$. Note that $\U_\CC$ and $S_\CC$ are countable sets and thus $\CC^\ast$ remains path-connected. In fact, there are uncountably many disjoint paths connecting any two points. Let $\text{x} := (x_0,\dots,x_{n-1})$ be a point in $\CC^\ast \cap \D_R^n$. Choose a path $\gamma$ in $\CC^\ast$ connecting it to a point $\text{y} := (y_0,\dots,y_{n-1}) \in \CC^\ast \cap (\D^n)^c$. After replacing this path with a subpath, we may assume without loss of generality that $\text{y}$ is the first time that $\gamma$ intersects $\CC^\ast \cap (\D^n)^c$. After permuting the coordinates if necessary, we must have 
\begin{equation}
\label{eqn:pts-on-bndry}
y_0,\dots,y_k \in \partial\D \text{ and } y_{k+1}, \dots, y_{n-1} \in \D,
\end{equation}
for some $k \le n - 1$. Note that we can take $k$ to be minimal so that if $\tilde{\gamma}$ is another path in $\CC^\ast$ satisfying 
\begin{itemize}
    \item[$\ast$] $\tilde{\gamma(0)} \in \D_R^n$
    \item[$\ast$] $\tilde{\gamma(s)} \in \D^n$ for all $ 0 \le s < 1$, and 
    \item[$\ast$] $\tilde{\gamma}(1)$ has $k'$ coordinates in $\partial\D$ and $n - k'$ coordinates in $\D$,
\end{itemize}
then $k \le k'$.

Let $\gamma' := \pi_1(\gamma)$. Thus, $\gamma'(0) = x_0$ and $\gamma'(1) = y_0$. We see that $\CC$ can be parametrized near $x_0,\dots,x_{n-1}$ as 
\[
(w,h_1(w),\dots,h_{n-1}(w)),
\]
and by the definition of $\gamma'$ we see that continuing each $h_i$ along $\gamma'$ we must have $h_i(\gamma'(s)) \in \D$ and $h_i(\gamma'(1)) = y_i$ for all $0 \le s < 1$ and $1 \le i \le n - 1$.

\begin{claim}
\label{claim:delta-nbhds}
There exists $\delta > 0$ such that the germ of $h_i$ at $\gamma'(s)$ obtained by continuing along $\gamma'$ is defined on $B_\delta(\gamma'(s))$ and satisfies
\[
h_i(B_\delta(\gamma'(s)) \cap \D) \subset \D,
\]
for all $s \in I$ and all $1 \le i \le n - 1$. 
\end{claim}
\begin{proof}
Choose $\epsilon > 0$ such that all $h_i$ are defined on a ball $B_\epsilon(y_0)$. If $\epsilon > 0$ is small enough, we may assume that $h_i(B_\epsilon(y_0)) \subset \D$ for all $i \ge k + 1$ using equation \eqref{eqn:pts-on-bndry}. Next, we want to establish that 
\begin{equation}
\label{eqn:h_i-D-to-D}
h_i(B_\epsilon(y_0) \cap \D) \subset \D,
\end{equation}
for all $1 \le i \le n-1$.
Suppose for the sake of contradiction that 
\[
h_i(z') \in \D^c
\]
for some $z' \in B_\epsilon(y_0) \cap \D$ and some $1 \le i \le n - 1$. Note that there exists $s' < 1$ such that $\gamma'(s) \in B_\epsilon(y_0)$ for all $s > s'$. Let $\gamma''$ be a path obtained by branching from $\gamma$ at $\gamma'(s_0)$ for some $s_0 > s$ and taking a simple path from $\gamma'(s_0)$ to $z'$ as shown in Figure \ref{fig:branched-path}. After replacing $\gamma''$ by a subpath if necessary, we may assume that, upon analytically continuing the functions $h_i$ along $\gamma''$, the point $\gamma''(1)$ is the first value for which
\[
\bigl(\gamma''(s), h_1(\gamma''(s)), \dots, h_{n-1}(\gamma''(s))\bigr)
\]
leaves $\mathbb{D}^n$. Hence, after possibly permuting the coordinates, we have reduced $k$ by at least one: the first coordinate now also remains in $\mathbb{D}$, together with the coordinates $k+1,\dots,n$. This contradicts the minimality of $k$. Hence, Equation \ref{eqn:h_i-D-to-D} must hold for all $1 \le i \le n-1$.

\begin{figure}[t]
    \centering
\includegraphics[width=0.5\linewidth]{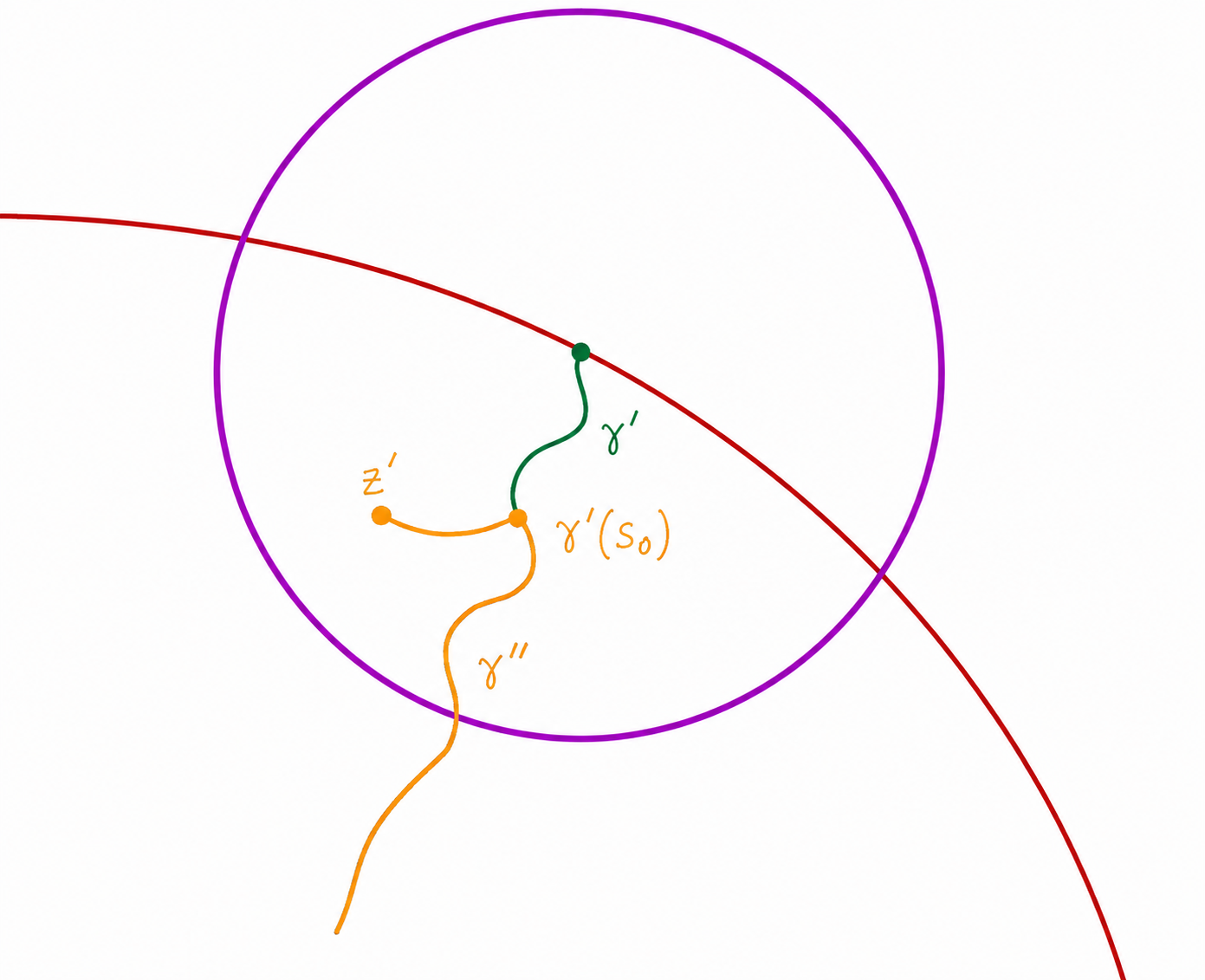}
    \caption{The branched path obtained by taking a turn at $\gamma'(s_0)$.}
    \label{fig:branched-path}
\end{figure}

To finish the proof, fix $s_0>s'$. Since the analytic continuation of each $h_i$ along $\gamma|_{[0,s_0]}$ stays in $\D$, we may choose $\delta>0$ so small that for every $s\le s_0$,
\[
B_\delta(\gamma'(s))\subset \D,
\]
the continued germ of $h_i$ at $\gamma'(s)$ is holomorphic on $B_\delta(\gamma'(s))$, and
\[
h_i\bigl(B_\delta(\gamma'(s))\bigr)\subset \D
\qquad\text{for all }1\le i\le n-1.
\]
This proves the claim for all $s\le s_0$.

Shrinking $\delta$ further if necessary, we also arrange that
\[
B_\delta(\gamma'(s))\subset B_\epsilon(y_0)
\qquad\text{for all } s\ge s_0.
\]
Then, equation \eqref{eqn:h_i-D-to-D} also yields the desired conclusion for $s\ge s_0$, completing the proof.

\end{proof}

\begin{figure}
    \centering
    \includegraphics[width=0.5\linewidth]{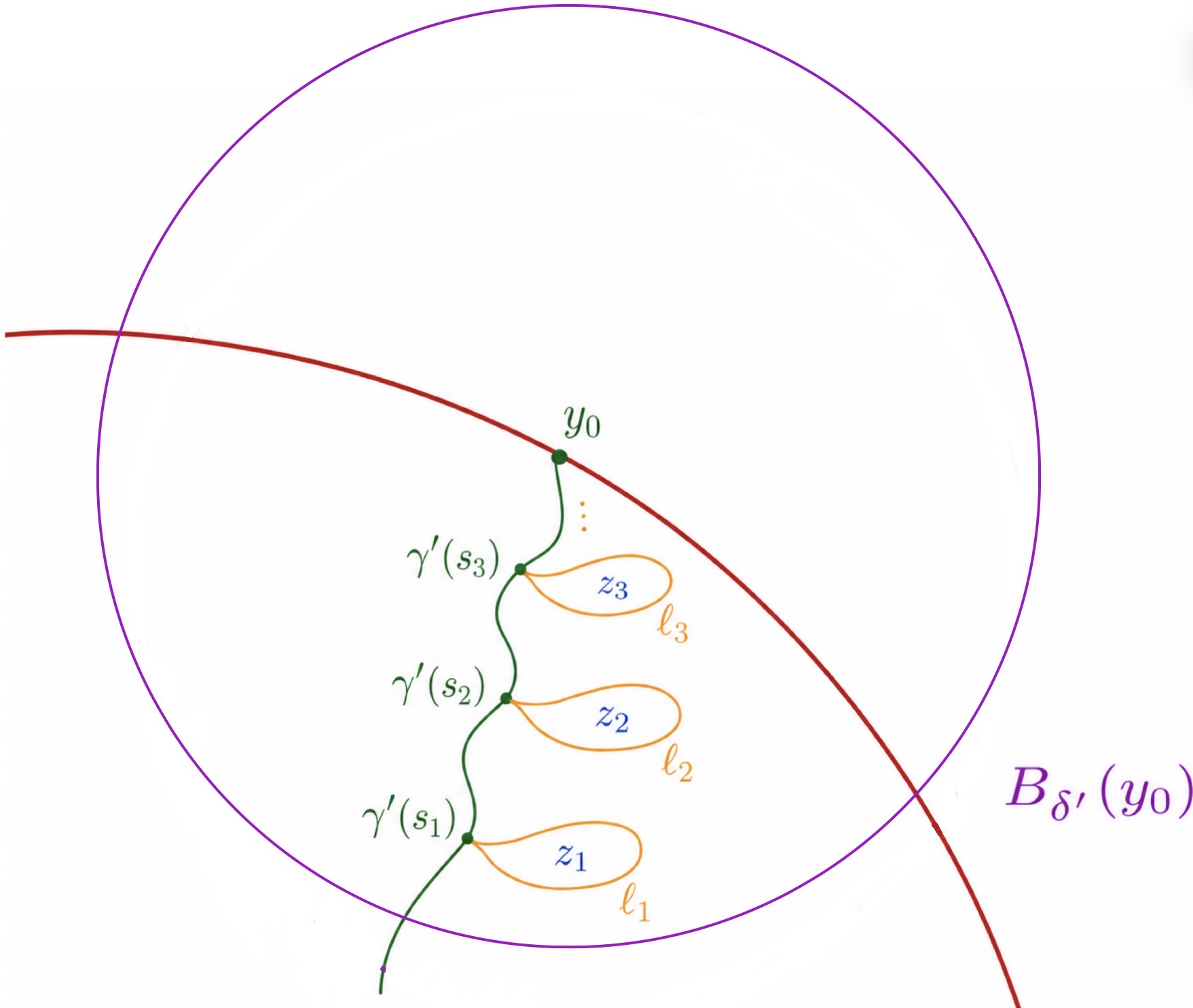}
    \caption{The loops $\ell_j$ around the points $z_j \in E_{j, c}$ inside $B_{\delta'}(y_0)$.}
    \label{fig:loops-around-z_j}
\end{figure}
Let $\delta$ be as in Claim~\ref{claim:delta-nbhds}. Fix $\delta' < \delta/2026$ and choose $c\in C_{f}$. 
Pick a sequence of points $z_j\in E_{j,c}$ with
\[
z_j\in B_{\delta'}(y_0)\qquad\text{for all }j\ge j_0,
\]
for some $j_0 \ge 1$. For each $j \ge j_0$ choose $s_j\in I$ such that $\gamma'(s_j)\in B_{\delta'}(y_0)$, and let $\ell_j$ be a small loop around $z_j$ (as shown by Figure \ref{fig:loops-around-z_j})contained in
\[
B_{\delta'}(y_0)\setminus \pi_1(\U_{\mathcal C}).
\]
Such a loop exists because $\pi_1(\U_{\mathcal C})$ is countable, while there are uncountably many loops in $B_{\delta'}(y_0)$ homotopic to a small circle about $z_j$; hence we can choose one avoiding $\pi_1(\U_{\mathcal C})$.

By Lemma \ref{lem:loop-monodromy-around-E_n} and Proposition ~\ref{prop:path-for-non-triv-monodromy}, for each $j\ge 1$ there exists $\zeta_j\in \mu_{d^j}$ such that analytic continuation of $\Psi$ along the loop
\begin{equation}\label{eqn:loops-gamma-j}
\gamma_j \;:=\; \zeta_j\cdot\Bigl(\gamma'|_{[0,s_j]}\cdot \ell_j \cdot (\gamma'|_{[0,s_j]})^{-1}\Bigr)
\end{equation}
sends the germ $\Psi(z)$ to $\Psi(\zeta'_j z)$ for some $\zeta'_j\in \mu_{d^j}$ with
\[
\ord(\zeta'_j)\ge 2^j.
\]
This means that the lopps $\gamma_j$ satisfy property 3 required in Definition \ref{def:increasing-psi-monodromy}. We now show that the first two properties can also be satisfied simultaneously. 

Since $\{\zeta_j\}\subset \Ss^1$ is an infinite sequence, it has an accumulation point $\mu\in \Ss^1$. Passing to a subsequence, we may assume $\zeta_j\to \mu$ as $j\to\infty$. 
Choose a $d^m$-th root of unity $\lambda$ for some $m \ge 1$ so that
\[
|\arg(\mu/\lambda)|<\epsilon_1
\]
for some $\epsilon_1>0$. Replacing $\mathcal C$ by its image under
\[
(t_1,t_2,\dots,t_n)\longmapsto (\lambda t_1,t_2,\dots,t_n),
\]
we may assume $\lambda=1$.

After discarding finitely many terms and choosing $\epsilon_1$ sufficiently small, we may arrange that for all $j\ge 1$,
\[
\zeta_j\,\gamma'(s)\in B_{\delta'}(\gamma'(s))\quad\text{for all }s\in I,
\qquad\text{and}\qquad
\zeta_j\cdot \ell_j \subset B_{\delta}(\gamma'(s_j)).
\]
Claim~\ref{claim:delta-nbhds} then implies that, for each $i$, the germ of $h_i$ at $\zeta_j\gamma'(s)$ obtained by continuing $h_i$ along $\zeta_j\gamma'|_{[0,s]}$ agrees with the restriction to $\zeta_j\gamma'(s)$ of the continuation of $h_i$ along $\gamma'|_{[0,s]}$, which is holomorphic on $B_\delta(\gamma'(s))$. So, further continuing $h_i$ along $\zeta_j \cdot \ell_j$ acts trivially on $h_j$. Consequently, the loops $\gamma_j$ satisfy property 2 required in Definition~\ref{def:increasing-psi-monodromy}. The fact that proporty 3 is also satisfied is a consequence of the definition of $\mathcal{U}_\CC$. 

\end{proof}

We are now ready to prove Theorem \ref{thm:DAS-totally-disconnected} in the curve case. 
\begin{proposition}
\label{prop:DAS-curve}
Let $\CC \subset \C^n$ be an algebraic curve and suppose that 
\[
(h_0(w), h_1(w),\dots,h_{n-1}(w)),
\]
is a branch of $\CC$ where $h_0,\dots,h_{n-1}$ are analytic functions defined on some small disk $U$ having the property that $h_i(U) \subset \D_R$ for all $0 \le i \le n-1$. Let 
\[
Z = \{(h_0(w),h_1(w),\dots,h_{n-1}(w),\Psi(h_0(w)),\Psi(h_1(w)), \dots,\Psi(h_{n-1}(w))): w \in U\}
\]
and assume that the branch
\[
(\Psi_f(h_0(x)), \dots,\Psi_f(h_{n-1}(x)))
\]
is not contained in a proper $f$-special subvariety of $\C^n$.
Then
\[
\dim_{\mathbb{C}}\bigl(\overline{Z}^{\zar}\bigr)=n+1,
\]
which is equivalent to the functions
\[
h_0(w), \Psi(h_0(w)), \Psi(h_1(w)),\dots,\Psi(h_{n-1}(w))
\]
being algebraically independent over $\mathbb{C}$.
\end{proposition}
\begin{proof}
We prove this by induction on $n$. The base case is $n=1$ which is clear since $\Psi(z)$ is a transcendental function by \cite[Theorem 1]{Bottcher-transcendence}. 

Now assume that $n \ge 2$. If any of the projections $\pi_i: \CC \lra \bP^1$ is not dominant, we see that $\CC$ is contained in a fiber of one of the $\pi_i$ which is $f$-special by definition and we are done. So, we may assume that the projection $\pi_i$ are all dominant.

It is clear that if $\sigma \in S_n$ is a permutation and Proposition \ref{prop:DAS-curve} holds for the curve $\CC^\sigma$ obtained by permuting the analytic function $h_0,\dots,h_{n-1}$, then it must also hold for $\CC$. Moreover, by Proposition \ref{prop:rotation-red}, it suffices to prove Proposition \ref{prop:DAS-curve} holds for $(\lambda h_0,h_1,\dots,h_{n-1})$ where $\lambda$ is some root of unity in $\mu_{d^{\infty}}$. Thus, by Lemma \ref{lem:permute+rotate-for-monodromy} we can assume that $\CC$ has the increasing-order $\Psi$-monodromy property. After replacing $w$ with $h_0^{-1}(w)$ we also assume without loss of generality that $h_0(w) = w$ and that $U \subset \D_R$. By the increasing-order $\Psi$-monodromy of $\CC$, we can find loops $\{\gamma_j\}$ satisfying conditions 1,2, and 3 given in Definition \ref{def:increasing-psi-monodromy}. Each $\gamma_j$ is a loop around a point $z_{n_j} \in E_{n_j}$ for some sequence $\{n_j\}_{j \ge 1}$ of positive integers. 

Now, if the functions 
\[
w, \Psi(w), \Psi(h_1(w)), \dots, \Psi(h_{n-1}(w)),
\]
are algebraically independent we are done. So, suppose that there exists an equation $Q \in \C[X_1,\dots,X_{n+1}]$ such that 
\begin{equation}
\label{eqn:alg-rel-2}
Q(w,\Psi(w),\Psi(h_1(w)), \dots,\Psi(h_{n-1}(w))) = 0,
\end{equation}
for all $w \in U$.

By the definition of $\gamma_j$, continuing $h_i(w)$ along $\gamma_j$ results in $h_i$-values along a loop that we will call $\gamma_{j,i}$ in $\D \setminus \U$. By compactness, all of the loops 
\[
\gamma_j,\gamma_{j,1},\dots,\gamma_{j,n-1},
\]
remain in $D_{m_j}$ for some $m_j \in \N$. Therefore, we can use the equation 
\[
\Psi(z^{m_j}) = f^{m_j}(\Psi(z))
\]
to continue $(\Psi(w),\Psi(h_0(w)), \dots, \Psi(h_{n-1(w)}))$ along $\gamma_j$ for all $j \ge 1$. This continuation sends $\Psi(w)$ to $\Psi(\zeta_{n_j}w)$ where $\zeta_{n_j}$ are as defined in Definition \ref{def:increasing-psi-monodromy}. Note that continuing $\Psi(h_i(w))$ along $\gamma_j$ is the same as continuing $\Psi(w)$ along $\gamma_{j,i}$. So, by Proposition \ref{prop:loop-monodromy}, continuing along $\gamma_j$ sends $\Psi(h_i(w))$ to $\Psi(\zeta_{i,j}h_i(w))$ for some root of unity $\zeta_{i,j} \in \mu_{d^{m_j}}$ and all $j \ge 1$ and $1\le  i \le n - 1$.

Suppose without loss of generality that the pairs $(w,h_i(w))$ are not branches of a translate of an algebraic subgroup by a torsion element for all $k \le i \le n-1$ and some $k \ge 1$ ($k \ge n$ implies that there are no such pairs). Recall that each $\gamma_j$ is given by
    \[
    p_j \cdot \ell_j \cdot p_j^{-1},
    \]
where $\ell_j$ is a small loop around some $z_{n_j} \in E_{n_j, c}$ for some $c \in C_{f}$ and $p_j$ is a path from $w_0$ to some $w'_j \in \ell_j$. Let $h_i(z_{n_j})$ denote the value at $z_{n_j}$ of the branch $h_i$ obtained by analytic continuation along $p_j$. (The germ is defined on a neighborhood containing $\ell_j$, and hence also contains $z_{n_j}$.) We can prove the next claim.
\begin{claim}
\label{claim:h-sends-to-E-comp}
At the expense of replacing $\{n_j\}_{j \in \N}$ with $\{n_{j + s}\}_{j \in \N}$ for some $s \ge 0$, we may assume without loss of generality that for all $k \le i \le n - 1$ and all $j \ge 1$, $h_i(z_{n_j}) \notin E_m$ for any $m \ge 1$.
\end{claim}
\begin{proof}
If $k \ge n$ there is nothing to show. So, assume $k \le n-1$. Suppose that the claim does not hold. Then, for every $j \ge 1$ there is $k \le i \le n-1$ such that $h_i(z_{n_j}) \in E_m$ for some $m \ge 1$. By the pigeonhole principle, we may assume that $i$ is fixed after replacing $\{n_j\}$ with a subsequence. By replacing with another subsequence we may also assume that $h_i(z_{n_j}) \in E_{m,c_1}$ for some $c_1 \in C_{f}$. It then follows from Lemma \ref{lem:many-preimages-then-special} that the algebraic curve defined by $(z,h_i(z))$ is a translate of an algebraic subgroup of $\bG_m^2$ by a torsion element. This contradicts the assumption on the pair $(z,h_i(z))$ and finishes the proof. 
\end{proof}

As a result of Claim \ref{claim:h-sends-to-E-comp} we see that the loops $\gamma_{i,j}$ do not go around any points in $E_m$ for $i \ge k$. Hence, continuing $\Psi$ along $\gamma_{i,j}$ has trivial monodromy and we may assume that 
\[
\zeta_{i,j} = 1 \text{ for all } i \ge k \text{ and } j \ge 1.
\]
Thus, continuing $\Psi(w),\Psi(h_0(w)),\dots,\Psi(h_{n-1}(w))$ along $\gamma_j$ and using Equation \ref{eqn:alg-rel-2} we get
\begin{equation}
\label{eqn:after-continuation}
Q\left(w, \Psi(\zeta_{j,0}w), \Psi(\zeta_{j,1}h_1(w)), \dots, \Psi(\zeta_{j,{k-1}}h_{k-1}(w)), \Psi(h_{k}(w)),\dots,\Psi(h_{n-1}(w))\right) = 0,
\end{equation}
for all $j \ge 1$ where we write $\zeta_{j,0}$ in place of $\zeta_{n_j}$ to keep the notation consistent. 

Suppose $k = 1$. Since the order of $\zeta_{j,0}$ is at least $2^{n_j}$, we have infinitely many different values of $\zeta_{n_j}$. We conclude that $Q$ has no dependence on the second variable. So,  
\[
Q\left(w, \Psi(h_1(w)), \dots, \Psi(h_{n-1}(w))\right) = 0,
\]
for all $w \in U$. Note that $w$ and $h_1(w)$ are algebraically dependent. So, there is another polynomial $Q_1 \in \C[X_1,\dots,X_n]$ such that 
\[
Q_1\left(h_1(w), \Psi(h_1(w)), \dots, \Psi(h_{n-1}(w))\right) = 0
\]
By the inductive hypothesis we conclude that $(\Psi(h_1(w)), \dots,\Psi(h_{k-1}(w)))$ is contained in a proper $f$-special subvariety of $\C^n$, say $W$. Thus, $(\Psi(w), \Psi(h_1(w)), \dots,\Psi(h_{n-1}(w)))$ is contained in $\pi_1^{-1}(W)$ which is also $f$-special and we are done. 

Now, suppose that $k \ge 2$. The set 
\begin{align}
\mathcal{S} := \{(s_0,\dots,s_{k-1})&: Q\left(w, \Psi(s_0w), \dots, \Psi(s_{k-1}h_{k-1}(w)), \Psi(h_{k}(w)),\dots,\Psi(h_{n-1}(w))\right) = 0, \notag \\
&\text{ for all } w \in U\}, \notag
\end{align}
is an analytic set. We already know that it contains $(\zeta_{j,0}, \dots,\zeta_{j,k-1})$ for all $j \ge 1$. Hence, $\dim(\mathcal{S}) \ge 1$. 

Let $\lambda := (\lambda_0,\dots,\lambda_{k-1}) \in (\Ss^1)^k$ be an accumulation point of the points $(\zeta_{j,0}, \dots,\zeta_{j,k-1})$ and let $G$ be the analytic germ of $\mathcal{S}$ near $\lambda$. After replacing with a suitable subsequence of tuples $(\zeta_{j,0}, \dots,\zeta_{j,k-1})$, we can assume that $G$ is an irreducible germ and contains all the points $(\zeta_{j,0}, \dots,\zeta_{j,k-1})$. Using equation \eqref{eqn:after-continuation} we get
\begin{equation}
\label{eqn:after-closure}
Q\left(w, \Psi(g_0w), \Psi(g_1h_1(w)), \dots, \Psi(g_{k-1}h_{k-1}(w)), \Psi(h_{k}(w)),\dots,\Psi(h_{n-1}(w))\right) = 0,
\end{equation}
for every $(g_0,\dots,g_{k-1}) \in G \cap N$ where $N$ is a sufficiently small neighborhood of $\lambda$ and all $w \in U$. 

Suppose that there exists $1 \le \ell \le k - 1$ such that for all but countably many $(g_0,\dots,g_{k-1}) \in G \cap N$, we can guarantee that none of the branches
\[
(g_0w, g_\ell h_{\ell}(w)), (g_0w, h_{k}(w)), \dots, (g_0w, h_{n-1}(w)),
\]
 is a torsion coset of an algebraic subgroup of $\bG_m^2$. Then, we can reduce $k$ by 1 and conclude inductively that 
\[
\left(\Psi(g_0w), \Psi(g_1h_1(w)), \dots, \Psi(g_{k-1}h_{k-1}(w)), \Psi(h_{k}(w)),\dots,\Psi(h_{n-1}(w)) \right)
\]
is contained in an $f$-special subvariety $V_{g_0,\dots,g_{k-1}}$. It is clear by our assumption on $h_1,\dots,h_{n-1}$ that $V_{g_0,\dots,g_{k-1}}$ cannot be contained in a fiber of any of the projections $\pi_i$. So, $V_{g_0,\dots,g_{k-1}}$ is an $F_{n}$-preperiodic subvariety of $\C^n$. There are countably many such varieties by \cite[Theorem 6.24]{Scanlon}. So, by the Baire category theorem, there must exist an $F_n$-preperiodic subvariety, say $T$, such that $T = V_{g_0,\dots,g_{k-1}}$ for all $(g_0,\dots,g_{k-1})$ in an analytically dense subset of $N \cap G$. It follows that 
\[
\Psi(g_1h_1(w)), \dots, \Psi(g_{k-1}h_{k-1}(w)), \Psi(h_{k}(w)),\dots,\Psi(h_{n-1}(w))
\]
is contained in $T$ for all $(g_0,\dots,g_{k-1}) \in N \cap G$. In particular, the branch $C_1$ defined by
\begin{equation}
\label{eqn:branch-2}
\left(\Psi(\zeta_{1,0}w),\Psi(\zeta_{1,1}h_1(w)), \dots, \Psi(\zeta_{1,k-1}h_{k-1}(w)), \Psi(h_{k}(w)),\dots,\Psi(h_{n-1}(w))\right),
\end{equation}
is contained in $T$. In other words, Proposition \ref{prop:DAS-curve} is vacuously true for the curve 
\begin{align}
( \zeta_{1,0}w &,\zeta_{1,1}h_1(w), \dots, \zeta_{1,k-1}h_{k-1}(w),h_{k}(w),\dots,h_{n-1}(w), \dots \notag \\
&,\Psi(\zeta_{1,0}w),\Psi(\zeta_{1,1}h_1(w)), \dots, \Psi(\zeta_{1,k-1}h_{k-1}(w)), \Psi(h_{k}(w)),\dots,\Psi(h_{n-1}(w)))
\end{align}
We can then conclude by Proposition \ref{prop:rotation-red} that the Proposition must also hold for the original branch 
\[
(w,h_1(w),\dots,h_{n-1}(w),\Psi(w),\Psi(h_1(w)),\dots,\Psi(h_{n-1}(w))).
\]

Finally, suppose that we cannot guarantee the existence of some $1 \le \ell \le k - 1$ such that for all but countably many $(g_0,\dots,g_{k-1}) \in G \cap N$, none of the branches
\[
(g_0w, g_\ell h_{\ell}(w)), (g_0w, h_{k}(w)), \dots, (g_0w, h_{n-1}(w)),
\]
 is a torsion coset of an algebraic subgroup of $\bG_m^2$. Note that there are only countably many values of $g_0$ that can result in one of $(g_0w,h_i(w))$ being a torsion coset of an algeraic subgroup for some $k\le i \le n-1$. We conclude that for each $1\le \ell \le k - 1$ and all but countably many $(g_0,\dots,g_{k-1}) \in G \cap N$, the branch $(g_0w, g_\ell h_{\ell}(w))$ is a torsion coset of an algebraic subgroup. Again, note that there are countably many such subtori. So, $(g_0w, g_\ell h_\ell(w))$ is a branch of a fixed algebraic curve given by $(g_0w, \tilde{h}_\ell(g_0w))$ for all $1 \le l \le k - 1$. Equation \eqref{eqn:after-closure} then becomes
\[
Q\left(w, \Psi(g_0w), \Psi(\tilde{h}_1(g_0w)), \dots, \Psi(\tilde{h}_{k-1}(g_0w)), \Psi(h_{k}(w)),\dots,\Psi(h_{n-1}(w))\right) = 0,
\]
for all $w \in U$ and all $g_0$ in some small neighborhood of $\lambda_0$. We can then change coordinates by $u = g_0w$ and get 
\[
Q\left(w, \Psi(u), \Psi(\tilde{h}_1(u)), \dots, \Psi(\tilde{h}_{k-1}(u)), \Psi(h_{k}(w)),\dots,\Psi(h_{n-1}(w))\right) = 0,
\]
for all $w \in U$ and all $u$ in a small neighborhood containing $g_0U$. Since $u$ and $w$ are algebraically independent variables we either must have an algebraic relation 
\begin{equation}
\label{eqn:u-rel}
Q_1(\Psi(u), \Psi(\tilde{h}_1(u)), \dots, \Psi(\tilde{h}_{k-1}(u))) = 0,
\end{equation}
or 
\begin{equation}
\label{eqn:w-rel}
Q_2\left(w, \Psi(h_{k}(w)),\dots,\Psi(h_{n-1}(w))\right) = 0.
\end{equation}
In the latter case, since $h_k(w)$ is algebraically dependent with respect to $w$ we conclude that 
\[
Q_3\left(h_k(w), \Psi(h_{k}(w)),\dots,\Psi(h_{n-1}(w))\right) = 0
\]
for some $Q_3 \in \C[x_1,\dots,x_{n-k + 1}]^\ast$. We conclude the proof by the inductive hypothesis. 

Finally, suppose equation \eqref{eqn:u-rel} holds. Let $W$ be the Zariski closure of the branch
\[
(\Psi(u), \Psi(\tilde{h}_1(u)), \dots, \Psi(\tilde{h}_{k-1}(u))),
\]
in $\C^k$ which must be irreducible. Since $(u,\tilde{h}_i(u))$ is a branch of a coset of a subtori by a torsion element there must exist roots of unity $\mu_i$ and integers $a_i,b_i \in \Z$ such that $(z,\tilde{h}_i(z))$ is parametrized as $(t^{a_i},\mu_it^{b_i})$ for $t \in \C$. We can choose integers $0 \le N < M$ such that $\mu_i^{d^N} = \mu_i^{d^M}$ for all $1 \le i \le k - 1$. This means that the branches 
\[
(u^{d^N},\tilde{h}_i(u)^{d^N}) = (t^{a_id^N}, \mu_i^{d^N}t^{b_id^N}), \text{ and } (u^{d^M},\tilde{h}_i(u)^{d^M}) = (t^{a_id^M}, \mu_i^{d^N}t^{b_id^M})
\]
define the same algebraic curve for all $1 \le i \le k - 1$. It follows that $$(F_n)^N(\Psi(u), \Psi(\tilde{h}_1(u)), \dots, \Psi(\tilde{h}_{k-1}(u))) \text{ and } F_n^M(\Psi(u), \Psi(\tilde{h}_1(u)), \dots, \Psi(\tilde{h}_{k-1}(u))) $$ must define the same analytic branches. So, these two branches have the same Zariski closure. On the other hand, the Zariski closure of these branches are equal to $F_n^N(W)$ and $F_n^M(W)$, respectively. So, $W$ must be preperiodic and we are done. 
\end{proof}

The proof of Theorem~\ref{thm:DAS-totally-disconnected}, in the case where
$h_0,\dots,h_{n-1}$ parametrize a branch of an algebraic variety of dimension
at least $2$, relies on a slicing argument that reduces the dimension and
allows us to proceed by induction. A crucial point in the argument is to show
that a generic slice still maps dominantly onto each coordinate. 
\begin{lemma}
\label{lem:dim-red-via-fibration}
Suppose $n \ge 3$, $1 \le i \le n$, and $X$ is an irreducible subvariety of $(\bP^1)^n$ of dimension at least $2$ that is not contained in finitely many fibers of the projections $\pi_i$ for $1 \le i \le n$. Then, there exists a dominant rational fibration $g: (\bP^1)^n \dra \bP^1$ such that for all but finitely many $a \in \bP^1(\C)$ the restriction of $\pi_i$ to $X \cap g^{-1}(a)$ is dominant. 
\end{lemma}
\begin{proof}
Note that since $X$ is not contained in a finite union of fibers we must have that the intersection of $X$ with $\C^n \subset (\bP^1)^n$ is Zariski dense in $X$. From now, we abuse notation and refer to $X \cap \C^n$ as $X$ as well.   We will show that we can take $g$ to be a fibration $g_{a_1,\dots,a_n}:\C^n \lra \C$  of the form
\[
(x_1,\dots,x_n) \mapsto a_1x_1 + \cdots + a_nx_n,
\]
for $(x_1,\dots,x_n) \in \C^n$.

Let \(K=\mathbb C(X)\) be the function field of $X$. Write \(x_i\in K\) for the coordinate functions. Since
\(X\subset \mathbb C^n\) is irreducible of dimension at least $2$, we must have $\trdeg_\C(K) \ge 2$. Moreover, the assumption that \(X\) projects dominantly to every coordinate
means that each \(x_i\) is nonconstant.

For each \(i\), let
\[
L_i=\overline{\mathbb C(x_i)}\cap K.
\]
Consider the
finite-dimensional \(\mathbb C\)-vector space
\[
V=\operatorname{span}_{\mathbb C}\{x_1,\dots,x_n\}\subset K.
\]
For fixed \(i\), define $B_i = V \cap L_i$.
Each $B_i$ is a $\C$-subspace of $V$ satisfying \(B_i\subsetneq V\). Indeed, if \(B_i=V\), then every coordinate
function \(x_j\) would be algebraic over \(\mathbb C(x_i)\). Since the
coordinate functions \(x_1,\dots,x_n\) generate \(K\) as a function field, this
would imply that \(K\) is algebraic over \(\mathbb C(x_i)\). Hence, $\trdeg_\C K = 1$, which contradicts that $X$ has dimension at least 2.

Thus, each \(B_i\) is a proper \(\mathbb C\)-linear subspace of \(V\). Therefore, $\bigcup_{i=1}^n B_i$ does not cover $V$ and we may choose
\[
g=a_1x_1+\cdots+a_nx_n\in V\setminus \bigcup_{i=1}^n B_i.
\]
Then for every \(i\), we have \(g\notin L_i\). Equivalently, \(g\) is
transcendental over \(\mathbb C(x_i)\). Using the fact that $x_i$ is non-constant, we conclude that $g$ and $x_i$ are algebraically
independent over $\mathbb C$. Hence, the morphism
\[
(g,x_i):X\to \mathbb A^2
\]
is dominant for every \(i\). Hence, for every $i$, there is a Zariski open subset $Y_i \subset \C$ such that for every $a \in Y_i$, $X \cap g^{-1}(a)$ maps dominantly to the $x_i$ coordinate. It follows that for all 
\[
a \in \bigcap_{i=1}^n Y_i,
\]
$X \cap g^{-1}(a)$ projects dominantly to every coordinate.
\end{proof}

Finally we prove Theorem \ref{thm:DAS-totally-disconnected} in the case where $(h_0,\dots,h_{n-1})$ is a branch of an arbitrary algebraic subvariety of $\C^n$.
\begin{proof}[Proof of Theorem \ref{thm:DAS-totally-disconnected}]
We prove this by induction on $m$. The base case is $m = 1$ which is done by Proposition \ref{prop:DAS-curve}.

Suppose that the theorem is shown for $r \le m$. Now let $r = m + 1$ and suppose that $V$ is the algebraic closure of $Z$ in $\C^n \times \C^n$. By the hypothesis, we know that $\pi_{1,\dots,n}(Z)$ is a branch of an irreducible algebraic subvariety $Z_{disk}$ of $\C^n$.

If $Z_{disk}$ is contained in a fiber of a projection $\pi_i$ for some $1 \le i \le n$, then $\pi_{n+1,\dots,2n}(Z)$ is also contained in a fiber of $\pi_i$ which is a proper $f$-special subvariety of $\C^n$ and we are done. So, assume this is not the case. Then, by Lemma \ref{lem:dim-red-via-fibration} we can choose a fibration $g: \C^n \lra \C$ such that each $\pi_i$ restricted to $Z_{disk} \cap g^{-1}(a)$ is dominant for all but finitely many $a \in \C$. Let $G: \C^n \times \C^n \lra \C$ be defined by $G = g\circ \pi_{1,\dots,n}$.

Let $D$ be as in the statement of Conjecture \ref{conj:dyn-ax-schanuel} and let $\iota: D\lra \D_R^n$ be the map
\[
(x_1,\dots,x_r) \mapsto (h_0(x_1,\dots,x_r), \dots, h_{n-1}(x_1,\dots,x_r)),
\]
locally parameterizing $Z_{disk}$. We let $V' := g(\iota(V))$ which is a non-empty open subset of $\C$. 

For any given $a \in \C$ and any set $W \subset \D_R^n \times B_\infty(f)^n $, let $W_a$ denote the intersection $W \cap G^{-1}(a)$. By our assumption on $g$ we must have that 
\begin{equation}
\label{eqn:dim-fiber-Z}
\dim(Z_a) = \dim(Z) - 1,
\end{equation}
for all but finitely many $a \in V'$. Now, there are two possible cases that we will handle separately.
\medskip

\textbf{Case 1.} For all but finitely many $a\in V'$, at least one $(r - 1)$-dimensional branch of the analytic set $\pi_{n+1,\dots,2n}(Z_a)$ is contained in an $f$-special subvariety of $(\bP^1)^n$.

By our assumption on $g$, there are only finitely many $a\in\C$ such that some branch of $\pi_{n+1,\dots,2n}(Z_a)$ is contained in a fiber of a coordinate projection $\pi_i$ for some $1\le i\le n$. Discarding these finitely many points, we may therefore assume that for all but finitely many $a\in V'$, the set $\pi_{n+1,\dots,2n}(Z_a)$ is not contained in any such fiber. Hence, for all but finitely many $a\in V'$, there exists an $F_n$-preperiodic subvariety of $\C^n$ containing a branch of $\pi_B(Z_a)$.

By \cite[Theorem~6.24]{Scanlon}, there are only countably many $F_n$-preperiodic hypersurfaces of $\C^n$. Since $V'$ is uncountable, the pigeonhole principle implies that there exists an $F_n$-preperiodic hypersurface $T\subset \C^n$ and an uncountable subset $V''\subset V'$ such that, for every $a\in V''$, some $(r-1)$-dimensional branch of $\pi_{n+1,\dots,2n}(Z_a)$ is contained in $T$. This shows that $\pi_{n+1,\dots,2n}(Z)$ must also be contained in $T$ and we are done.

\medskip
\textbf{Case 2.} For infinitely many $a\in V'$, no $(r-1)$-dimensional branch of $\pi_{n+1,\dots,2n}(Z_a)$ is contained in an $f$-special subvariety of $\C^n$.

In this case, the inductive hypothesis together with \eqref{eqn:dim-fiber-Z} yields, for infinitely many $a\in\C$,
\[
\dim(V_a)\ \ge\ \dim(Z^{\zar}_a)\ \ge\ \dim(Z_a)+n
 = n + r - 1.
\]
In particular, the generic fiber of $G|_V$ has dimension at least $r+n-1$. Hence
\[
\dim(V) \ge \bigl(\dim(Z)+n-1\bigr)+1\\ = r+n,
\]
and we are done.
\end{proof}

\section{Proof of Theorem \ref{thm:classification-of-bialgebraic}}
\label{sec:pf-of-bialg-classification}

In this section we prove Theorem \ref{thm:classification-of-bialgebraic}, which classifies bialgebraic sets under the assumption that $J_f$ is either disconnected or has a non-degenerate locally connected model. We break the proof into a series of Propositions. The first Proposition handles the case of bialgebraic curves. 

\begin{proposition}
\label{prop:bialg-curve-case}
Theorem \ref{thm:classification-of-bialgebraic} holds when $n=2$ and $V_1$ is an irreducible curve. 
\end{proposition}
\begin{proof}
When $J_f$ is disconnected the proposition is a consequence of Theorem \ref{thm:DAS-totally-disconnected} as follows. Suppose that $V_1$ is an $f$-bialgebraic curve. Then, there must an irreducible algebraic curve $V_2 \subset \C^2$ such that $\bpsi_{f,2}(V_1) = (\Psi_f,\Psi_f)(V_1) \subset V_2$. If $V_1$ is contained in a fiber of one of the projections $\pi_1$ and $\pi_2$, then so is $V_2$. So, $V_2$ must be $f$-special and we are done. Hence, we assume that $V_1$ projects dominantly to both coordinates. Let $(x, h(x))$ be some local chart of $V_1$ for some analytic function $h$ on an open subset $U \subset \D_R$ satisfying $h(U) \subset \D_R$. Then, 
\[
(x,h(x), \Psi(x), \Psi(h(x))),
\]
is contained in $V_1 \times V_2$ which has dimension 2 in $\C^4$. Therefore, by Theorem \ref{thm:DAS-totally-disconnected} we must have $(\Psi(x), \Psi(h(x))$ is contained in an $(f,f)$-preperiodic curve. This shows that $V_2$ is $(f,f)$-preperiodic and finishes the proof in the disconnected case. From now, we assume that $J_f$ is connected and has a non-degenerate locally connected model.  

Let $V_1$ and $V_2$ be as before and suppose they are defined by irreducible polynomials $P,Q \in \C[X,Y]$, respectively. Suppose that $(z,h(z))$ is a branch of the curve $V_1$ where $h$ is an analytic function defined on some open subset $U \subset \D$ such that $h(U) \subset \D$. Fix a point $z_0 \in U$.  Recall that $h$ is non-constant. This allows us to assume that $\frac{\partial P}{\partial X}(w,z)\frac{\partial P}{\partial Y}(w,z)$ does not vanish identically for $(w,z) \in V_1$. Similarly, we can assume that $\frac{\partial Q}{\partial X}(w,z)\frac{\partial Q}{\partial Y}(w,z)$ does not vanish identically for $(w,z) \in V_2$.

Recall the map $\iota = \iota_f$ defined in equation \eqref{eqn:def-of-iota}. The next claim establishes the existence of a suitable boundary point $w_0 \in \partial\mathbb{D}$, which will serve as the target endpoint for the analytic continuation of $h$ starting at $z_0$.
\begin{claim}
\label{claim:w_0-is-nice}
We can choose $w_0 \in \partial \D$ satisfying the following properties: 
\begin{itemize}
    \item[(a)] $\iota(w_0)$ is either a $J_\sim$-endpoint, an interval point, or a Jordan point; and \label{type-of-iota-w0}
    \item [(b)] For any $(w_0,z_0) \in V_1$, we have $\iota(z_0)$ is not a poly-accessible point of $J_\sim$; and \label{second-coord-not-poly-acc}
    \item[(c)] $\frac{\partial P}{\partial X}(w_0,z)\frac{\partial P}{\partial Y}(w_0,z) \ne 0$ for all $(w_0,z) \in V_1$; and \label{P-partial-prod-non-zero}
    \item[(d)] $\frac{\partial Q}{\partial X}(w',z)\frac{\partial Q}{\partial Y}(w',z) \ne 0$ for all $w' \in \imp(w_0)$ and all $(w',z) \in V_2$. \label{Q-partiao-prod-non-zero}
\end{itemize}  
\end{claim}
\begin{proof}
Since $\frac{\partial P}{\partial X}(w,z)\frac{\partial P}{\partial Y}(w,z)$ is not identically zero, it vanishes at finitely many points of $V_1$. Similarly, $\frac{\partial Q}{\partial X}(w,z)\frac{\partial Q}{\partial Y}(w,z)$ is not identically zero for $(w,z) \in V_2$, so it vanishes at finitely many points of $V_2$. Also, note that if $J_\sim$ is an interval, then all but two points are interval points, if $J_\sim$ is a Jordan curve, then all points are Jordan points. If none of these cases occurs, then by Lemma \ref{lem:model-types}, $J_\sim$ has uncoutably many distinct endpoints and there are only countably many poly-accessible points in $J_\sim$ by \cite[Proposition 2.18]{Pommerenke}. So, there are only countably many points $w_0$ that do not satisfy at least one of the properties (a)-(d). Therefore, regardless of the topological type of $J_\sim$, conditions (a)-(d) can be satisfied simultaneously. 
\end{proof}

Let $\mathfrak{M}$ be the finite set of points $x$ such that 
\begin{equation}
\label{eqn:partials}
\frac{\partial P}{\partial X}(x,y)\frac{\partial P}{\partial Y}(x,y) = 0,
\end{equation}
for some $(x,y) \in V_1$. Then, for any path $\gamma:[0,1] \lra \C \setminus \mathfrak{M}$ with $\gamma(0) \in U$, we can use the algebraic equation $P$ to analytically continue $h$ along $\gamma$. 

We have $w_0 \notin \mathfrak{M}$ by our assumption. So, we can choose a simple path $\gamma:[0,1] \lra \C\setminus \mathfrak{M}$ from $z_0$ to $w_0$ with $\gamma((0,1)) \subset \D \setminus \mathfrak{M}$ and continue $h$ along $\gamma$. By \eqref{eqn:partials}, we must have $h'(z) \ne 0$ for all $z \in \gamma$. 

\begin{claim}
\label{lem:extend-psi}
Suppose that $h(\gamma(s)) \in \D$ for all $s \in I$. Then, $J_f$ must be locally connected and $f$ must be an exceptional polynomial.
\end{claim}
\begin{proof}
Let $w_1 := h(w_0)$ for the analytic continuation of $h$ along $\gamma$. Suppose that $h$ is defined on a small ball $U_0$ around $w_0$. Since $h(w_0) \in \D$ we can also assume that $h(U_0) \subset \D$ after making $U_0$ smaller if necessary. We first show that $J_f$ is locally connected. Note that we have 
\begin{equation}
\label{eqn:rel-with-h}
P(\Psi(z), \Psi(h(z))) = 0,
\end{equation}
for all $z \in U_0 \cap \D$. Let $z_0$ be a point in the impression $\imp(w)$ for some $w \in U_0 \cap \partial\D$. Then, by continuity we conclude that $P(z_0, \Psi(h(w))) = 0$. Therefore, $z_0$ can only take finitely many values. Since $\imp(w)$ is connected (\cite[Lemma 2.4(1)]{kiwi-real-lamination}), it follows that the impression $\imp(w)$ is a singleton. Therefore, the impressions of all points in $I_0 = U_0 \cap \partial\D$ are singletons. As a result of \cite[Lemma 2.4(2)]{kiwi-real-lamination}, we conclude that all elements of $\sigma_d^{n}(I_0)$ also have singleton impressions for every $n \in \Z$. We conclude that all points of $\Ss^1$ have singleton impressions as 
\[
\Ss^1 = \bigcup_{n \in \Z}\sigma_d^{n}(I_0).
\]
This finishes the proof of the fact that $J_f$ is locally connected. As a result, $\Psi$ extends continuously to the boundary $\partial \D$. Moreover, by \cite[Proposition 7]{zdunik-cheb}, $J_f$ cannot be an interval since $f$ is not a Chebyshev polynomial up to sign. So, $w_0$ is either an endpoint or a Jordan point.

We now want to use the algebraic relation \eqref{eqn:rel-with-h} to show that we can analytically extend $\Psi$ to some non-empty open ball contained in $U_0$. Note that because of property (d) in Claim \ref{claim:w_0-is-nice}, we can parameterize the curve $V_2$ near the point $(\Psi(w_0), \Psi(w_1))$ as 
\[
(h_2(w), w),
\]
for all $w$ in a small neighborhood $U_1$ of $\Psi(w_1)$. After replacing $U_0$ with a smaller neighborhood if necessary we can assume that $$\Psi(h(U_0)) \subset U_1.$$ Then, the branches $(h_2(w),w)$ and $(\Psi(z), \Psi(h(z)))$ agree whenever $w = \Psi(h(z))$ and $z \in U_0 \cap \overline\D$. Let $\widetilde{\Psi}(z)$ denote the local restriction of $\Psi(z)$ to $U_0 \cap \D$. We can extend the function $\widetilde{\Psi}(z)$ to $U_0$ by letting
\[
\widetilde{\Psi}(z) = h_2(\Psi(h(z))). 
\]
After replacing $U_0$ with a smaller open subset we assume that $\widetilde{\Psi}$ sends $U_0$ biholomorphically onto its image. 

Let 
\[
N_0 \supset N_1 \supset N_2 \supset \cdots,
\]
be a neighborhood base of open balls at $w_0$ with $N_0 = U_0$. Our goal now is to show that \begin{equation}
\label{eqn:local-extension-sends-boundary-to-Jf}
    \widetilde{\Psi}(N_j \cap \partial \D) = \widetilde{\Psi}(N_j) \cap J_f,
\end{equation} for some $j \ge 0$. By continuity we clearly must have $\widetilde{\Psi}(N_0 \cap \partial \D) \subset \widetilde{\Psi}(N_0) \cap J_f$. If the reverse inclusion does not hold for any $j \ge 0$, then there must exist $x_j \in N_j \setminus \partial \D$ such that $y_j := \widetilde{\Psi}(x_j) \in J_f$. Using the fact that $J_f$ is locally connected and $\Psi$ extends continuously to the boundary of $\D$, there is some $\lambda_j \in \partial \D$ such that $y_j = \Psi(\lambda_j)$. In fact, $\lambda_j \in \partial\D \setminus (\partial \D \cap U_0)$. To see this note that if $\lambda_j \in \partial \D \cap U_0$, then we get 
\begin{equation}
\label{eqn:psi-psi-tilde}
y_j = \widetilde{\Psi}(x_j) = \Psi(\lambda_j) = \widetilde{\Psi}(\lambda_j),
\end{equation}
which contradicts the fact that $\widetilde{\Psi}$ is a biholomorphism on $U_0$ since $x_j$ by definition lies outside of $\partial\D$ and thus is not equal to $\lambda_j$. 

Take a subsequence $\{\lambda_{n_i}\}$ of $\{\lambda_i\}$ converging to some $\lambda \in \partial\D \setminus (\partial \D \cap U_0)$. Equation \eqref{eqn:psi-psi-tilde} shows that
\[
\Psi(\lambda) = \lim_{i \to \infty} \Psi(\lambda_{n_i}) = \lim_{i \to\infty}\widetilde{\Psi}(x_{n_i}) = \widetilde{\Psi}(w_0) = \Psi(w_0).
\]
Since $\lambda \ne w_0$, this contradicts the fact that $w_0$ is either an endpoint or a Jordan point and finishes the proof of the reverse inclusion. 

Finally, we get a contradiction since we have shown that equation \eqref{eqn:local-extension-sends-boundary-to-Jf} holds for some $j\ge 0$ and this proves that an open subset of the Julia set is smooth which contradicts Fatou's theorem \cite[pp. 250]{Fatou} (see also \cite{milnor}).  
\end{proof}
Now, if $h(\gamma) \subset \D$ we are done by Lemma \ref{lem:extend-psi}. Suppose this is not the case and let $0 \le t \le 1$ be the smallest $t$ such that $h(\gamma(t)) \in \partial\D$. If $t < 1$, we can switch the first and second coordinates and conclude the proof by Claim \ref{lem:extend-psi} again. So, we assume that $t = 1$ which means that $h(\gamma(t)) \in \D$ for all $t < 1$.

Now, take a small ball $U_0$ around $w_0$ where $h$ is defined and such that $U_0$ is divided by $\partial \D$ into two connected open pieces $U_0^0 = U_0 \cap \D$ and $U_0^1 = U_0 \cap \overline{\D}^c$. We may assume $h(U_0^0) \subset \D$. This is because otherwise we can find a point $w_0' \in U_0^0$ with $h(w_0') \in \partial \D$. Choose a path $\gamma'$ to $w_0'$ obtained by branching from $\gamma$ inside $U_0^0$ and assume without loss of generality that $w_0'$ is the first point of $\gamma'$ whose image under $h$ lies in $\partial \D$. After perturbing $w_0'$ if necessary, we may assume that $h(w_0')$ satisfies properties (a)-(d) of Claim \ref{claim:w_0-is-nice}. So, after switching the first and second coordinates we can conclude the proof by Claim \ref{lem:extend-psi}.

So, we may assume that $h(U_0^0) \subset \D$. Similarly, we can assume $h(U_0 \cap \partial \D) \subset \partial \D$ as otherwise we can conclude the proof using Claim \ref{lem:extend-psi}. As a result of the Schwartz reflection principle we also conclude $h(U_0^1) \subset \overline{\D}^c$. Therefore, we actually have
\begin{equation}
\label{eqn:h-sends-b-to-b}
h(U_0 \cap \partial\D) = h(U_0) \cap \partial\D,
\end{equation}
and 
\[
h(U_0^0) = h(U_0) \cap \D, \text{ and } h(U^{1}_0) = h(U_0) \cap \overline\D^c.
\]
We may also replace $U_0$ with a smaller open subset of assume without loss of generality that $h|_{U_0}$ is a biholomorphism onto its image.
\begin{claim}
\label{claim:image-also-endpoint}
We may assume that $\iota(w_0)$ and $\iota(h(w_0))$ are the same type of $J_\sim$-point. 
\end{claim}
\begin{proof}
If $\iota(w_0)$ is a Jordan point, this is clear since all points are Jordan points. If $\iota(w_0)$ is an interval point, then all but two points of $J_\sim$ are interval points. Hence, there must exist an interval point $w'_0$ near $w_0$ such that its image $h(w'_0)$ is also an interval point. We can replace $w_0$ with $w'_0$ and assume that $w_0$ has the desired property of the claim.

Now suppose $y_0 := \iota(w_0)$ is an endpoint $J_\sim$ and $\Ss^1$ has infinitely many endpoints. Since $w_0$ is an endpoint, we must have $\phat^{-1}(y_0) = \imp(w_0)$. By Proposition \ref{prop:D_i-is-base}, the sets $\imp(w_0) \subset D_{j,y_0}$ constructed in \ref{cons:base-at-fibers} form a simply connected neighborhood base at $\phat^{-1}(y_0) = \imp(w_0)$. Choose $j \ge 1$ large enough so that 
\[
D_{j, y_0} \cap B_\infty(f) \subset \Psi(U_0 \cap \D).
\]
This is acheivable since $w_0$ is a $J_\sim$-endpoint. Moreover, using property \((d)\) of Claim \ref{claim:w_0-is-nice}, we can choose $j \ge 1$ large enough so that 
\[\frac{\partial Q}{\partial X}(w',z)\frac{\partial Q}{\partial Y}(w',z) \ne 0,
\]
for all $w' \in D_{j,y_0}$ and all $(w',z) \in V_2$. As a result, the branch $(w, \tilde h(w))$ given by 
\[
\tilde h(w) = \Psi(h(\Phi(w))),
\]
for $w \in D_{j,y_0} \cap B_\infty(f)$ has a unique extension to $D_{j,y_0}$ since $D_{j,y_0} \cap B_\infty(f)$ is connected by Proposition \ref{prop:D_i-cap-b_infty-connected} and $D_{j,y_0}$ is simply connected by Proposition \ref{prop:D_i-is-base}. 

By making $U_0$ smaller and enlarging $j$ if necessary, we assume without loss of generality that $U_0$ is equal to one the open sets $U_i$ constructed in Lemma ~\ref{lem:h-respects-lamination}. Thus by the same lemma we have:
\begin{itemize}
\item for any $\sim$-class $[x]\subset J_f$ with $x\in D_{j,y_0}$, one has $\tilde h([x])\subset [\tilde h(x)]$; and
\item for any $\theta, \theta'\in U_0\cap \partial\D$ with $\imp(\theta), \imp(\theta') \subset D_{j,y_0}$ we have
\[
\theta \sim_f \theta' \ \Longrightarrow \ h(\theta) \sim_f h(\theta').
\]
\end{itemize}
Since $w_0$ is an endpoint, there exists a sequence of non-degenerate leaves $\overline{t_jt'_j}$ with $t_j,t'_j\in U_0$ converging to $w_0$ such that $w_0$ lies on the shorter arc of $\Ss^1$ joining $t_j$ to $t'_j$ for every $j\ge 1$. Choose $j_0$ sufficiently large so that $\imp(t_{j_0})$ and $\imp(t'_{j_0})$ are contained in $D_{j, y_0}$. It follows that for all $\theta$ on the arc $\overline{t_{j_0}t'_{j_0}}$ we also have 
\begin{equation}
\label{eqn:all-imps-in-D}
\imp(\theta) \subset D_{i,y_0}.
\end{equation}
Moreover, by Lemma \ref{lem:model-types} the image of the arc $\overline{t_{j_0}t'_{j_0}}$ under $h$ contains uncountably many $J_\sim$-endpoints. Any such endpoint must be an image of some point of $\overline{t_{j_0}t'_{j_0}}$. Using the fact that $h|_{U_0}$ is biholomorphism onto its image along with equation \eqref{eqn:all-imps-in-D} and the two inclusion properties above, every $J_\sim$-endpoint on $\overline{h(t_{j_0})h(t'_{j_0})}$ can only arise as the image under $h$ of a $J_\sim$-endpoint in $\overline{t_{j_0}t'_{j_0}}$. Therefore, after replacing $w_0$ by a different $J_\sim$-endpoint in $\overline{t_{j_0}t'_{j_0}}$ if necessary, we may assume that $h(w_0)$ is also a $J_\sim$-endpoint.
\end{proof}

% \newpage
% \begin{proof}
% Let $E$ be the set of $J_f$-endpoints in $S^1$. Then, the Hausdorff dimension of $S^1 \setminus E$ is less than 1 by \cite[Lemma 4.3]{Meerkamp}. It follows that $\dim_{\text{H}}((U_0 \cap \partial\D)\setminus E) < 1$. If all but finitely many $J_f$-endpoints in $U_0 \cap \partial \D$ are mapped by $h$ to some point in $h(U_0 \cap \partial\D)\setminus E$ we conclude using the fact that $h$ is a local biholomorphism satisfying equation \eqref{eqn:h-sends-b-to-b} that 
% \[
% 1 = \dim_{\text{H}}(U_0 \cap \partial \D \cap E) = \dim_{\text{H}}(h(U_0 \cap \partial \D \cap E)) \le \dim_{\text{H}}(h(U_0 \cap \partial\D)\setminus E) < 1,
% \]
% which is absurd. Therefore, there are infinitely many $w \in U_0 \cap \partial \D \cap E$ such that $h(w) \in E$. Any such $w$ automatically satisfies property (a). Since there are infinitely many such points, properties (b) and (c) can also be ensured by some $w \in U_0 \cap \partial \D \cap E$. We can then replace $w_0$ with this $w$ if necessary and conclude the proof.
% \end{proof}

Define $H(w) := \Psi(h(\Phi(w)))$ for all $w \in \Psi(U_0 \cap \D)$. Recall that by the definition of $V_2$ we must have 
\[
Q(\Psi(z), \Psi(h(z))) = 0,
\]
for all $z \in U_0 \cap \D$. Substituting $z = \Phi(w)$, we get 
\begin{equation}
\label{eqn:Q-H-rel}
Q(w, \Psi(h(\Phi(w)))) = Q(w, H(w)) = 0,
\end{equation}
for all $w \in \Psi(U_0 \cap \D)$. 

Let $y_0 := \iota(w_0)$. Shrink $U_0$ if necessary to ensure $\Psi(U_0 \cap \D) \subset D_{i,y_0}$ for all $i \ge 1$. By property (d) of $w_0$ demonstrated in Claim \ref{claim:w_0-is-nice}, there exists a parametrization $(w,\widetilde{H}(w))$ of $V_2$ defined for $w \in D_{i, y_0}$ for some $i \ge 1$ and agreeing with $(w, H(w))$ for all $w \in \Psi(U_0 \cap \D)$. This allows us to extend the function $H(w)$ to an analytic function on $D_{i,y_0}$ letting 
\begin{equation}
\label{eqn:continuation-of-H}
H(w) = \widetilde{H}(w),
\end{equation}
for all $w \in D_{i,y_0}$. 

Our goal is to show that a restricted version of $H$ gives a local symmetry of the Julia set in the sense of Levin \cite{Levin-symm} i.e. 
\begin{equation}
\label{eqn:H-is_sym}
H^{-1}(H(D_{i,y_0}) \cap J_f) = D_{i,y_0} \cap J_f,
\end{equation}
for all $i \ge i_0$ and some $i_0 \ge 1$. We break the proof into two cases: 1) $y_0$ is either an endpoint or a Jordan point, or 2) $y_0$ is an interval point.

\textbf{Case 1.} $y_0$ is an endpoint or a Jordan point. Choose $i_0$ sufficiently large so that $$D_{i,y_0} \cap B_\infty(f) \subset \Psi(U_0 \cap \D)$$ for all $i \ge i_0$. So, 
\[
H(w) = \Psi(h(\Phi(w))),
\]
for all $w \in D_{i,y_0} \cap B_\infty$. It follows from continuity along with equation \eqref{eqn:h-sends-b-to-b} that 
\[
H(D_{i,y_0} \cap J_f) \subset J_f,
\]
for all $i \ge i_0$. So, 
\[
H^{-1}(H(D_{i,y_0}) \cap J_f) \supset D_{i,y_0} \cap J_f,
\]
for all $i \ge i_0$ where we regard $H$ as a map on $D_{i,y_0}$ and take the inverse image with respect to this restriction. This proves the backward inclusion required for showing $H(D_{i,y_0} \cap J_f) = H(D_{i,y_0}) \cap J_f$. To prove the reverse inclusion, we apply the same argument to the inverse of $h$ defined on $h(U_0)$ as follows. 

Let $U_0' := h(U_0)$ and recall that $h$ has a local inverse defined on $U_0'$ which we will denote by $h^{-1}$. Choose $j_0$ sufficiently large so that $D_{j,y_1} \subset H(D_{i_0,y_0})$ for all $j \ge j_0$. This is possible since $H(D_{i_0,y_0})$ is a neighborhood of $\imp(h(w_0))$ by Lemma \ref{lem:h-respects-lamination} and since $D_{j,y_1}$ is a neighborhood base at $\imp(h(w_0)) = \phat^{-1}(y_1)$. Since $y_1 := \iota(h(w_0))$ is also an endpoint or Jordan point by Claim \ref{claim:image-also-endpoint}, we can run the same argument with $(U_0', h^{-1})$ in place of $(U_0, h)$ to get an analytic function $H_{-1}$ defined on $D_{j,y_1}$ satisfying 
\begin{equation}
\label{eqn:h{-1}-def}
H_{-1}(w) = \Psi(h^{-1}(\Phi(w))),
\end{equation}
for all $w \in D_{j,y_1} \cap B_\infty$ and also 
\begin{equation}
\label{eqn:H-1-sends-J-to-J}
H_{-1}(D_{j,y_1} \cap J_f) \subset J_f,
\end{equation}
for all $j \ge j_0$ and some $j_0 \ge 1$. Choosing $i_0$ to be larger if necessary, we can assume that $H(D_{i_0,y_0}) \subset D_{j_0,y_1}$. Thus, using equation \eqref{eqn:h{-1}-def} we see that 
\begin{equation}
\label{eqn:H-1-is-inverse}
H_{-1}(H(w)) = w,
\end{equation}
for all $w \in D_{i_0, y_0}$.
Now suppose that $H(w) \in J_f$ for some $w \in D_{i_0,y_0}$. Then, by equations \eqref{eqn:H-1-sends-J-to-J} and \eqref{eqn:H-1-is-inverse} we must have $w = H_{-1}(H(w)) \in J_f$. This shows that 
\[
H^{-1}(H(D_{i,y_0}) \cap J_f) \subset D_{i,y_0} \cap J_f,
\]
as desired.

\textbf{Case 2.} $y_0$ is an interval point. Choose $i_0$ sufficiently large so that $D_{i,y_0}^s$ is completely contained in $\Psi(U_0) \cap\D$ for all $i \ge i_0$ and some $s \in \{1,2\}$. Assume without loss of generality that $s = 1$. So, 
\begin{equation}
\label{eqn:def-of-H}
H(w) = \Psi(h(\Phi(w))),
\end{equation}
for all $w \in D_{i,y_0}^1$. By continuity along with equation \eqref{eqn:h-sends-b-to-b} we see that $H(x) \in J_f$ whenever $x \in \partial D_{i,y_0}^1\cap J_f$. But, Proposition \ref{prop:D_i-cap-B_infty-2-conn-comps} implies that 
\[
\partial D_{i,y_0}^1\cap J_f = \partial D_{i,y_0}^2 \cap J_f \supset  D_{i,y_0}\cap J_f.
\]
Thus, 
\[
H(D_{i,y_0} \cap J_f) \subset J_f,
\]
which proves the inclusion
\[
H^{-1}(H(D_{i,y_0}) \cap J_f) \supset D_{i,y_0} \cap J_f,
\]
for all $i \ge i_0$. To prove the reverse inclusion we proceed by analyzing $h^{-1}$, similar to the proof of Case 1. 

Let $U_0' := h(U_0)$ and recall that $h$ has a local inverse defined on $U_0'$ which we will denote by $h^{-1}$. Setting $y_1 := \iota(h(w_0))$ and running the same argument with $(U_0', h^{-1})$ in place of $(U_0, h)$, we get an analytic function $H_{-1}$ defined on $D_{j_0,y_1}$ satisfying 
\begin{equation}
\label{eqn:h{-1}-def-interval-case}
H_{-1}(w) = \Psi(h^{-1}(\Phi(w))),
\end{equation}
for all $w \in D^1_{j_0,y_1} \cap B_\infty$ and also 
\begin{equation}
\label{eqn:H-1-sends-J-to-J-interval-case}
H_{-1}(D_{j_0,y_1} \cap J_f) \subset J_f,
\end{equation}
for some $j_0 \ge 1$. 

Note that the open sets $H(D_{i,y_0})$ define a neighborhood base at $H(\phat^{-1}(y_0))$. Also, any point in $\phat^{-1}(y_0)$ can be approximated by a sequence of points in $D_{i,y_0}^1$ by Proposition \ref{prop:D_i-cap-B_infty-2-conn-comps} which shows using \eqref{eqn:def-of-H} that $H(\phat^{-1}(y_0)) \subset \phat^{-1}(y_1)$. So, we can choose $i_0$ to be larger if necessary and assume $H(D_{i_0,y_0}) \subset D_{j_0,y_1}$. 

Thus, using equation \eqref{eqn:h{-1}-def-interval-case} we see that 
\begin{equation}
\label{eqn:H-1-is-inverse-interval-case}
H_{-1}(H(w)) = w,
\end{equation}
for all $w \in D_{i_0, y_0}$ and the desired reverse inclusion follows as in the proof of Case 1.

Therefore, equation \eqref{eqn:H-is_sym} holds for some $i \ge i_0$ and $H$ is a local symmetry in the sense of Levin. Proposition \ref{prop:alg-symmetries} then shows that $(x, H(x))$ must define an $(f,f)$-preperiodic curve which finishes the proof.  
\end{proof}

With Proposition~\ref{prop:bialg-curve-case} in hand, we now turn to the case where $n \ge 3$ and $V_1$ is a hypersurface in $\mathbb{C}^n$.
\begin{proposition}
\label{prop:bialg-hypersurface-case}
Theorem \ref{thm:classification-of-bialgebraic} holds when  $V_1$ is a hypersurface. 
\end{proposition}
\begin{proof}
We will prove this by induction on $n$. The base case is $n = 2$ which is true by Proposition \ref{prop:bialg-curve-case}. Now assume that the theorem holds for $n - 1 \ge 1$ and let $V_1$ be a hypersurface in $\C^n$ with $V_2$ being the algebraic closure of $\bpsi_{f,n}(V'_1)$ where $V_1'$ is a branch of $V_1$ contained in $\D_R^n$. We break the proof into two cases: 
\begin{enumerate}
    \item For every subset $S \subsetneq \{1,\dots,n\}$, the projection $\pi_S: V_1 \lra (\bP^1)^{|S|}$ is dominant; or
    \item There is a subset $S \subsetneq \{1,\dots, n\}$ where the projection $\pi_S: V_1 \lra (\bP^1)^{|S|}$ is not dominant.
\end{enumerate}
\textbf{Case 1.} The projection of $V_1'$ onto the first $n - 1$ coordinates is dominant in this case. This allows us to find a local parametrization
\[
(x_1,\dots,x_{n-1}, u(x_1,\dots,x_{n-1})),
\]
contained in $V_1'$ where $u$ is a holomorphic function on an open polydisk of the form
\[
\{(x_1,\dots,x_{n-1}): |x_i - a_i| < \epsilon\},
\]
for some $(a_1,\dots,a_{n-1}) \in \D_R^{n-1}$ and some $\epsilon > 0$. 

Let $a \in B_\epsilon(a_1)$ and consider the slice of $V_1'$ given by
\begin{equation}
\label{eqn:form-of-v1}
V_1^{a} := \{(x_2,\dots,x_{n-1}, u(a,x_2,\dots, x_{n-1})): |x_i - a_i| < \epsilon \text{ for all } i \ge 2\}.
\end{equation}
This set is clearly $(n-2)$-dimensional and its image under $\Psi_{f,n-1}$ is contained in $$\pi_{2,\dots, n}(V_2 \cap \pi_1^{-1}(\Psi_f(a)))$$
which is also an $(n-2)$-dimensional subvariety of $\C^{n - 1}$. Hence, for all $a \in B_\epsilon(a_1)$, the branch $V_1^a$ is $f$-bialgebraic. Using the inductive hypothesis, this means that for every $a \in B_\epsilon(a_1)$
$$
\overline{\Psi_{f, n -1}(V_1^a)}^\zar
$$
is $f$-special. 

Since the projection of $V_1$ onto any subset of coordinates is dominant, we may replace $B_\epsilon(a_1)$ by a smaller subset if necessary to assume that, for any fixed $a \in B_\epsilon(a_1)$, the function
\[
u(a,x_2,\dots,x_n),
\]
is non-constant. Indeed, otherwise, the image of $V_1'$ under $\pi_{1,n}$ is contained in an algebraic curve which contradicts our assumption. 
This ensures by using equation \eqref{eqn:form-of-v1}, that $\overline{(V_1^a)}^\zar$ cannot be contained in a fiber of the projections $\pi_1,\dots, \pi_{n-1}$. Hence, $\overline{(V_1^a)}^\zar$ must be an $F_{n-1}$-preperiodic hypersurface. By \cite[Theorem 6.24]{Scanlon}, there are countably many such hypersurfaces. 
It follows that there is a fixed $F_{n-1}$-preperiodic hypersurface $W \subset \C^{n-1}$ such that $$
\Psi_{f, n - 1}(V_1^a) \subset W
$$ 
for all $a \in B_\epsilon(a_1)$. This shows that
$$
\Psi_{f,n}(V_1') \subset  \pi_{2,\dots,n}^{-1}(W).
$$  
But, this means that the projection of $V_1'$ to the last $n-1$ coordinates cannot be dominant which contradicts our assumption. So, this case cannot occur. 

% Now, note that since $W$ is $F_{n-1}$-preperiodic, $\pi_{2,\dots,n}^{-1}(W)$ is an irreducible $F_n$-preperiodic hypersurface of $\C^n$ which finishes the proof. 

\textbf{Case 2.} Assume without loss of generality that the projection $\pi_{1,\dots,r}: V_1 \lra \C^r$ is not dominant for some $r \ge 1$. Let $W_1 := \pi_{1,\dots,r}(V_1)$ which must be an irreducible hypersurface of $\C^r$ since $V_1$ is a hypersurface. Also, note that $V_1 = \pi_{1,\dots,r}^{-1}(W_1)$. We claim that $W_1' := \pi_{1,\dots,r}(V_1')$ is an $f$-bialgebraic branch of $W_1$. Let $$(u_1(x_1,\dots, x_{r-1}), \dots, u_{r}(x_1,\dots,x_{r-1}))$$ be a local chart of $W_1$ contained in the branch $W_1'$. Then, $V'_1$ admits a local parametrization
\[
(u_1(x_1,\dots, x_{r-1}), \dots, u_r(x_1,\dots, x_{r-1}), x_{r},\dots,x_{n - 1}).
\]
So, $V_2$ contains the open complex manifold $\mathcal{M}$ defined by the chart
\[
(\Psi(u_1),\dots, \Psi(u_r), \Psi(x_{r}),\dots, \Psi(x_{n-1})). 
\]
The projection $\pi_{1,\dots,r}$ restricted to this submanifold of $V_2$ clearly maps to an analytic set of dimension $r - 1$. We conclude that $\pi_{1,\dots,r}$ restricted to $V_2$ cannot be dominant. Hence, $\pi_{1,\dots,r}(V_2)$ is also a hypersurface and the claim follows from $$\bpsi_{f,r}(W_1') \subset \pi_{1,\dots, r}(V_2).$$ So, $\pi_{1,\dots,r}(V_2)$ is $f$-special by Proposition \ref{prop:bialg-curve-case}. Thus, $V_2 = \pi_{1,\dots,r}^{-1}(\pi_{1,2}(V_2))$ is also $f$-special, which finishes the proof.    
\end{proof}
We finally prove Theorem \ref{thm:classification-of-bialgebraic} when $V_1$ is an arbitrary subvariety of $\C^n$.
\begin{proof}[Proof of Theorem \ref{thm:classification-of-bialgebraic}]
% First suppose that the Julia set is disconnected. Let $V_1$ be a proper irreducible bialgebraic subvariety of $(\D_R)^n$, and let $V_2$ be the Zariski closure of $\bpsi_n(V_1)$. Since $V_1$ is bialgebraic we must have $\dim(V_2) = \dim(V_1) < n$. Then, $V_1 \times V_2$ must contain the graph $\Gamma(\bpsi_n|_{V_1})$. Since 
% \[
% \dim(\Gamma(\bpsi_n|_{V_1})^{\zar}) \le \dim(V_1 \times V_2) = 2k \le n + k,
% \]
% we conclude by Theorem \ref{thm:DAS-totally-disconnected} that $\Psi_n(V_1)$, and consequently $V_2$, are $F_n$-preperiodic. 

% We assume from this point on that $J_f$ is connected and locally connected. 

Let $V_1$ be a proper irreducible subvariety of $\C^n$ of dimension $k$. Suppose that a branch $V_1' \subset \D_R^n$ of $V_1$ is $f$-bialgebraic and let $V_2$ be the Zariski closure of $\bpsi_{f,n}(V_1)$. After permuting the coordinates if necessary, we may assume that $V_1$ is mapped dominantly onto $\C^k$ by the projection $\pi_{1,\dots,k}$. 

Since $V'_1$ is bialgebraic it follows that $\pi_{1,\dots,k,\ell}(V'_1)$ is a bialgebraic hypersurface in $\C^{k + 1}$ for every $k + 1 \le \ell \le n$. By Proposition \ref{prop:bialg-hypersurface-case} we conclude that $\pi_{1,\dots,k,\ell}(V_1)$ is special for all $k + 1 \le \ell \le n$. Let $S \subset \{k + 1,\dots, n\}$ be the subset of all indices $\ell$ such that $\pi_{1,\dots,k,\ell}(V_1)$ is contained in a fiber of $\pi_{k+1}: (\bP^1)^{k + 1} \lra \bP^1$. Then, for every $\ell \in \{k + 1,\dots,n\} \setminus S$, $\pi_{1,\dots,k,\ell}(V_2)$ is $F_{k + 1}$-preperiodic. We can choose $0 \le N < M$ such that 
\[
F_{k+1}^N(\pi_{1,\dots,k,\ell}(V_2)) = F_{k+1}^M(\pi_{1,\dots,k,\ell}(V_2)),
\]
for all $\ell \in \{k+1,\dots,n\} \setminus S$. It follows that $V_2$ is contained in a fiber of $\pi_S$ and that $\pi_{S^c}(V_2)$ satisfies
\[
F_{n-s}^N(\pi_{S^c}(V_2)) = F_{n-s}^M(\pi_{S^c}(V_2)),
\]
where $s = |S|$. Thus, $V_2$ is an $f$-special subvariety of $\C^n$. 
\end{proof}

\section{Proof of Corollary \ref{cor:shared-basin}}
\label{sec:pf-of-shared-comp}
This final section is devoted to the proof of Corollary~\ref{cor:shared-basin}. Recall that Corollary~\ref{cor:shared-basin} classifies all rational functions that share a periodic basin with a non-exceptional polynomial whose Julia set is either disconnected or admits a non-degenerate locally connected model.
\begin{proof}[Proof of Corollary \ref{cor:shared-basin}]
Suppose that $r$ is a rational function for which $B_\infty(f)$ is a Fatou component and such that
\[
r^b(B_\infty(f)) = B_\infty(f).
\]
for some $b \ge 1$. After replacing $r$ with $r^b$ we may assume that
\[
r(B_\infty(f)) = B_\infty(f).
\]
Then
\[
\tilde r := \Phi_f \circ r \circ \Psi_f
\]
defines a finite proper self-map of the disk $\D$ that sends $\partial \D$ to itself. Hence $\tilde r$ is a finite Blaschke product. It follows that the algebraic curve
\[
\{(x,y)\in \D^2 : y=\tilde r(x)\}
\]
is mapped by $(\Psi_f,\Psi_f)$ onto the curve
\[
\{(x,y)\in \C^2 : y=r(x)\},
\]
and is therefore $f$-bialgebraic. By Theorem~\ref{thm:classification-of-bialgebraic}, the curve
\[
\{(x,y)\in \C^2 : y=r(x)\}
\]
is $(f,f)$-preperiodic. Thus, for some $n \ge 0$, the curve parametrized by
\[
\bigl(f^n(z),\, f^n(r(z))\bigr)
\]
is periodic. By the classification theorem of Medvedev and Scanlon \cite[Theorem~6.24]{Scanlon}, it follows that
\[
f^n \circ r = g \circ f^n
\]
for some polynomial $g$ commuting with an iterate of $f$. In particular, this implies that
\[
B_\infty(g)=B_\infty(f).
\]
Moreover, the relation
\[
f^n \circ r = g \circ f^n
\]
shows that $r$ must in fact be a polynomial, and that
\[
B_\infty(r)=B_\infty(g)=B_\infty(f).
\]
Consequently,
\[
J_r=J_g=J_f.
\]
The desired conclusion now follows from \cite[Theorem~1]{atela} (see also \cite{schmidt-steinmetz}).
\end{proof}
\bibliography{bibfile}
\bibliographystyle{alpha}
\end{document}